\documentclass[letterpaper,10pt]{article}

\usepackage[utf8]{inputenc}
\usepackage[dvipsnames]{xcolor}
\usepackage[utf8]{inputenc}
\usepackage[
backend=biber,
style=alphabetic,
sorting=nyt,
maxbibnames=99]{biblatex}
\DeclareFieldFormat*{title}{#1}
\usepackage{amssymb}
\usepackage{amsmath} 
\usepackage{appendix}
\usepackage{accents}
\usepackage[colorinlistoftodos]{todonotes}
\usepackage{amsthm}
\usepackage[margin=1in]{geometry}
\usepackage{fancyhdr}
\usepackage{upgreek}
\usepackage{enumitem}
\numberwithin{equation}{section}

\newcommand{\E}{\mathbb{E}}

\newcommand{\N}{\mathbb{N}}

\renewcommand{\P}{\mathbb {P}}

\newcommand{\R}{\mathbb{R}}

\newcommand{\Ecal}   {{\mathcal E }}

\newcommand{\Fcal}   {{\mathcal F }} 
\newcommand{\Gcal}   {{\mathcal G }} 
\newcommand{\Hcal}   {{\mathcal H }}

\newcommand{\Mcal}   {{\mathcal M }}

\newcommand{\Rcal}   {{\mathcal R }}

\newcommand{\Zcal}   {{\mathcal Z }}

\usepackage[colorlinks]{hyperref}
\title{The heat equation with time-correlated random potential in $d=2$: Edwards-Wilkinson fluctuations}
\author{Sotirios Kotitsas\footnote{Email:\href{mailto:sotirios.kotitsas@dm.unipi.it}{sotirios.kotitsas@dm.unipi.it},  University of Pisa, Italy}}
\date{}

\begin{document}

\newtheorem{mythm}{Theorem}[section]
\newtheorem{myprop}{Proposition}[section]
\newtheorem{mydef}{Definition}[section]
\newtheorem{mycor}{Corollary}[section]
\newtheorem{mylem}{Lemma}[section]
\newtheorem*{myackn}{Acknowledgments}
\newtheorem{refproof*}{Proof}
\newtheorem{remark}{Remark}[section]
\setlength{\headheight}{13.59999pt}
\addtolength{\topmargin}{-1.59999pt}
\maketitle

\begin{abstract}
   We consider the stochastic PDE:
   $$\partial_tu(t,x)=\frac{1}{2}\Delta u(t,x)+{\beta}{}u(t,x)V(t,x),$$
   in dimension $d=2$, where the potential $V$ is the space and time mollification of the two-dimensional space-time white noise. We show that, after renormalizing, the fluctuations of the solution converge to the Edwards-Wilkinson limit with an explicit effective variance and constant effective diffusivity. Our main tool is a Markov chain on the space of paths, which we use to establish an extension of the Kallianpur-Robbins law \cite{Kal-Rob} to a specific regenerative process.\vskip0.3cm

\vspace{0.3cm}
\end{abstract}

\tableofcontents
\newpage

\section{Introduction}\label{sec:Intro}
We are interested in the following  stochastic heat equation in dimension $d=2$:

\begin{equation}\label{eq:1.1}
    \partial_tu(t,x)=\frac{1}{2}\Delta u(t,x)+{\beta}u(t,x)V(t,x),\quad u(0,x)=1,
\end{equation}
where
\begin{equation}\label{noise}
    V(t,x):=\int_{\mathbb{R}^{1+2}}\phi(t-s)\psi(x-y)d\xi(s,y).
\end{equation}
Here  $\xi$ is the space-time white noise on $\mathbb{R}\times\mathbb{R}^2$ over a probability space $(\boldsymbol{\Omega},\Fcal,\mathbf{P})$, and  $\phi:\mathbb{R}\rightarrow\mathbb{R},\textbf{ }\psi:\mathbb{R}^2\rightarrow\mathbb{R}$ are positive smooth functions. We assume that $\phi$ is supported on $[0,1]$ and  that $\psi$ is symmetric and is supported on $|x|\leq1/2$. The coupling constant ${\beta}$ is tuning the strength of the random potential $V$.\par
The potential $V$ is  Gaussian and stationary. We denote by $R$ its correlation function:
$$R(s,y):=\mathbf{E}\biggr[V(s,y)V(0,0)\biggl]=\phi\star\Tilde{\phi}(s)\textbf{ }\psi\star\psi(y),$$
where $\star$ denotes the convolution of two functions, $\mathbf{E}$ is the expectation with respect to the white noise $\xi$ and $\Tilde{\phi}(s)=\phi(-s)$. The function $R$ is compactly supported and  $R(s,\cdot)=0$ for $|s|\geq1$.\par
Our goal is to understand the large space-time behavior of the solution to \eqref{eq:1.1} i.e. the behavior of the random variable $u_{\epsilon}(t,x)=u(\frac{t}{\epsilon^2},\frac{x}{\epsilon})$, as $\epsilon\rightarrow0$. More specifically, we aim to understand the fluctuations of $u_{\epsilon}(t,x)$ as a distribution: For $g\in C^{\infty}_c({\mathbb{R}^2})$ we aim to find the limiting law as $\epsilon\rightarrow0$ of
\begin{equation}\label{eq:local_average}
    \int_{\mathbb{R}^2}u_{\epsilon}(t,x)g(x)dx,
\end{equation}
after centering and re-scaling.\par
Similar problems have been considered before (see the next subsection for more details). In \cite{Gu_2018}  the authors considered the same problem in $d\geq3$, where they proved that the limiting law is Gaussian with mean 0 and with an explicit variance (see also \cite{Mukherjee2017CentralLT}, for an earlier, weaker result in this setting). More specifically, they proved that the limiting law is given by
$$\int_{\mathbb{R}^2}\mathcal{U}(t,x)g(x)dx,$$
where  $\mathcal{U}$ solves the additive stochastic heat equation with an effective diffusivity $a_{eff}({\beta})\in\mathbb{R}^{2\times2}$ (the space of $2\times2$  matrices with real coefficients) and effective variance $v_{eff}({\beta})>0$:
\begin{equation}\label{eq:1.2}
    \partial_t\mathcal{U}=\frac{1}{2}\nabla a_{eff}({\beta})\nabla\mathcal{U}+\beta v_{eff}({\beta})\xi\textit{,  }\mathcal{U}(0,x)=0.
\end{equation}
The additive stochastic heat equation is sometimes called the Edwards-Wilkinson equation.\par
Here, we extend the Edwards-Wilkinson limit for the equation \eqref{eq:1.1} to dimension $d=2$. The important difference is that $d=2$ is the critical dimension (see below for more details). For this reason, we also tune the coupling constant ${\beta}$ to go $0$ as $\epsilon\rightarrow0$. More specifically, we consider the following equation:
\begin{equation}\label{mSHEmod}
    \partial_tu(t,x)=\frac{1}{2}\Delta u(t,x)+\frac{\hat{\beta}}{\sqrt{\log\frac{1}{\epsilon}}}u(t,x)V(t,x)\textbf{,  }u(0,x)=1.
\end{equation}
Set:
$$u_{\epsilon}(t,x):=u(\frac{t}{\epsilon^2},\frac{x}{\epsilon}),$$ 
where $u$ solves \eqref{mSHEmod}. Our main result is the following:

\begin{mythm}\label{thm:Main}
Let $g\in C_c({\mathbb{R}^2})$ and:
$$\hat{\beta}_c(R)=\sqrt{2\pi/||R||_1},$$
where $||R||_1=\int_{\mathbb{R}^{1+2}}R(s,y)dyds$. There exists $\zeta^{(\epsilon)}:\mathbb{R}_{>0}\rightarrow\mathbb{R}_{>0}$ such that $\zeta^{(\epsilon)}_{t/\epsilon^2}\rightarrow+\infty$ as $\epsilon\rightarrow0$, and such that for all $\hat{\beta}<\hat{\beta}_c(R)$
$$\int_{\mathbb{R}^2}e^{-\zeta^{(\epsilon)}_{t/\epsilon^2}}u_{\epsilon}(t,x)g(x)dx\rightarrow\int_{\mathbb{R}^2}g(x)dx,$$
in probability as $\epsilon\rightarrow0$. Moreover,
\begin{equation}\label{CLT}
    \sqrt{\log\frac{1}{\epsilon}}\int_{\mathbb{R}^2}e^{-\zeta^{(\epsilon)}_{t/\epsilon^2}}(u_{\epsilon}(t,x)-\mathbf{E}[u_{\epsilon}(t,x)])g(x)dx\Rightarrow\int_{\mathbb{R}^2}\mathcal{U}(t,x)g(x)dx
\end{equation}
in distribution, where  $\mathcal{U}$ solves the additive heat equation $(\ref{eq:1.2})$ with effective variance $v_{eff}(\hat{\beta})>0$ given by the formula:
\begin{equation}\label{eq:effvar}
    v_{eff}(\hat{\beta})^2=\biggl(1-\frac{\hat{\beta}^2}{\hat{\beta}_c(R)^2}\biggr)^{-1}\cdot||R||^2_1,
\end{equation}
and with trivial effective diffusivity $a_{eff}(\hat{\beta})=I_{2\times2}$, where $I_{2\times2}$ is the $2$ by $2$ identity matrix.
\end{mythm}
The choice of flat initial data is for simplicity, and our results can be generalized to more general initial conditions by using similar arguments.\par
The (first-order) asymptotics of $\zeta^{(\epsilon)}_{t/\epsilon^2}$, {as  $\epsilon\rightarrow0$}, are:
$$\zeta^{(\epsilon)}_{t/\epsilon^2}=(C_1+o(1))\frac{t}{\epsilon^2\log \frac{1}{\epsilon}}+o(1),$$
where $C_1$ is an explicit constant and $o(1)$ indicates a function of $\epsilon$ that goes to $0$ as $\epsilon\rightarrow0$ (see \textbf{Appendix \ref{sec:app2}}).

\subsection{Comments and related results}
Equation \eqref{eq:1.1} is a regularization of the multiplicative stochastic heat equation (mSHE),  formally written as
\begin{equation*}
    \partial_tu(t,x)=\frac{1}{2}\Delta u(t,x)+{\beta} u(t,x)\xi(t,x),
\end{equation*}
where $x\in\mathbb{R}^d$, and $\xi(t,x)$ is the space-time white noise i.e. a  mean zero Gaussian distribution with covariance
$$\mathbf{E}[\xi(s,y)\xi(t,x)]=\delta(s-t)\delta(x-y),$$
where $\delta$ is the Dirac delta function. Again,  ${\beta}>0$ is the coupling constant. This equation has a long history and is related to the KPZ equation (via the Cole-Hopf transform) and to directed polymers (via the Feynman-Kac formula). In dimension $1$ we can give a direct meaning to this equation,  by using the notions of the It\^o-Walsh integral and the mild solution (see for example \cite{Dalang2008AMO}). However, the noise $\xi$ becomes increasingly irregular as the dimension increases, and this theory does not work when $d\geq2$. In particular,    the product $u(t,x)\xi(t,x)$ does not make sense in $d\geq2$.\par
The mSHE is a singular SPDE, a term usually reserved for SPDEs that do not make direct sense, usually due to some non-linearity in the equation, and we have to interpret them in a specific way. This class contains many interesting equations from mathematical physics like the KPZ equation, the $\Phi^4-$model, the Allen-Cahn equation, and the stochastic Navier-Stokes equations. Making direct sense of these equations is a challenging mathematical problem.\par
Some progress in studying these types of equations has been made in recent years, most famously with the theories of regularity structures \cite{c28fc3fc-a79f-3fe7-abb0-cce0ee5be155},\cite{Hairer_2014}, paracontrolled distributions \cite{Gubinelli_2016} and renormalization group \cite{renormalization,duch2022flow}. These theories focus on sub-critical singular SPDEs, which, loosely speaking, are those SPDEs that on the small scales the cause of the problem (e.g. the non-linearity) is formally small, and the equation should behave as a nicer SPDE (see \cite{AFST_2017_6_26_4_847_0} for this semi-formal definition). The property of being sub-critical usually depends on the dimension of the underlying space and the specific form of the equation. For example, mSHE is sub-critical in $d=1$, critical for $d=2$, and supercritical for $d\geq3$. Indeed if we take $u^{\lambda}(t,x)=u(\lambda^2t,\lambda x)$, where $u$ formally solves the mSHE, by the scaling properties of the space-time white noise  $u^{\lambda}(t,x)$ is equal in distribution to the solution of
\begin{equation*}
    \partial_tu(t,x)=\frac{1}{2}\Delta u(t,x)+\lambda^{\frac{2-d}{2}}{\beta} u(t,x)\xi(t,x).
\end{equation*}
which demonstrates exactly this principle as $\lambda\rightarrow0$.\par
For $d\geq2$ the usual way to make sense of this equation is by regularizing the noise and studying the corresponding solution as we remove the regularization. For example, we can mollify the noise in space by an approximation of the identity. We define:
$$\xi_{\epsilon}(t,x)=\psi_{\epsilon}\star'\xi(t,x)$$
where $\star'$ denotes convolution in space, $\psi\in C_c^{\infty}(\mathbb{R}^d)$ and $\psi_{\epsilon}(x)=\epsilon^{-d}\psi(x/\epsilon)$. We then consider the corresponding solution $u^{\epsilon}$ with initial data equal\footnote{This is just for simplicity.} to 1. This solution now exists (again see \cite{Dalang2008AMO}) for every $\epsilon>0$. The interest now lies in  the various limits of this random variable as $\epsilon\rightarrow0$.\par
In general, as $\epsilon\rightarrow0$, we also need to tune the coupling constant to get nontrivial limits. More specifically, we need to take $\beta_{\epsilon}=\hat{\beta}\cdot\epsilon^{(d-2)/2}$. With this choice $u^{\epsilon}(t,x)$ is equal in distribution to $u(\frac{t}{\epsilon^2},\frac{x}{\epsilon})$ where $u$ solves \eqref{eq:1.1} with $\beta=\hat{\beta}$ and $V$ white in time (i.e. $\phi=\delta$).\par
In $d=1$ the solution $u^{\epsilon}$ converges in $L^2$ to the mSHE without mollification, which is well defined. In higher dimensions the situation is different. For $d\geq3$ it has been shown that there is a phase transition in $\hat{\beta}$ where, for $\hat{\beta}$ below a specific threshold $\hat{\beta}_c$ (called the weak disorder regime), the solution $u^{\epsilon}(t,x)$ has a nontrivial positive limit in distribution for every $(t,x)$. For $\hat{\beta}\geq\hat{\beta}_c$ the solution converges to zero (see \cite{Weak_Strong_Mukh_Sha_Zei}). There is even a smaller regime  $\hat{\beta}<\hat{\beta}_{L^2}$, called the $L^2$ regime, where the solution $u^{\epsilon}(t,x)$ has a finite second moment as $\epsilon\rightarrow0$. In $d=2$ we need to introduce a further modification to the equation by tuning the coupling constant as $\beta_{\epsilon}=\hat{\beta}/\sqrt{\log\frac{1}{\epsilon}}$. In that case, a similar phase transition has been identified (see \cite{a2c66a1e-c941-34a5-824d-2b7be2b94dd1}) where the solution $u^{\epsilon}(t,x)$, for fixed $(t,x)$, converges in distribution to a nontrivial limit\footnote{In particular to a log-normal random variable.} for $\hat{\beta}$ below a critical value (called the subcritical regime) and to zero otherwise. It should be noted that these critical values $\hat{\beta}_c$ depend on the mollifier that we used. For example, in $d=2$, $\hat{\beta}_c=\sqrt{2\pi}/||\psi||_1$. \par 
The solution $u^{\epsilon}$ can be studied as a random distribution, in particular after we test it against a test function. In this case, the works \cite{a2c66a1e-c941-34a5-824d-2b7be2b94dd1} for $d=2$ and \cite{Gu2019FluctuationsOA}, \cite{10.1214/21-AIHP1173}, \cite{NakajimaShuta2022Loln} 
for $d\geq3$, showed that we have a limit theorem and Gaussian fluctuations in the subcritical/$L^2$ regime respectively. Regarding the fluctuations, it was proved that we fall into the Edwards-Wilkinson universality class.
This means that
\begin{equation*}\label{largeflu}
    \frac{\hat{\beta}}{\beta_{\epsilon}}\int_{\mathbb{R}^2}(u^{\epsilon}(t,x)-\mathbf{E}[u^{\epsilon}(t,x)])g(x)dx
\end{equation*}
converges in distribution to
$$\int_{\mathbb{R}^2}\mathcal{U}(t,x)g(x)dx,$$
where $\mathcal{U}(t,x)$ solves equation \eqref{eq:1.2}. In both cases ($d=2$ or $d\geq3$), the effective diffusivity is trivial, $a_{eff}=I_{2\times2}$. For $d\geq3$ we point out that the Edwards-Wilkinson limit has been proved in the full $L^2$ regime in \cite{10.1214/21-AIHP1173}, \cite{NakajimaShuta2022Loln} while in \cite{Gu2019FluctuationsOA} it was proved for $\hat{\beta}$ small enough.\par
The regularization we study here is via mollification of the noise in both space and time. We focus on studying the large-scale behavior of the solution with the noise regularized at a fixed scale. The added difficulty to this approach is that, since we do not use the It\^o-Walsh integral, we destroy the martingale structure of the solution and create time correlations that add up exponentially. This is the reason for the exponential correction term appearing in $\textbf{Theorem \ref{thm:Main}}$ (which appears $d\geq3$ as well). This term appears because the mean of $u_{\epsilon}(t,x)$ blows up exponentially as $\epsilon\rightarrow0$ and the  term $e^{-\zeta^{(\epsilon)}_{t/\epsilon^2}}$ is introduced exactly so that the mean remains bounded as $\epsilon\rightarrow0$.\par
This regularization has been studied in other dimensions as well. In \cite{c0a84c4c-9371-3683-aa99-89bc719e4a36} the authors study the mSHE with space-time mollification has been done in $d=1$, and show that the corresponding solution converges to the solution of the mSHE without mollification in every $L^p$ space. There is also the more general result of \cite{Wonkzakai} where the authors show that solutions of a general class of  SPDEs with space-time mollification converge to the corresponding SPDEs without mollification after renormalization.\par
In $d\geq3$, in \cite{Mukherjee2017CentralLT}, the author proves a homogenization result for the equation, by showing that the local average \eqref{eq:local_average} converges in probability to the corresponding local average of a heat equation with an effective diffusivity. In  \cite{Gu_2018}, the authors extend this result and prove the Edwards-Wilkinson limit as described above. The interesting aspect of both of these works is that they get an added effective diffusivity term i.e. $a_{eff}\neq I_{d\times d}$. We also refer to \cite{dunlap2021random}, where the authors were able to study the pointwise statistics of the solution and express the effective diffusivity and effective variance in terms of objects from stochastic homogenization theory. It should be that the results in \cite{Gu_2018, dunlap2021random} were proved for $\hat{\beta}$ sufficiently small.\par
In $d=2$ there are no results, to our knowledge, that treat the mSHE with space-time mollification either at the level of the pointwise statistics or at the level of local fluctuations (i.e. after we test against a test function). In analogy with the case where we use only mollification in space, we scale the coupling constant as $\beta=\hat{\beta}/\sqrt{\log\frac{1}{\epsilon}}$. Again, we expect that there is a subcritical regime. This is suggested by \textbf{Theorem \ref{thm:Main}}, since for all $\hat{\beta}<\hat{\beta}_c(R)$ we have a nontrivial Central Limit Theorem (CLT) and the effective variance that we get blows up at $\hat{\beta}=\hat{\beta}_c(R)$. Given this observation, we can prove \textbf{Theorem \ref{thm:Main}} in the optimal range of $\hat{\beta}$.\par
We should also note that unlike in $d\geq3$, we get a "trivial" effective diffusivity for the Edwards-Wilkinson limit. The reason for this becomes apparent from the proof. Loosely speaking, the effective diffusivity in \cite{Mukherjee2017CentralLT, Gu_2018} is the second moment of the total path increment of a path with law absolutely continuous with respect to the Wiener measure. More specifically, its law is an exponentially tilted measure with the Wiener measure as the reference measure. These paths appear due to the Feynman-Kac formula and the exponential tilting is due to time mollification (see below for more details). It so happens that the exponential tilting depends on the coupling constant $\beta_{\epsilon}=\hat{\beta}/\sqrt{\log\frac{1}{\epsilon}}$. As $\epsilon\rightarrow0$ the tilting disappears making the total path increment look like the total path increment of a Brownian motion, which gives us the trivial effective diffusivity.\par
{Finally, we should compare \textbf{Theorem \ref{thm:Main}} with its analog in $d=2$ when we use space mollification. In the latter case, let $u_{\psi}$ solve the mSHE with only space mollification, using $\psi$ as the mollifier,  with coupling constant $\beta_{\epsilon}=\hat{\beta}/\sqrt{\log\frac{1}{\epsilon}}$. From \cite{a2c66a1e-c941-34a5-824d-2b7be2b94dd1}, the effective variance of the Edwards-Wilkinson limit is:
$$\mathcal{V}_{eff}^2(\hat{\beta})=\biggr(1-\frac{\hat{\beta}^2||\psi||_1^2}{2\pi}\biggl)^{-1}||\psi||_1^2.$$}
{If we additionally mollify  time using $\phi$, as in \eqref{noise}, then \textbf{Theorem \ref{thm:Main}} gives  an effective variance:
$$v_{eff}^2(\hat{\beta})=\biggl(1-\frac{\hat{\beta}^2 ||\psi||_1^2||\phi||_1^2}{2\pi}\biggr)^{-1}\cdot||\psi||_1^2||\phi||_1^2.$$
This shows that, when $\phi$ is a probability density
\begin{equation*}\label{stmolflu}
    \sqrt{\log\frac{1}{\epsilon}}\int_{\mathbb{R}^2}e^{-\zeta^{(\epsilon)}_{t/\epsilon^2}}(u_{\epsilon}(t,x)-\mathbf{E}[u_{\epsilon}(t,x)])g(x)dx
\end{equation*}
converges to the same limit as
\begin{equation*}\label{smolflu}
    \sqrt{\log\frac{1}{\epsilon}}\int_{\mathbb{R}^2}(u_{\psi}(t/\epsilon^2,x/\epsilon)-\mathbf{E}[u_{\psi}(t/\epsilon^2,x/\epsilon)])g(x)dx.
\end{equation*}
Observe that this does not happen in $d\geq3$. Indeed, in $d\geq3$ when $V$ is white in time the Edwards-Wilkinson limit has effective diffusivity  $a_{eff}=I_{d\times d}$, while when $V$ is correlated in time, $a_{eff}\neq I_{d\times d}$.}
\subsection{Future Directions}
The mSHE is related to the KPZ equation via the Cole-Hopf transform. If $u(t,x)$ solves \eqref{mSHEmod} then 
$$h(t,x):=\log u(t,x)$$
solves the equation
$$\partial_th(t,x)=\frac{1}{2}\Delta h(t,x)+|\nabla h(t,x)|^2+\frac{\hat{\beta}}{\sqrt{\log\frac{1}{\epsilon}}}V(t,x),\textbf{ }h(0,x)=0.$$
In \cite{10.1214/19-AOP1383}, it is proved that when $V$ is white in time the random variable $h(t/\epsilon^2,x/\epsilon)$ converges in distribution to a Gaussian random variable for all $\hat{\beta}$ in the subcritical regime. Moreover, the fluctuations converge to the Edwards-Wilkinson limit. This is also proved in \cite{Gu2018GaussianFF} for small enough $\hat{\beta}$. It would be interesting to try to combine our methods with the methods of \cite{10.1214/19-AOP1383} and of  \cite{Gu2018GaussianFF} to prove that, when $V$ is as in \eqref{noise},  $h(t/\epsilon^2,x/\epsilon)$ is asymptotically Gaussian, with fluctuations in the Edwards-Wilkinson class.\par
Finally, we mention that in $d=2$ with $V$ white in time and at the critical point, the mSHE has an interesting and nontrivial behavior. Loosely speaking, when $\hat{\beta}=\hat{\beta}_c$ and for $g\in C_c(\mathbb{R}^2)$ the random variable 
\begin{equation}\label{SHFav}
    \int_{\mathbb{R}^2}u_\epsilon(t,x) g(x)dx
\end{equation}
converges in distribution as $\epsilon\rightarrow0$ \cite{SHF}. The limiting object is called the critical $2d$ Stochastic Heat Flow. The moments of the critical $2d$ Stochastic Heat Flow admit explicit expressions \cite{SHFmoments} and it is known that it is not a Gaussian multiplicative chaos \cite{Caravenna_2023}. \textbf{Theorem \ref{thm:Main}} indicates that a critical point $\hat{\beta}_c$ exists when $V$ is correlated in time. Therefore, it would be interesting to consider \eqref{mSHEmod}, with $V$ as in \eqref{noise}, at the critical point and establish convergence in distribution for \eqref{SHFav}. If this is possible, it would also be interesting to compare this limiting object with the critical Stochastic Heat Flow obtained when $V$ is white in time.

\subsection{Idea of the proof and outline}\label{sec:2}
We use methods and arguments inspired from \cite{Gu_2018} to prove $\textbf{Theorem \ref{thm:Main}}$. More specifically, for any $g\in C_c(\mathbb{R}^2)$, we need
 two formulas from \cite{Gu_2018} for the random variable:
\begin{equation}\label{eq:2.1}
    \int_{\mathbb{R}^2}(u_{\epsilon}(t,x)-\mathbf{E}[u_{\epsilon}(t,x)])g(x)dx.
\end{equation}
These formulas are proved in \cite{Gu_2018} for $d\geq3$, and their proof carries exactly in the case $d=2$. We start with the Feynman-Kac representation of $u(t,x)$. {Since} $V(t,x)$ is a (random) smooth function, we have
\begin{equation}\label{eq:2.2}
    u(t,x)=\E_B\bigg[\exp\biggl({\frac{\hat{\beta}}{\sqrt{\log\frac{1}{\epsilon}}}}\int_0^tV(t-s,x+B_{s})ds\biggr)\biggr],
\end{equation}
where $\E_B$ is the expectation with respect to  Brownian motion starting at 0, independent from the noise. Due to the time correlations induced by the time mollification, the usual martingale structure of the solution is destroyed. In practice, this is what induces the exponential tilting of the Wiener measure. More precisely, on the space
$$\Omega_T=\{\omega\in C([0,T])\textbf{ }|\textbf{ }\omega(0)=0\},$$
equipped with the usual Borel $\sigma-$algebra, we define the following probability measure:
\begin{equation}\label{expmeasure}
    d\hat{\P}^{(\epsilon)}_T(B):=\exp\biggl(\frac{\hat{\beta}^2}{2\log\frac{1}{\epsilon}}\int_{[0,T]^2}R(s-u, B(s)-B(u))dsdu-\zeta_T^{(\epsilon)}\biggr)W_T(dB),
\end{equation}
where $W_T(dB)$ is the usual Wiener measure on $\Omega_T$, and $\zeta_T^{(\epsilon)}$ is the normalization constant
\begin{equation}\label{zeta}
    \zeta^{(\epsilon)}_T:=\log\int_{\Omega_T}\exp\biggl(\frac{\hat{\beta}^2}{2\log\frac{1}{\epsilon}}\int_{[0,T]^2}R(s-u, B(s)-B(u))dsdu\biggr)W_T(dB).
\end{equation}
We denote by $\hat{\E}_{B,T}$ the corresponding expectation. This exponential tilting already appears at the level of the moments of $u(t,x)$, as the next calculation shows:

\begin{myprop}\label{thm:prop2.1}
    If R is the correlation function of V then, for every $n\in\mathbb{N}$, we have:
    
    $$e^{-n\zeta^{(\epsilon)}_{t}}\mathbf{E}[u(t,x)^n]=\Hat{\E}_{B^1,...,B^n;t}\biggl[\exp\biggl(\frac{\hat{\beta}^2}{{\log\frac{1}{\epsilon}}}\sum_{1\leq i< j\leq n}\int_{[0,t]^2}R(s-u,B^i(s)-B^j(u))dsdu\biggr)\biggr],$$
    where $\Hat{\E}_{B^1,...,B^n;t}$ denotes the expectation under $n$ independent paths sampled from the exponentially tilted measure $\Hat{\P}^{(\epsilon)}_t$.
\end{myprop}

\begin{proof}
    For simplicity, we prove this for $n=2$. The general formula is proved similarly. By plugging into $(\ref{eq:2.2})$ the definition of $V$, \eqref{noise}, we get:
    $$u(t,x)=\E_B\biggl[\exp\biggl({\frac{\hat{\beta}}{\sqrt{\log\frac{1}{\epsilon}}}}\int_0^t\int_{\mathbb{R}}\int_{\mathbb{R}^2}\phi(t-s-s_1)\psi(x+B(s)-x_1)d\xi(s_1,x_1)\biggr)\biggr].$$
    Now we define:
    \begin{equation}\label{eq:2.3}
      \Phi_{t,x,B}(s_1,x_1):=\int_0^t\phi(t-s-s_1)\psi(x+B(s)-x_1)ds,  
    \end{equation}
    and:
    \begin{equation}\label{eq:2.4}
        M_{t,x,B}(\infty):=\int_{\mathbb{R}^{1+2}}\Phi_{t,x,B}(s_1,x_1)d\xi(s_1,x_1).
    \end{equation}
    For every Brownian path $B$, $M_{t,x,B}(\infty)$ is a mean zero Gaussian random variable with second moment equal to 
    \begin{equation}\label{eq:2.6}
       \int_{\mathbb{R}^{1+2}}\Phi_{t,x,B}(s_1,x_1)^2ds_1dx_1=\int_{[0,t]^2}R(s-u, B(s)-B(u))dsdu.
    \end{equation}
    With this notation, we can write
    $$u(t,x)=\E_B\biggl[\exp\biggl({\frac{\hat{\beta}}{\sqrt{\log\frac{1}{\epsilon}}}}M_{t,x,B}(\infty)\biggr)\biggr].$$ 
    From \eqref{zeta} and $(\ref{eq:2.6})$
    $$\mathbf{E}[u(t,x)]=\exp(\zeta_t^{(\epsilon)}).$$
    For the second moment, we have
    $$\mathbf{E}[u(t,x)^2]=\E_{B^1}\E_{B^2}\biggr[\mathbf{E}\biggr[\exp\biggl({\frac{\beta}{\sqrt{\log\frac{1}{\epsilon}}}}(M_{t,x,B^1}(\infty)+M_{t,x,B^2}(\infty))\biggr)\biggl]\biggl],$$
    which is equal to
    \begin{equation}\label{eq:2.7}
        \E_{B^1}\E_{B^2}\biggr[\exp\biggl(\frac{\hat{\beta}^2}{{2\log\frac{1}{\epsilon}}}(\mathbf{E}[M_{t,x,B^1}(\infty)^2]+\mathbf{E}[ M_{t,x,B^2}(\infty)^2]+2\mathbf{E}[M_{t,x,B^1}(\infty)M_{t,x,B^2}(\infty)])\biggr)\biggr],
    \end{equation}
    A straightforward calculation shows that
    $$\mathbf{E}[M_{t,x,B^1}(\infty)M_{t,x,B^2}(\infty)]=\int_{[0,t]^2}R(s-u,B^1(s)-B^2(u))dsdu.$$
    Plugging this in $(\ref{eq:2.7})$, we get 
    $$e^{-2\zeta_{t}^{(\epsilon)}}\mathbf{E}[u(t,x)^2]=\Hat{\E}_{B^1,B^2;t}\biggl[\exp\biggl(\frac{\hat{\beta}^2}{{\log\frac{1}{\epsilon}}}\int_{[0,t]^2}R(s-u,B^1(s)-B^2(u))dsdu\biggr)\biggr].$$
    \end{proof}
The above calculation indicates that $e^{-\zeta^{(\epsilon)}_{t/\epsilon^2}}$ is the correct exponential correction needed to keep $\mathbf{E}[u_{\epsilon}(t,x)]$ bounded. Indeed, for all $\epsilon>0$,
\begin{equation}\label{renorm}
    e^{-\zeta^{(\epsilon)}_{t/\epsilon^2}}\mathbf{E}[u_{\epsilon}(t,x)]=1.
\end{equation}
The proof of \textbf{Theorem \ref{thm:Main}} is split into two steps. The first step is to find the limiting variance of
\begin{equation}\label{eq:2.8}
   \sqrt{\log\frac{1}{\epsilon}}\int_{\mathbb{R}^2}e^{-\zeta^{(\epsilon)}_{t/\epsilon^2}}(u_{\epsilon}(t,x)-\mathbf{E}[u_{\epsilon}(t,x)])g(x)dx.
\end{equation}
{Observe that the convergence of the variance of (\ref{eq:2.8}) proves that  $\int_{\mathbb{R}^2}e^{-\zeta^{(\epsilon)}_{t/\epsilon^2}}u_{\epsilon}(t,x)g(x)dx\rightarrow\int_{\mathbb{R}^2}g(x)dx$ in $\mathbf{P}$-probability as $\epsilon\rightarrow0$}.\par
The second step is using mixing arguments to prove the central limit theorem \eqref{CLT}. These arguments allow us to approximate \eqref{eq:2.8} in $L^2$  by a sum of independent random variables. Then we conclude by applying Lindenberg's CLT.\par
As mentioned, we need two formulas from \cite{Gu_2018}. We define 
\begin{equation}\label{eq:2.4}
        M_{t,x,B}(r):=\int_{-\infty}^r\int_{\mathbb{R}^2}\Phi_{t,x,B}(s_1,x_1)d\xi(s_1,x_1).
    \end{equation}
    Observe that $(M_{t,x,B}(r))_{r\in\mathbb{R}}$ is a square-integrable continuous  martingale with quadratic variation equal to:
    \begin{equation}\label{eq:2.5}
        \langle M_{t,x,B}\rangle(r)=\int_{-\infty}^r\int_{\mathbb{R}^2}\Phi_{t,x,B}(s_1,x_1)^2ds_1dx_1.
    \end{equation}
The first formula we need is contained in the following proposition. The proof in $d=2$ is the same as in \cite{Gu_2018} (see \textbf{Lemma 2.1} from the same paper):
\begin{myprop}\label{thm:prop2.2}
    We have that:
    $$(u_{\epsilon}(t,x)-\mathbf{E}[u_{\epsilon}(t,x)])e^{-\zeta^{(\epsilon)}_{t/\epsilon^2}}=$$    
    $$\frac{\hat{\beta}}{\sqrt{\log\frac{1}{\epsilon}}}\int_{-1}^{t/\epsilon^2}\int_{\mathbb{R}^2}\Hat{\E}_{B,t/\epsilon^2}\biggl[\Phi^{\epsilon}_{t,x,B}(r,y)\exp\biggl({\frac{\hat{\beta}}{{\sqrt{\log\frac{1}{\epsilon}}}}M^{\epsilon}_{t,x,B}(r)-\frac{\hat{\beta}^2}{{2{\log\frac{1}{\epsilon}}}}\langle M^{\epsilon}_{t,x,B}\rangle(r)}\biggr)\biggr]d\xi(r,y),$$
    where $\Phi^{\epsilon}_{t,x,B}=\Phi^{}_{t/\epsilon^2,x/\epsilon,B}$ and $M^{\epsilon}_{t,x,B}=M^{}_{t/\epsilon^2,x/\epsilon,B}$, with $\Phi^{}_{t,x,B}$ and $M^{}_{t,x,B}$ defined in $(\ref{eq:2.2})$ and $(\ref{eq:2.3})$ respectively.
    \end{myprop}
This formula is proved by combining the Clark-Ocone formula\footnote{$D_{r,y}u(t,x)$ is the Malliavin derivative of $u(t,x)$ at $(r,y)$, see \cite{nualart2009malliavin}.}
$$u(t,x)-\mathbf{E}[u(t,x)]=\int_{-\infty}^t\int_{\mathbb{R}^2}\mathbf{E}[D_{r,y}u(t,x)|\Fcal_r]d\xi(r,y),$$
with the Feynman-Kac formula and similar calculations as in the previous proof.\par
\textbf{Proposition \ref{thm:prop2.2}} gives us a stochastic integral representation for $(\ref{eq:2.8})$:
$$\sqrt{\log\frac{1}{\epsilon}}\int_{\mathbb{R}^2}e^{-\zeta^{(\epsilon)}_{t/\epsilon^2}}(u_{\epsilon}(t,x)-\mathbf{E}[u_{\epsilon}(t,x)])g(x)dx=\hat{\beta}\int_{-1}^{t/\epsilon^2}\int_{\mathbb{R}^2}Z_t^{\epsilon}(r,y)d\xi(r,y),$$
where:
$$Z_t^{\epsilon}(r,y):=\int_{\mathbb{R}^2}g(x)\Hat{\E}_{B,t/\epsilon^2}\biggl[\Phi^{\epsilon}_{t,x,B}(r,y)\exp\biggl({\frac{\hat{\beta}}{{\sqrt{\log\frac{1}{\epsilon}}}}M^{\epsilon}_{t,x,B}(r)-\frac{\hat{\beta}^2}{{2{\log\frac{1}{\epsilon}}}}\langle M^{\epsilon}_{t,x,B}\rangle(r)}\biggr)\biggr]dx.$$
We are going to use this formula in the mixing arguments mentioned above to prove \textbf{Theorem \ref{thm:Main}}. The mixing arguments are implemented by modifying the martingale $M^{\epsilon}_{t,x, B}$ so that the stochastic integral above can be split into a sum of independent random variables and a 'negligible' remainder. Then, as mentioned,  Lindenberg's CLT and the information we will get from the limiting variance will allow us to conclude. \par
This stochastic integral representation of $(\ref{eq:2.8})$ is also useful for the calculation of its limiting variance, which is the more complicated step. This is the content of the final formula needed from \cite{Gu_2018}. First, we need to introduce some notation. We define:
\begin{equation}\label{eq:2.9}
    I_{\epsilon}(x_1,x_2,y,s_1,s_2,r):=\prod_{i=1}^2g(\epsilon x_i-\epsilon B^i({(t-r)/\epsilon^2-s_i})+y),
\end{equation}
\begin{equation}\label{Rphi}
  R_{\phi}(t_1,t_2):=\int_0^{\infty}\phi(s-t_1)\phi(s-t_2)ds,    
\end{equation}

\begin{align}\label{eq:2.10}
    J_{\epsilon}(M_1,M_2):=\frac{\hat{\beta}^2}{{\log\frac{1}{\epsilon}}}\int_{-1}^{M_1}\int_{-1}^{M_2}R_{\phi}(u_1,u_2)\psi\star\psi(x_1-x_2+
    \Delta B^1_{(t-r)/\epsilon^2-s_1,(t-r)/\epsilon^2+u_1}-\nonumber\\
    -\Delta B^2_{(t-r)/\epsilon^2-s_2,(t-r)/\epsilon^2+u_2})du_1du_2,
\end{align}
where  for a path $B$, we define $\Delta B_{s,u}=B(s)-B(u)$. 
\begin{myprop}\label{thm:prop2.3}
    For any $t_1<t-\epsilon^2$ we have the following formula:
    $$\mathbf{Var}\biggr(\int_{-\infty}^{t_1/\epsilon^2}\int_{\mathbb{R}^2}Z_t^{\epsilon}(r,y)d\xi(r,y)\biggl)=$$
    \begin{align}\label{eq:2.11}
        =\int_{0}^t\int_{\mathbb{R}^{6}}\int_{[0,1]^2}\Hat{\E}_{B,t/\epsilon^2}[I_{\epsilon}(x_1,x_2,s_2,s_2,r)e^{J_{\epsilon}(r/\epsilon^2,r/\epsilon^2)}]\prod_{i=1}^2\phi(s_i)\psi(x_i)d\Bar{s}d\Bar{x}dydr,
    \end{align}
    with $d\Bar{s}=ds_1ds_2$, $d\Bar{x}=dx_1dx_2$.
\end{myprop}
The calculation is the same as in $d\geq3$ and involves using It\^o's isometry for the stochastic integral appearing above, and a series of changes of variables. It is done in detail in \cite{Gu_2018} (see \textbf{Lemma 2.3} in the paper). Also, a similar formula is true for a general $t_1\leq t$, with some modifications. The modifications are the following:
\begin{itemize}
    \item The domain of $s_1,s_2$ changes from $[0,1]^2$ to $[0,(t-r)/\epsilon^2]^2.$
    \item The domain of $u_1,u_2$, that appear in the definition of $J_{\epsilon}(r/\epsilon^2,r/\epsilon^2)$, changes from $[-1,r/\epsilon^2]^2$ to $[-(t-r)/\epsilon^2,r/\epsilon^2]^2.$
\end{itemize}
Since $\phi$ is supported in $[0,1]$,  $R_{\phi}(u_1,u_2)$ is zero if either $u_i$ is less than $-1$. This shows that these changes create a difference only when $t-r\leq\epsilon^2$.\par
We are going to use this formula to find the limiting variance of \eqref{eq:2.8}. To do this, we analyze the functional appearing inside the integral on the right-hand side of \eqref{eq:2.11}. In particular, we define:\footnote{Sometimes, for brevity, we will avoid writing all the variables that $I_\epsilon$ depends on.}
\begin{equation}\label{Fdefinition}
   \Fcal_{\epsilon}(r,y,M_1,M_2)=\int_{\mathbb{R}^{4}}\int_{[0,1]^2}\Hat{\E}_{B,t/\epsilon^2}[I_{\epsilon}e^{J_{\epsilon}(M_1,M_2)}]\prod_{i=1}^2\phi(s_i)\psi(x_i)d\Bar{s}d\Bar{x}. 
\end{equation}
We seek to find a limit for this object, along with appropriate bounds in $y$ to apply the dominated convergence theorem. This is done in the following proposition, which is proved in \textbf{Section \ref{sec:6}}:
\begin{myprop}\label{thm:prop2.4}
    Let $0<M_1(\epsilon),M_2(\epsilon)\leq r$ be such that $\log M_i(\epsilon)/\log\frac{1}{\epsilon}\rightarrow0$ and such that for all $\epsilon$ small enough either $M_1(\epsilon)=M_2(\epsilon)$ or $M_1(\epsilon)-M_2(\epsilon)\geq c>0$ for some $c$. Then for any $r\in(0,t), y\in\mathbb{R}^2$, $\hat{\beta}<\hat{\beta}_c(R)$ and $k>0$ we have:
    \begin{equation}\label{eq:2.12}
        |\Fcal_{\epsilon}(r,y,M_1(\epsilon)/\epsilon^2,M_2(\epsilon)/\epsilon^2)|\leq C(1\wedge|y|^{-k}),
    \end{equation}
    for some constant $C$, depending only on $k$. Furthermore, as $\epsilon\rightarrow0$ we have:
    $$\Fcal_{\epsilon}(r,y,M_1(\epsilon)/\epsilon^2,M_2(\epsilon)/\epsilon^2)\rightarrow v^2_{eff}(\hat{\beta})p_{t-r}\star g(y)^2,$$
    where $p_t(y)=\frac{1}{2\pi t}e^{-|x|^2/2t}$, is the two dimensional heat kernel, and $v_{eff}^2(\hat{\beta})$ as in \textbf{Theorem \ref{thm:Main}}.
\end{myprop}
{\begin{remark}\label{Nikosremark}
    We put the condition
    $$M_1(\epsilon)-M_2(\epsilon)\geq c>0,$$ 
     only for convenience. In this case, the proof simplifies to the case where $M_1(\epsilon)=M_2(\epsilon)$. Indeed, if $M_1(\epsilon)-M_2(\epsilon)\geq c>0$ then, for all $\epsilon$ small enough, $M_1(\epsilon)/\epsilon^2\geq M_2(\epsilon)/\epsilon^2+1$. Moreover, since $\phi$ is supported on $[0,1]$, $R_{\phi}(u_1,u_2)=0$ when $|u_1-u_2|\geq1$. Therefore, if $M_1(\epsilon)-M_2(\epsilon)\geq c>0$
    $$J_{\epsilon}(M_1(\epsilon)/\epsilon^2,M_2(\epsilon)/\epsilon^2)=J_{\epsilon}(M_2(\epsilon)/\epsilon^2+1,M_2(\epsilon)/\epsilon^2),$$
    for all $\epsilon$ sufficiently small. This implies that 
    $$\Fcal_{\epsilon}(r,y,M_1(\epsilon)/\epsilon^2,M_2(\epsilon)/\epsilon^2)=\Fcal_{\epsilon}(r,y,M_2(\epsilon)/\epsilon^2+1,M_2(\epsilon)/\epsilon^2).$$
    Using this equality, we are going to argue in \textbf{Section \ref{sec:6}} that 
    $$|\Fcal_{\epsilon}(r,y,M_1(\epsilon)/\epsilon^2,M_2(\epsilon)/\epsilon^2)-\Fcal_{\epsilon}(r,y,M_2(\epsilon)/\epsilon^2,M_2(\epsilon)/\epsilon^2)|\rightarrow0,$$
    as $\epsilon\rightarrow0$. 
\end{remark}}
This proposition gives us the limit for the variance appearing in $(\ref{eq:2.11})$ for any $t_1<t-\epsilon^2$. We will argue in \textbf{Section \ref{sec:6}} that this can be extended to the case $t_1=t$, thus giving us the limiting variance. This is done in \textbf{Corollary \ref{cor6.1}}. Observe that the limiting variance is then equal to
$$\hat{\beta}^2v^2_{eff}(\hat{\beta})\int_0^t\int_{\mathbb{R}^2}p_{t-r}\star g(y)^2dydr,$$
which is equal to
$$\mathbf{Var}\biggl(\int_{\mathbb{R}^2}\mathcal{U}(t,x)g(x)dx\biggr),$$
with $\mathcal{U}(t,x)$  as in \textbf{Theorem \ref{thm:Main}}.\par
In $\textbf{Section \ref{sec:2)}}$ we prove $\textbf{Theorem \ref{thm:Main}}$ assuming \textbf{Proposition \ref{thm:prop2.4}} and \textbf{Corollary \ref{cor6.1}}. As mentioned, we use the mixing arguments appearing in \cite{Gu_2018}. However, their arguments can give the CLT only in a restricted region of $\hat{\beta}$. This is because they eventually have to verify Lindenberg's principle and they do this by essentially estimating the fourth moment of $(\ref{eq:2.1})$. This restricts the range of $\hat{\beta}$ since they have estimates only for the second moment. We avoid this problem by using an argument from \cite{10.1214/19-AOP1383} which uses Gaussian hypercontractivity and the fact that we are strictly inside the subcritical regime.\par
The rest of the paper is dedicated to the proof of \textbf{Proposition \ref{thm:prop2.4}}. It is clear that to prove it we have to deal with functionals of continuous paths which are distributed according to $\hat{\P}^{(\epsilon)}_T$. This is the bulk of this paper. Similar measures on the space of continuous paths have been considered before due to their connection to quantum mechanical models (see \cite{Mukherjee2017CentralLT} and the references therein). We analyze $\hat{\P}_T^{(\epsilon)}$ in a similar way as in \cite{Gu_2018, Mukherjee2017CentralLT}, by introducing a Markov chain on the space of continuous paths over $[0,1]$ and expressing the measure $\hat{\P}^{(\epsilon)}_T$ as a transition probability of this Markov chain. The way we can guess this Markov chain is pretty simple: We could try to build a path distributed according to $\hat{\P}^{(\epsilon)}_T$ by gluing together $T-$many smaller paths over $[0,1]$ distributed according to $\hat{\P}^{(\epsilon)}_1$. However, these paths will interact under $\hat{\P}^{(\epsilon)}_T$, which will lead us to an explicit probability kernel that captures this interaction. This can be used to define the Markov chain, thus giving us an explicit way to construct a path distributed according to $\hat{\P}^{(\epsilon)}_T$ (for more details see the next section).\par
In \textbf{Sections \ref{sec:3}} and \textbf{\ref{sec:4}} we study the properties of this Markov chain. This Markov chain satisfies a Doeblin condition: the transition probability measure is bounded below by a small multiple of the Wiener measure. From standard arguments, via a coupling with a sequence of Bernoulli random variables, we have a positive probability of sampling independently from the Wiener measure at each step of the Markov chain. This coupling gives the Markov chain a regeneration structure, and more specifically the path built from the Markov chain is a regenerative process \cite{Rege_pro_Tal}. As a result,  the total path increment of the path built from the Markov chain, up to a regeneration time, is a random walk in $\R^2$, a fact that plays an important role in our arguments.  Moreover, the fact that we tuned the coupling constant as:
$${\beta}_{\epsilon}=\frac{\hat{\beta}}{\sqrt{\log\frac{1}{\epsilon}}},$$
leads to some simplifications when compared to the case of $d\geq3$, for example, a stronger mixing property for the Markov chain as $\epsilon\rightarrow0$ (see $\textbf{Lemma \ref{thm:lem4.1}}$).\par

{The second moment calculations necessary to prove \textbf{Proposition \ref{thm:prop2.4}} is our main departure from the techniques in \cite{Gu_2018}. The added difficulty in our case is  best seen at the level of the moments of $u_{\epsilon}(t,x)$. From \textbf{Proposition \ref{thm:prop2.1}} we have:
$$e^{-2\zeta_{t/\epsilon^2}^{(\epsilon)}}\mathbf{E}[u_{\epsilon}(t,x)^2]=\Hat{\E}_{B^1,B^2;t/\epsilon^2}\biggl[\exp\biggl(\frac{\hat{\beta}^2}{{\log\frac{1}{\epsilon}}}\int_{[0,t/\epsilon^2]^2}R(s-u,B^1(s)-B^2(u))dsdu\biggr)\biggr].$$
The formula holds true when $d\geq3$, only with $\hat{\beta}^2/{\log\frac{1}{\epsilon}}$ replaced by an $\epsilon-$independent constant. To find the limit of $e^{-2\zeta_{t/\epsilon^2}^{(\epsilon)}}\mathbf{E}[u_{\epsilon}(t,x)^2]$ as $\epsilon\rightarrow0$, we need to consider the limiting distribution of:
\begin{equation}\label{additive}
    \int_{[0,t/\epsilon^2]^2}R(s-u,\omega^1(s)-\omega^2(u))dsdu,
\end{equation}
where $\omega^1,\omega^2$ are two independent paths built from the Markov chain introduced in \textbf{Section \ref{sec:3}}.}\par
{In $d\geq3$, this is relatively easy to do, since \eqref{additive} converges to a random variable with exponential moments. This is proved in \cite{Gu_2018} (Corollary $4.4$ in the same paper). This result is somewhat expected since $(\epsilon \omega^1_{s/\epsilon^2})_{s\leq t}$ converges to a Brownian motion with an effective diffusivity (see Proposition 4.1 in \cite{Gu_2018} for the precise statement).}\par
{In $d=2$ we expect $(\epsilon \omega^1_{s/\epsilon^2})_{s\leq t}$ to behave similarly and to converge to a Brownian motion. Since in $d=2$ the Brownian motion is recurrent, \eqref{additive}  diverges almost surely as $\epsilon\rightarrow0$. This is the reason for the logarithmic dependence of the coupling constant $\beta$. Indeed, in $\textbf{Section \ref{sec:5}}$,  we will prove that:
\begin{equation}\label{additive2}
    \frac{1}{{\log\frac{1}{\epsilon}}}\int_{[0,t/\epsilon^2]^2}R(s-u,\omega^1(s)-\omega^2(u))dsdu
\end{equation}
converges in distribution to a multiple of an exponential random variable of rate $1$.}\par
{If $R$ does not depend on time, and $\omega^1, \omega^2$ are independent Brownian motions, then this result is the Kallianpur-Robbins law for the Brownian motion \cite{Kal-Rob}. Therefore, we can think of the result of \textbf{Section \ref{sec:5}} as a 'non-directed' version of the Kallianpur-Robbins law for the regenerative process $(\omega^1,\omega^2)$. Naturally, to prove this, we use the regenerative structure of the process. This means that the process can be split into paths of random length called cycles, such that the cycles are independent and identically distributed. We exploit this structure to analyze the additive functional \eqref{additive2}, by splitting the functional as a sum over the cycles and then averaging over the law of the cycles. This will give us an additive functional of the total path increments. As mentioned, the total path increment is a random walk in $\mathbb{R}^2$.  Then, to properly analyze this additive functional of the total path increments, we will prove a 'non-directed' version of a result in \cite{Kallianpur1954TheSO} (the discrete version of the Kallianpur-Robbins law)  (see \textbf{Lemma \ref{lemm4.3}}).}\par
In $\textbf{Section \ref{sec:6}}$, we prove \textbf{Proposition \ref{thm:prop2.4}} by exploiting the mixing property of the Markov chain and the information we got in $\textbf{Section \ref{sec:5}}$ for the limiting distribution of \eqref{additive2}. The arguments we use here allow us to prove \textbf{Proposition \ref{thm:prop2.4}} in the full subcritical regime:
$$\hat{\beta}<\hat{\beta}_c(R).$$
As mentioned, this regime is optimal since the effective variance is infinite for $\hat{\beta}\geq\hat{\beta}_c(R)$. We can get an explicit form for the effective variance and the corresponding critical value because we know the limiting distribution of the additive functionals appearing in our calculations.\par
Finally, there are four appendices presenting respectively, some lemmas concerning the total path increments, the  'non-directed' version of the 'discrete' Kallianpur-Robbins law that was mentioned before, a lemma that we use to bound exponential moments of additive functionals and a result giving us the (first order) asymptotics of $\zeta_{t/\epsilon^2}^{(\epsilon)}$.

\subsection{Notation}
\begin{itemize}
    \item For a probability measure, $\mu$, on a measurable space, we write $X\sim\mu$ when any random variable, defined over any probability space, has law equal to $\mu$. We write $\E_{\mu}$ for the expectation with respect to $\mu$. Abusing notation, we will write $\E_{X}$ for the expectation with respect to the law of $X$, again defined over any probability space. Finally, $\mu_1\times\mu_2$ denotes the product measure of $\mu_1$ and $\mu_2$.
    \item $Geo(\gamma)$ denotes the geometric distribution with parameter $\gamma$ and $W_T$ the Wiener measure on $C([0,T])$.
    \item For $t\in\mathbb{R}$ we write $[t]\in\mathbb{Z}$ for the integer part of t.
    \item We use the notation $x\lesssim y$ to mean $x\leq C\cdot y$ for some constant $C$, irrelevant to the current argument. We also use the big-O notation to mean the same thing: $x=O(y)$ means $x\lesssim y$. Similarly, we say that $x_\epsilon=o(y_\epsilon)$ as $\epsilon\rightarrow0$ if $x_\epsilon/y_\epsilon\rightarrow0$ as $\epsilon\rightarrow0$. Finally, we write $x_\epsilon\sim y_\epsilon$ to indicate that $x_\epsilon/y_\epsilon\rightarrow1$ as $\epsilon\rightarrow0$ and we say that $x_\epsilon$ is asymptotic to $y_\epsilon$. 
    \item We use bold symbols to emphasize that the symbol is a vector. For two vectors $\mathbf{y},\mathbf{x}\in\mathbb{R}^d$ we denote by $\mathbf{x}^*$ the transpose of  $\mathbf{x}$ and by $\mathbf{x}\otimes\mathbf{y}=\mathbf{x}\mathbf{y}^*$. Moreover, we denote by $\langle \mathbf{x},\mathbf{y}\rangle$, their inner product. Finally, we denote by $|\mathbf{x}|$ the standard Euclidean norm of $\mathbf{x}$.
\end{itemize}

\section{Proof of Theorem \ref{thm:Main}}\label{sec:2)}
{In this section, we prove $\textbf{Theorem \ref{thm:Main}}$ assuming \textbf{Proposition \ref{thm:prop2.4}} and \textbf{Corollary \ref{cor6.1}}. From the latter, we get that the variance of 
$$\sqrt{\log\frac{1}{\epsilon}}\int_{\mathbb{R}^2}e^{-\zeta^{(\epsilon)}_{t/\epsilon^2}}(u_{\epsilon}(t,x)-\mathbf{E}[u_{\epsilon}(t,x)])g(x)dx$$
converges to a real number as $\epsilon\rightarrow0$. This observation, when combined with \eqref{renorm} and the Markov inequality, proves that 
$$\int_{\mathbb{R}^2}e^{-\zeta^{(\epsilon)}_{t/\epsilon^2}}u_{\epsilon}(t,x)g(x)dx\rightarrow\int_{\mathbb{R}^2}g(x)dx,$$ 
in $\mathbf{P}$-probability, as $\epsilon\rightarrow0$.}\par
To prove the central limit theorem in \textbf{Theorem \ref{thm:Main}}, we recall that, from \textbf{Proposition \ref{thm:prop2.2}}, we have
$$\sqrt{\log\frac{1}{\epsilon}}\int_{\mathbb{R}^2}e^{-\zeta^{(\epsilon)}_{t/\epsilon^2}}(u_{\epsilon}(t,x)-\mathbf{E}[u_{\epsilon}(t,x)])g(x)dx=\hat{\beta}\int_{-1}^{t/\epsilon^2}\int_{\mathbb{R}^2}Z_t^{\epsilon}(r,y)d\xi(r,y),$$
where
\begin{equation}\label{eq7.1}
    Z_t^{\epsilon}(r,y):=\int_{\mathbb{R}^2}g(x)\Hat{\E}_{B,t}\biggl[\Phi^{\epsilon}_{t,x,B}(r,y)\exp\biggl({\frac{\hat{\beta}}{{\sqrt{\log\frac{1}{\epsilon}}}}M^{\epsilon}_{t,x,B}(r)-\frac{\hat{\beta}^2}{{2{\log\frac{1}{\epsilon}}}}\langle M^{\epsilon}_{t,x,B}\rangle(r)}\biggr)\biggr]dx.
\end{equation}
We will split $[-1, t/\epsilon^2]$ into 'short' and 'long' intervals and then modify the martingale $M^{\epsilon}_{t,x, B}(r)$ over the long intervals. This will be done so that the contributions from the short intervals are negligible and the contributions over the long intervals are independent. Then, \textbf{Proposition \ref{thm:prop2.4}} and Lindenberg's criterion, will give the full central limit theorem.\par
More specifically, for $0<a<\lambda<2$,  we split $[-1,t/\epsilon^2]$ into successive intervals of length $\epsilon^{-a}$ (the short intervals) and of length $\epsilon^{-\lambda}$ (the long intervals):
$$[-1,t/\epsilon^2]=[-1,\epsilon^{-a}]\cup(\epsilon^{-a},\epsilon^{-a}+\epsilon^{-\lambda}]\cup...\cup(t_{\epsilon},t/\epsilon^2],$$
where $t/\epsilon^2-t_{\epsilon}=O(\epsilon^{-\lambda})$. For technical reasons, we choose $a$ and $\lambda$ to depend on $\epsilon$. We choose
$$a=2-\log\log\frac{1}{\epsilon}/\log\frac{1}{\epsilon}\textit{ and }\lambda=2-\log\log\frac{1}{\epsilon}/2\log\frac{1}{\epsilon}.$$ 
We denote by $(I_{a,j})$ the collection of all short intervals (where we also include the interval containing $t/\epsilon^2$ to this collection) and by $(I_{\lambda,j})$ the collection of all long intervals. The modification of ${M}^{\epsilon}_{t,x,B}$ is defined as
\begin{equation}\label{mod}
    \Tilde{M}^{\epsilon}_{t,x,B}(r):=\int_{-\infty}^r\int_{\mathbb{R}^2}\biggl(\int_0^{r_{\epsilon}}\phi(t/\epsilon^2-s_1-s)\psi(x/\epsilon+B({s_1})-y)ds_1\biggr)d\xi(s,y),
\end{equation}
where $r_{\epsilon}=t/\epsilon^2-r+1/2\epsilon^{a}$ and $r\in I_{\lambda,j}$, for some $j$ (note that in this case, $r\geq\epsilon^{-a}$, and therefore $r_\epsilon<t/\epsilon^2$). Since $r_\epsilon$ depends on $r$, $\tilde{M}^\epsilon_{t,x,B}$ is not a martingale. Nevertheless, we abuse notation and we define
$$\langle\Tilde{M}^{\epsilon}_{t,x,B}\rangle(r)=\int_{-\infty}^r\int_{\mathbb{R}^2}\biggr(\int_0^{r_{\epsilon}}\phi(t/\epsilon^2-s_1-s)\psi(x/\epsilon+B({s_1})-y)ds_1\biggl)^2 ds_1dyds.$$
Also, for $r\in I_{\lambda,j}$,  define $\Tilde{Z}_{t}^{\epsilon}(r,y)$ by $(\ref{eq7.1})$, with $\Tilde{M}^{\epsilon}_{t,x,B}$ replacing ${M}^{\epsilon}_{t,x,B}$. Finally, define
\begin{equation}\label{eq7.2}
    \mathcal{X}_{j}(\hat{\beta})=\int_{I_{\lambda,j}}\int_{\mathbb{R}^2}\Tilde{Z}^{\epsilon}_t(r,y)d\xi(r,y).
\end{equation}
From \cite{Gu_2018}, the random variables $\mathcal{X}_{j}(\hat{\beta})$ are independent. {Indeed,  the integrand in \eqref{mod} vanishes when $s\leq r-\epsilon^{-a}$, and $\epsilon$ is small enough. That is because $t/\epsilon^2-s_1-s\geq1/2\epsilon^{a}>1$ when $s\leq r-\epsilon^{-a}$ and $\epsilon$ small enough and so $\phi(t/\epsilon^2-s_1-s)=0$. Therefore, $\tilde{M}_{t,x,B}^\epsilon(r)$ depends on the underlying noise $\xi(s,y)$ only for $s\in(r-\epsilon^{-a},r]$. This proves that $(\mathcal{X}_j(\hat{\beta}))_{j}$ are independent.}\par
Moreover, the following are true:

\begin{itemize}
    \item If $I_{\lambda}$ is the union of all long intervals then
    \begin{equation}\label{7.3}
        \int_{I_{\lambda}}\int_{\mathbb{R}^2}\mathbf{E}[|{Z}_{t}^{\epsilon}(r,y)-\Tilde{Z}_{t}^{\epsilon}(r,y)|^2]dydr\rightarrow0,
    \end{equation}
    
    as $\epsilon\rightarrow0$. Indeed, from \cite{Gu_2018} {(proof of Lemma $3.1$)} we have
    \begin{equation}\label{7.4}
        \int_{I_{\lambda}}\int_{\mathbb{R}^2}\mathbf{E}[|{Z}_{t}^{\epsilon}|^2](r,y)dydr=\int_0^t\int_{\mathbb{R}^2}\textbf{1}_{\{r/\epsilon^2\in I_{\lambda}\}}\Fcal_{\epsilon}(r,y,r/\epsilon^2,r/\epsilon^2)dydr,
    \end{equation}
    \begin{equation}\label{7.5}
        \int_{I_{\lambda}}\int_{\mathbb{R}^2}\mathbf{E}[|\Tilde{Z}_{t}^{\epsilon}|^2](r,y)dydr=\int_0^t\int_{\mathbb{R}^2}\textbf{1}_{\{r/\epsilon^2\in I_{\lambda}\}}\Fcal_{\epsilon}(r,y,1/2\epsilon^{a},1/2\epsilon^{a})dydr,
    \end{equation}
    \begin{equation}\label{7.6}
        \int_{I_{\lambda}}\int_{\mathbb{R}^2}\mathbf{E}[{Z}_{t}^{\epsilon}(r,y)\Tilde{Z}_{t}^{\epsilon}(r,y)]dydr=\int_0^t\int_{\mathbb{R}^2}\textbf{1}_{\{r/\epsilon^2\in I_{\lambda}\}}\Fcal_{\epsilon}(r,y,r/\epsilon^2,1/2\epsilon^{a})dydr.
    \end{equation}
     Recalling that $ a=2-\log\log\frac{1}{\epsilon}/\log\frac{1}{\epsilon}$, so that $\epsilon^{-a}=\epsilon^{-2}/\log\frac{1}{\epsilon}$, and using \textbf{Proposition \ref{thm:prop2.4}}, 
    we can  verify that $(\ref{7.4})$, $(\ref{7.5})$ and \eqref{7.6} have the same limit. This proves \eqref{7.3}.
    \item Similarly, if $I_{a}$ is the union of all short intervals
    \begin{equation}\label{7.7}
         \int_{I_{a}}\int_{\mathbb{R}^2}\mathbf{E}[|{Z}_{t}^{\epsilon}(r,y)|^2]dydr\rightarrow0,
    \end{equation}
    as $\epsilon\rightarrow0$. Indeed, from \cite{Gu_2018} {(proof of Lemma $3.2$)}
    $$\int_{I_{a}}\int_{\mathbb{R}^2}\mathbf{E}[|{Z}_{t}^{\epsilon}(r,y)|^2]dydr=\int_0^t\int_{\mathbb{R}^2}\textbf{1}_{\{r/\epsilon^2\in I_{a}\}}\Fcal_{\epsilon}(r,y,r/\epsilon^2,r/\epsilon^2)dydr.$$
        Note that here $\Fcal_{\epsilon}(r,y,r/\epsilon^2
        ,r/\epsilon^3)$ has the modifications described after \textbf{Proposition \ref{thm:prop2.3}}. Nevertheless, \textbf{Proposition \ref{thm:prop2.4}} still holds (as it is argued in \textbf{Corollary \ref{cor6.1}} for example). By the bound provided by \textbf{Proposition \ref{thm:prop2.4}} and  $|\{r\in[0,t]\textbf{ : }r/\epsilon^2\in I_{a}\}|\rightarrow0$ as $\epsilon\rightarrow0$ we see  that $(\ref{7.7})$ holds.

    \item Now we will prove that for all $\hat{\beta}<\hat{\beta}_c(R)$
    \begin{equation}\label{7.8}
        \hat{\beta}\sum_{j}\mathcal{X}_j^{\epsilon}(\hat{\beta})\Rightarrow\int_{\mathbb{R}^2}\mathcal{U}(t,x)g(x)dx.
    \end{equation}
    where $\Rightarrow$ means convergence in distribution.  Again from \cite{Gu_2018} {(proof of Lemma $3.3$)}, we have
    $$\hat{\beta}^2\sum_j\mathbf{Var}[\mathcal{X}_j^{\epsilon}(\hat{\beta})]=\beta^2\sum_j\int_0^t\int_{\mathbb{R}^2}\textbf{1}_{\{r/\epsilon^2\in I_{\lambda,j}\}}\Fcal_{\epsilon}(r,y,1/2\epsilon^{a},1/2\epsilon^{a})dydr.$$
    From  \textbf{Proposition \ref{thm:prop2.4}} this converges to $\hat{\beta}^2v^2_{eff}(\hat{\beta})\int_0^t\int_{\mathbb{R}^2}p_{t-r}\star g(y)^2dydr$, which is equal to the variance of $\int_{\mathbb{R}^2}\mathcal{U}(t,x)g(x)dx$. Now, since $\mathcal{X}_j^{\epsilon}$ are independent, to prove the full central limit theorem, we need to check Lindeberg's condition
    \begin{equation}\label{7.9}
            \sum_j\mathbf{E}[|\mathcal{X}_j^{\epsilon}(\hat{\beta})|^2\textbf{1}_{\{|\mathcal{X}_j^{\epsilon}(\hat{\beta})|>\delta\}}]\rightarrow0,
    \end{equation}
    for any $\delta>0$ as $\epsilon\rightarrow0$. Choose $p>1$, so that $(2p-1)\hat{\beta}<\hat{\beta}_c(R)$. From the Hölder and the Chebyshev inequality
    \begin{equation*}
       \sum_j\mathbf{E}[|\mathcal{X}_j^{\epsilon}(\hat{\beta})|^2\textbf{1}_{\{|\mathcal{X}_j^{\epsilon}(\hat{\beta})|>\delta\}}]\leq\frac{1}{\delta}\sum_j\mathbf{E}[|\mathcal{X}_j^{\epsilon}(\hat{\beta})|^{2p}]^{1/p}(\mathbf{E}[\mathcal{X}_j^{\epsilon}(\hat{\beta})|^2])^{1-1/p}.
    \end{equation*}
     Recall that $\lambda=2-\log\log\frac{1}{\epsilon}/2\log\frac{1}{\epsilon}$. From the bound in \textbf{Proposition \ref{thm:prop2.4}},  $\mathbf{E}[\mathcal{X}_j^{\epsilon}(\hat{\beta})|^2]\lesssim\epsilon^{2-\lambda}\lesssim1/\log\frac{1}{\epsilon}$, for all $j$. Moreover
    $$\sum_j\mathbf{E}[|\mathcal{X}_j^{\epsilon}(\hat{\beta})|^{2p}]^{1/p}=\sum_j||\mathcal{X}_j^{\epsilon}(\hat{\beta})||_{2p}^{2}\lesssim\sum_j||\mathcal{X}_j^{\epsilon}((2p-1)\hat{\beta})||_{2}^2\lesssim1,$$
    as $\epsilon\rightarrow0$, where we used \textbf{Lemma \ref{lemm7.1}} below and the fact that when $(2p-1)\hat{\beta}<\hat{\beta}_c(R)$, the sum in the right-most side of the above inequality converges as $\epsilon\rightarrow0$. These two observations prove $(\ref{7.9})$, which proves \eqref{7.8}. 
\end{itemize}

These three items prove the full central limit theorem appearing in \textbf{Theorem \ref{thm:Main}}. The only thing left to prove is the following:
\begin{mylem}\label{lemm7.1}
    For $p>2$ such that $(p-1)\hat{\beta}<\hat{\beta}_c(R)$ we have:
    $$||\mathcal{X}_j^{\epsilon}(\hat{\beta})||_p\lesssim||\mathcal{X}_j^{\epsilon}((p-1)\hat{\beta})||_2$$
\end{mylem}

\begin{proof}
     {As we will see, this is just an instance of hypercontractivity for Wiener chaos.} Recall the definition of $\mathcal{X}_j(\hat{\beta})$:

    $$\mathcal{X}_{j}(\hat{\beta})=\int_{I_{\lambda,j}}\int_{\mathbb{R}^2}\Tilde{Z}^{\epsilon}_t(r,y)d\xi(r,y),$$
    where

    $$\Tilde{Z}_t^{\epsilon}(r,y)=\int_{\mathbb{R}^2}g(x)\Hat{\E}_{B,t}\biggl[\Phi^{\epsilon}_{t,x,B}(r,y)\exp\biggl({\frac{\hat{\beta}}{{\sqrt{\log\frac{1}{\epsilon}}}}\Tilde{M}^{\epsilon}_{t,x,B}(r)-\frac{\hat{\beta}^2}{{2{\log\frac{1}{\epsilon}}}}\langle\Tilde{M}^{\epsilon}_{t,x,B}\rangle(r)}\biggr)\biggr]dx.$$
    Observe that the random variable $\Tilde{M}_{t,x, B}^{\epsilon}$, for a fixed realization of the path B,  is an integral of a function that is deterministic with respect to the noise. Hence, it is a mean zero Gaussian random variable, and $\langle\Tilde{M}_{t,x, B}^{\epsilon}\rangle$ is its second moment. Therefore, the exponential in the above equation is a Wick exponential \cite{janson_1997}:
    $$\Tilde{Z}_t^{\epsilon}(r,y)=\int_{\mathbb{R}^2}g(x)\Hat{\E}_{B,t}\biggl[\Phi^{\epsilon}_{t,x,B}(r,y):\exp\biggl(\frac{\hat{\beta}}{{\sqrt{\log\frac{1}{\epsilon}}}}\Tilde{M}^{\epsilon}_{t,x,B}(r)\biggr):\biggr]dx,$$
    where $:\exp X:$ is the Wick exponential of the random variable $X$. We use this to find a chaos expansion for $\mathcal{X}_{j}^{\epsilon}(\hat{\beta})$. We have that:
    $$\Tilde{Z}_t^{\epsilon}(r,y)=$$
    $$\int_{\mathbb{R}^2}g(x)\Hat{\E}_{B,t}\biggl[\Phi^{\epsilon}_{t,x,B}(r,y)\biggr(1+\sum_{k=1}^{\infty}\frac{\hat{\beta}^k}{(\log\frac{1}{\epsilon})^kk!} \int\cdot\cdot\cdot\int_{\Delta_k\times\mathbb{R}^{2k}}\biggr[\prod_{j=1}^{k}\tilde{\Phi}^{\epsilon}_{t,x,B}(r_i,y_i)\biggl]\prod_{j=1}^{k}d\xi(r_i,y_i)\biggl)\biggr]dx,$$
   where $\Delta_k=\{-\infty<r_1<...<r_{k}<r\}$ and $\tilde{\Phi}^{\epsilon}(s,y)=\int_0^{r_{\epsilon}}\phi(t/\epsilon^2-s_1-s)\psi(x/\epsilon+B_{s_1}-y)ds_1$ with $r_{\epsilon}$ defined after \eqref{mod}. Plugging that in the expression for $\mathcal{X}_j^{\epsilon}(\beta)$  yields
    \begin{align*}
        \mathcal{X}_j^{\epsilon}(\hat{\beta})=\mathcal{W}_0^{\epsilon}+\sum_{k=1}^{\infty}\frac{\hat{\beta}^k}{(\log\frac{1}{\epsilon})^kk!}\mathcal{W}_k^{\epsilon},
    \end{align*}
    where    
    \begin{equation*}
        \mathcal{W}_k^{\epsilon}=\int\cdot\cdot\cdot\int_{\Delta_{k+1}'\times\mathbb{R}^{2(k+1)}}\biggr[\int_{\mathbb{R}^2}g(x)\Hat{\E}_{B,t}\biggr[\prod_{j=1}^{k}\tilde{\Phi}^{\epsilon}_{t,x,B}(r_i,y_i)\Phi^{\epsilon}_{t,x,B}(r_{k+1},y_{k+1})\biggl]dx\biggl]\prod_{j=1}^{k+1}d\xi(r_i,y_i),
    \end{equation*}
    with $\Delta_{k+1}'=\{-\infty<r_1<...<r_{k+1}<\infty\textbf{, }r_{k+1}\in I_{\lambda,j}\}$. Clearly $\mathcal{W}_k^{\epsilon}$  lies in the $k+1$ homogeneous Wiener chaos (again see \cite{janson_1997}) for all $k\in\mathbb{N}_0$.  By hypercontractivity for Wiener chaos \cite{janson_1997}, we get that:

    $$||\mathcal{X}_{j}^{\epsilon}(\hat{\beta})||_p\lesssim\biggl(\sum_{k=0}^{\infty}\frac{((p-1)\beta)^{k}}{(\log\frac{1}{\epsilon})^{k}k!}||\mathcal{W}_{k}^{\epsilon}||_2^2\biggr)^{1/2}=||\mathcal{X}_{j}^{\epsilon}((p-1)\hat{\beta})||_{2}.$$
\end{proof}

\section{The Markov Chain on $\Omega_1$}
\subsection{The construction}\label{sec:3}
As mentioned, to prove \textbf{Proposition \ref{thm:prop2.4}}, we have to be able to control expressions of the form:
$$\Hat{\E}_{B,t/\epsilon^2}[f(B)],$$
as $\epsilon\rightarrow0$. This is done by introducing a Markov chain on $\Omega_1$. We follow the steps of \cite{Gu_2018} to construct this Markov chain and detail its properties.\par
The goal is to express the measure $\hat{\P}^{(\epsilon)}_T$ as a transition probability of a well-chosen Markov chain. To achieve this, we make the following crucial observation:\par
Let $B$ be distributed according to the measure $\hat{\P}^{(\epsilon)}_T$ and recall the definition of this measure:
$$\hat{\P}^{(\epsilon)}_T(dB)=\exp\biggl(\frac{\beta^2}{2\log\frac{1}{\epsilon}}\int_{[0,T]^2}R(s-u, B(s)-B(u))dsdu-\zeta^{(\epsilon)}_T\biggr)W_T(dB).$$
Since $R(s,\cdot)=0$ for $|s|\geq1$, two points $B(s),B(u)$ interact only when $|s-u|<1$. This suggests that we can split the path $B$ into segments of length 1 and only neighboring segments will interact with each other. Now, heuristically, this interaction term can be seen as a transition probability. Therefore, if we choose the first segment according to an appropriate distribution, we can build a Markov chain using this transition probability. So we will get a sequence of paths in $\Omega_1$. Then, by gluing all of these paths together, we will get a path on $\Omega_T$ which will be distributed according to $\hat{\P}^{(\epsilon)}_T$.\par
Strictly speaking, the distribution of this path will not be exactly $\hat{\P}^{(\epsilon)}_T$ since we must account for some edge effects: The endpoint T may not be a natural number and for some technical reasons we may need to choose the first segment to be of length $\tau<1$. In practice, in all expectations that we will encounter, there will be an extra term accounting for the edge effects. This term can be ignored since, in $d\geq3$, it asymptotically decouples from all relevant random variables we are considering (see \textbf{Lemma A.1} in \cite{Gu_2018}), and in $d=2$,  it goes to $1$ uniformly, as $\epsilon\rightarrow0$.\newline\par

Let us be more specific. We consider any probability space $(\mathcal{X},\Fcal,\P)$ that is rich enough to carry the random variables that we are going to encounter and we write $\E$ for the corresponding expectation. As mentioned above we seek to build a continuous path over $[0, T]$ starting from 0 by gluing together paths from $\Omega_1$. The initial segment will be defined over $[0,\tau]$ where  $\tau\in(0,1]$. Then the final segment will be defined over $[0, T-\tau-N]$, where $N=[T-\tau]$. Every other segment will be defined over $[0,1]$. \par
To formalize this 'gluing' of the paths we set for $k=1,..., N$, $\tau_{k+1}=\tau_k+1$ with $\tau_0=0, \tau_1=\tau$ and $\tau_{N+2}=T$. Now choose paths $x_0\in\Omega_{\tau}$, $x_k\in\Omega_{1}$ for $k=1,...,N$ and $x_{N+1}\in\Omega_{T-\tau-N}$. We can patch these paths together over the intervals $(\tau_{k},\tau_{k+1})$ and make a path in $\Omega_T$ in the following way:
\begin{equation}\label{eq3.1}
    B(s)=\begin{cases}
        x_0(s) & \text{if $s\in[0,\tau]$},\\
        B(\tau+k-1)+x_k(s-\tau-k+1) & \text{if $s\in[\tau+k-1,\tau+k]$}, k=1,...,N,\\
        
        B(\tau+N)+x_k(s-\tau-N) & \text{if $s\in[\tau+N,T]$}.
    \end{cases}\newline
\end{equation}
It is easily checked that $B\in\Omega_T$. Again following \cite{Gu_2018} we will write
$$B=[x_0,...,x_{N+1}]$$
for a path in $\Omega_T$ built from  $(x_k)_{k=0,...,N+1}$.\par
With this terminology, we explain how to decompose the measure $\hat{\P}_T$. We write 
$$\int_{[0,T]^2}R(s-u,B(s)-B(u))dsdu=\sum_{k,m=1}^{N+1}\int_{[\tau_k,\tau_{k+1}]}\int_{[\tau_m,\tau_{m+1}]}R(s-u,B(s)-B(u))dsdu.$$
Since $R(s,\cdot)=0$ when $|s|\geq1$ the integral inside the sum is nonzero only when $|k-m|\leq1$ and so
     $$\int_{[0,T]^2}R(s-u,B(s)-B(u))dsdu=$$
     $$\sum_{k=0}^{N+1}\int_{[\tau_k,\tau_{k+1}]}\int_{[\tau_k,\tau_{k+1}]}R(s-u,B(s)-B(u))dsdu\\ +2\sum_{k=0}^{N}\int_{[\tau_k,\tau_{k+1}]}\int_{[\tau_{k+1},\tau_{k+2}]}R(s-u,B(s)-B(u))dsdu.$$
{For $B=[x_0,...,x_{N+1}]$ we have 
\begin{equation}\label{Markovprop1}
    B(\tau_k+s)=B(\tau_k)+x_k(s)\textit{, }s\in[0,1],
\end{equation}
and 
\begin{equation}\label{Markovprop2}
    B(\tau_{k+1}+s)=B(\tau_k)+x_k(1)+x_{k+1}(s)\textit{, }s\in[0,1].
\end{equation}
Therefore, from \eqref{Markovprop1}, and the change of variables $(s,u)\rightarrow(\tau_k+s,\tau_k+u)$
$$\int_{[\tau_k,\tau_{k+1}]}\int_{[\tau_k,\tau_{k+1}]}R(s-u,B(s)-B(u))dsdu=\int_{[0,1]^2}R(s-u,x_k(s)-x_k(u))dsdu,$$
for $k=1,...,N$. From \eqref{Markovprop2}, and the change of variables $(s,u)\rightarrow(\tau_{k+1}+s,\tau_{k}+u)$, we also get
$$\int_{[\tau_{k+1},\tau_{k+2}]}R(s-u,B(s)-B(u))dsdu=\int_{[0,1]^2}R(s-u,x_{k+1}(s)+x_k(1)-x_k(u))dsdu,$$
for $k=1,...,N$. Similar identities hold when $k=0$ and when $k=N+1$.} \par
For $\beta_{\epsilon}=\hat{\beta}/\sqrt{\log\frac{1}{\epsilon}}$ the above identities lead us to the following definitions:
$$D^{(\epsilon)}_{\tau}(x,y):=\beta_{\epsilon}^2\int_{[0,\tau]^2}R(s-u,y(s)-x(u))dsdu,\textbf{ }x,y\in\Omega_{\tau},$$
\begin{equation}\label{selfinteration}
    D^{(\epsilon)}(x,y):=\beta_{\epsilon}^2\int_{[0,1]^2}R(s-u,y(s)-x(u))dsdu, \textbf{ }x,y\in\Omega_{1},
\end{equation}
and
$$I^{(\epsilon)}_{0,1}(x,y):=\beta_{\epsilon}^2\int_{[0,\tau]}\int_{[0,1]}R(s-u,y(s)+x(1)-x(u))dsdu,\textbf{ }x\in\Omega_{\tau},\textbf{ }y\in\Omega_1,$$
\begin{equation}\label{Iedge}
    I^{(\epsilon)}_{N,N+1}(x,y):=\beta_{\epsilon}^2\int_{[0,1]}\int_{[0,T-\tau-N]}R(s-u,y(s)+x(1)-x(u))dsdu,\textbf{ }x\in\Omega_{1},\textbf{ }y\in\Omega_{T-\tau-N},
\end{equation}
\begin{equation}\label{iotepdef}
    I^{(\epsilon)}(x,y):=\beta_{\epsilon}^2\int_{[0,1]^2}R(s-u,y(s)+x(1)-x(u))dsdu, \textbf{ }x,y\in\Omega_{1}.
\end{equation}
The terms $D^{(\epsilon)}_{\tau},D^{(\epsilon)}$ capture the self-interactions of each segment and the terms $I^{(\epsilon)}_{0,1},I^{(\epsilon)},I^{(\epsilon)}_{N,N+1}$ capture the interactions between neighboring segments. By our previous observations, and by writing $B=[x_0,...,x_{N+1}]$, we have that
    $$\beta_{\epsilon}^2\int_{[0,T]^2}R(s-u,B(s)-B(u))dsdu=$$
    $$D^{(\epsilon)}_{\tau}(x_0,x_0)+\sum_{k=1}^{N+1}D^{(\epsilon)}(x_k,x_k)+2\biggr(I^{(\epsilon)}_{0,1}(x_0,x_1)+\sum_{k=1}^{N-1}I^{(\epsilon)}(x_k,x_{k+1})+I^{(\epsilon)}_{N,N+1}(x_N,x_{N+1})\biggl).$$
It is straightforward to check that $B=[x_0,...,x_{N+1}]$, where $x_0,...,x_{N+1}$ are distributed like standard Brownian motions, is distributed like a Brownian motion in $\Omega_T$. By  substituting the above formula to the definition of $\hat{\P}^{(\epsilon)}_T(dB)$, and using this observation the measure $\hat{\P}^{(\epsilon)}_T(dB)$ is proportional to
$$\hat{\P}^{(\epsilon)}_{\tau}(dx_0)e^{I^{(\epsilon)}_{0,1}(x_0,x_1)}\prod_{k=1}^{N-1}\hat{\P}^{(\epsilon)}_1(dx_k)e^{I^{(\epsilon)}(x_k,x_{k+1})}\hat{\P}_1(dx_N)e^{I^{(\epsilon)}_{T-\tau-N}(x_N,x_{N+1})}\hat{\P}^{(\epsilon)}_{T-\tau-N}(dx_{N+1}).$$
We want to interpret this as a transition probability of a  Markov chain. For this reason, we use the he Doob-Krein-Rutman theorem \cite{sneddon2012mathematical}, which implies that we can find  $\rho^{(\epsilon)}>0$ and $\Psi^{(\epsilon)}$ on $\Omega_1$ solving the following eigenvalue problem:
\begin{equation}\label{eigenv}
    \int_{\Omega_1}e^{I^{(\epsilon)}(x,y)}\Psi^{(\epsilon)}(y)\hat{\P}^{(\epsilon)}_1(dy)=\rho^{(\epsilon)}\Psi^{(\epsilon)}(x),
\end{equation}
so that  $\rho^{(\epsilon)}$ is the largest possible eigenvalue, $\Psi^{(\epsilon)}$  is bounded above and below by positive constants, is normalized to have total mass $1$ with respect to $\Hat{\P}^{(\epsilon)}_1$ and is the unique eigenvector associated with $\rho^{(\epsilon)}$. More specifically, we have the bound
\begin{align}\label{eq:eigenv_bound}
    e^{-||I^{(\epsilon)}||_{\infty}}\leq\rho^{(\epsilon)}\leq e^{||I^{(\epsilon)}||_{\infty}},
\end{align}
which we get by integrating the above eigenvalue equation over $y$. This implies that
\begin{equation}\label{eq:eigenvec_bound}
    e^{-2||I^{(\epsilon)}||_{\infty}}\leq\Psi^{(\epsilon)}(x)\leq e^{2||I^{(\epsilon)}||_{\infty}}.
\end{equation}
Now, we can define the transition probabilities
\begin{equation}\label{transition}
    \hat{\pi}^{(\epsilon)}(x,dy):=\frac{e^{I^{(\epsilon)}(x,y)}\Psi^{(\epsilon)}(y)\hat{\P}^{(\epsilon)}_1(dy)}{\rho^{(\epsilon)}\Psi^{(\epsilon)}(x)},
\end{equation}
\begin{equation}\label{finaledgetran}
    \hat{\pi}^{(\epsilon)}_{N,N+1}(x,dy):=\frac{e^{I^{(\epsilon)}_{N,N+1}(x,y)}\hat{\P}^{(\epsilon)}_{T-1-N}(dy)}{f^{(\epsilon)}_{N,N+1}(x)},
\end{equation}
\begin{equation}\label{firstedgetran}
   \hat{\pi}^{(\epsilon)}_{0,1}(x,dy):=\frac{e^{I^{(\epsilon)}_{0,1}(x,y)}\hat{\P}^{(\epsilon)}_{\tau}(dy)}{f^{(\epsilon)}_{0,1}(x)},
\end{equation}
where $f^{(\epsilon)}_{N,N+1}(x),f^{(\epsilon)}_{0,1}(x)$ are the normalization constants. With this notation the measure $\hat{\P}^{(\epsilon)}_{T}$ is proportional to
$$f^{(\epsilon)}_{0,1}(x_0)\hat{\P}^{(\epsilon)}_{\tau}(dx_0)\hat{\pi}^{(\epsilon)}_{0,1}(x_0,dx_1)\prod_{k=1}^{N-1}\hat{\pi}^{(\epsilon)}(x_k,dx_{k+1})\hat{\pi}^{(\epsilon)}_{N,N+1}(x_N,dx_{N+1})\frac{f^{(\epsilon)}_{N,N+1}(x_N)}{\Psi^{(\epsilon)}(x_N)}.$$
The Markov chain we are looking for is built from the probability kernels that appear above. In particular, we sample:
\begin{itemize}
    \item $X_0\in\Omega_{\tau}$ according to the distribution $(\Zcal_{\tau}^{(\epsilon)})^{-1}f^{(\epsilon)}_{0,1}(X_0)\hat{\P}^{(\epsilon)}_{\tau}(dX_0)$, where $\Zcal_{\tau}^{(\epsilon)}=\int_{\Omega_{\tau}}f^{(\epsilon)}_{0,1}(x)\hat{\P}^{(\epsilon)}_{\tau}(dx)$.
    \item $(X_1,...,X_N)$ according to $\hat{\pi}^{(\epsilon)}_{0,1}(X_0,dX_1)\biggl(\prod_{k=1}^{N-1}\hat{\pi}^{(\epsilon)}(X_k,dX_{k+1})\biggr)\hat{\pi}^{(\epsilon)}_{N,N+1}(X_N,dX_{N+1}).$
\end{itemize}
After sampling these points we can construct a path $B=[X_0,..., X_{N+1}]$ on $\Omega_{T}$ by gluing these paths together according to \eqref{eq3.1}. For any measurable function $F:\Omega_T\rightarrow\mathbb{R}$, we then have
$$\Hat{\E}_{B,T}[F(B)]=\E\biggl[F([X_0,...,X_{N+1}])c^{(\epsilon)}_{\tau,T}\frac{f^{(\epsilon)}_{N,N+1}(X_N)}{\Psi^{(\epsilon)}(X_N)}\biggr],$$
where $c^{(\epsilon)}_{\tau,T}$ is the appropriate normalization constant, {determined by the equation
\begin{equation}\label{normterm}
    \frac{1}{c_{\tau,T}}=\E\biggl[\frac{f^{(\epsilon)}_{N,N+1}(X_N)}{\Psi^{(\epsilon)}(X_N)}\biggr].
\end{equation}}
The term
$$c^{(\epsilon)}_{\tau,T}\frac{f^{(\epsilon)}_{N,N+1}(X_N)}{\Psi^{(\epsilon)}(X_N)}$$
is the term accounting for the edge effects born from splitting the interval $[0, T]$ into an interval of length $\tau$ and then intervals of length 1. As we will see, this term does not play a major role in our calculations.\par
We will mainly have to deal with expressions of the form
$$\hat{\E}_{B,T}[F(B^1,B^2)],$$
where $B^1,B^2$ are two independent paths sampled according to (\ref{expmeasure}). By following the same construction as above we can convert this expectation into the following expectation:
$$\E\biggr[F([X_0,...,X_{N+1}],[Y_0,...,Y_{N+1}])c_{\tau,T}^2\frac{f^{(\epsilon)}_{N,N+1}(X_N)}{\Psi^{(\epsilon)}(X_N)}\frac{f^{(\epsilon)}_{N,N+1}(Y_N)}{\Psi^{(\epsilon)}(Y_N)}\biggl],$$
where $[X_0,...,X_{N+1}]$ and $[Y_0,...,Y_{N+1}]$  are paths built from two independent trajectories $(X_i)_{i=1,...,N+1}$, $(Y_i)_{i=1,...,N+1}$ respectively, sampled as described above. Let
\begin{equation}\label{eqedge}
\Gcal_{\epsilon}(x,y):=c_{\tau,T}^2\frac{f^{(\epsilon)}_{N,N+1}(x)}{\Psi^{(\epsilon)}(x)}\frac{f^{(\epsilon)}_{N,N+1}(y)}{\Psi^{(\epsilon)}(y)}    
\end{equation}
be the term accounting for the 'edge effects'. Then, we have the identity:
\begin{equation}\label{Markovcon}
    \hat{\E}_{B,T}[F(B^1,B^2)]=\E\biggr[F([X_0,...,X_{N+1}],[Y_0,...,Y_{N+1}])\Gcal_{\epsilon}(X_N,Y_N)\biggl].
\end{equation}

\subsection{The coupling in $d=2$}\label{sec:4}
In the remainder of the paper the Markov chain on $\Omega_1\times\Omega_1$ with transition probability kernel $$\hat{\boldsymbol{{\pi}}}^{(\epsilon)}:=\hat{\pi}^{(\epsilon)}\times\hat{\pi}^{(\epsilon)}$$
plays an important role. We dedicate the next few sections to its properties. First, we will introduce a coupling of this Markov chain to a sequence of Bernoulli random variables based on Doeblin's inequality.\par
Recall the definition of the transition measure 
$$\hat{\pi}^{(\epsilon)}(x,dy)=\frac{e^{I^{(\epsilon)}(x,y)}\Psi^{(\epsilon)}(y)\hat{\P}^{(\epsilon)}_1(dy)}{\rho^{(\epsilon)}\Psi^{(\epsilon)}(x)},$$
and the definition of $\hat{\P}^{(\epsilon)}_1$
$$\hat{\P}^{(\epsilon)}_1(dx)=\exp\biggl(\frac{\hat{\beta}^2}{2\log \frac{1}{\epsilon}}\int_{[0,1]^2}R(s-u, x(s)-x(u))dsdu-\zeta^{(\epsilon)}_1\biggr)W_1(dx).$$
Observe that, as $\epsilon\rightarrow0$, the measure $\hat{\P}^{(\epsilon)}_1$ converges to $W_1$ in total variation. Moreover, recall the definition of $I^{(\epsilon)}(x,y)$:
$$I^{(\epsilon)}(x,y)=\beta_{\epsilon}^2\int_{[0,1]^2}R(s-u, x(1)+y(s)-x(u))dsdu,$$
We observe that as $\epsilon\rightarrow0$, $I^{(\epsilon)}(x,y)\rightarrow0$ uniformly over $x,y$. From \eqref{eq:eigenv_bound}, we get $\rho^{(\epsilon)}\rightarrow1$, and from \eqref{eq:eigenvec_bound} we get $\Psi^{(\epsilon)}(x)\rightarrow1$ uniformly in $x$. These observations imply that:
$$\frac{e^{I^{(\epsilon)}(x,y)}\Psi^{(\epsilon)}(y)}{\rho^{(\epsilon)}\Psi^{(\epsilon)}(x)}\rightarrow1,$$
uniformly over $x,y$. This proves the following proposition:

\begin{myprop}\label{thm:prop4.1}
    The transition probability kernel $\hat{\pi}^{(\epsilon)}(x,dy)$ converges in total variation to $W_1(dy)$ uniformly over $x$. In other words, we have:
    $$\sup_{x\in\Omega_1}d_{TV}(\hat{\pi}^{(\epsilon)}(x,dy),W_1(dy))\rightarrow0,$$
    as $\epsilon\rightarrow0$, where $d_{TV}$ is the total variation distance between two probability measures. Furthermore for any  $0<\gamma<1$, there is an $\epsilon_0=\epsilon_0(\gamma)$ such that for all $x\in\Omega_1$ and $A$ a Borel set in $\Omega_1$
    \begin{equation}\label{eq:4.1}
        \hat{\pi}^{(\epsilon)}(x,A)\geq\gamma W_1(A),
    \end{equation}
    for all $\epsilon<\epsilon_0$.
\end{myprop}
{\begin{remark}\label{Remedge}
    Let $\tau_\epsilon\in(0,1)$, $T=T(\epsilon)\rightarrow\infty$, as $\epsilon\rightarrow0$ and $N_{\epsilon}=[T(\epsilon)-\tau_\epsilon]$. Similar observations as above prove that $\Gcal_{\epsilon}$, defined in \eqref{eqedge}, converges to $1$ uniformly as $\epsilon\rightarrow0$. Indeed, we see that $I_{N,N+1}^{(\epsilon)}\rightarrow0$ uniformly as $\epsilon\rightarrow0$, where $I_{N_\epsilon,N_\epsilon+1}^{(\epsilon)}$ is defined in \eqref{Iedge}. By looking at \eqref{finaledgetran} we see that $f_{N_\epsilon,N_\epsilon+1}^{(\epsilon)}\rightarrow1$ uniformly as $\epsilon\rightarrow0$. This, along with our previous observations, prove that $c_{\tau,T}^{(\epsilon)}\rightarrow1$ as $\epsilon\rightarrow0$. These observations imply that $\Gcal_{\epsilon}\rightarrow1$, uniformly as $\epsilon\rightarrow0$. 
\end{remark}}
Since $\hat{\boldsymbol{\pi}}^{(\epsilon)}=\hat{\pi}^{(\epsilon)}\times\hat{\pi}^{(\epsilon)}$, this proposition implies that for any $\gamma\in(0,1)$, there is an $\epsilon_0$ small enough such that for all Borel sets $A\subset\Omega_1^2$, for all $z\in\Omega_1\times\Omega_1$ and all $\epsilon<\epsilon_0$ 
\begin{equation}\label{eq4.9}
    \hat{\boldsymbol{{\pi}}}^{(\epsilon)}(z,A)\geq\gamma (W_1\times W_1)(A).
\end{equation}
We can choose $\gamma$ to be such that both (\ref{eq:4.1}) and (\ref{eq4.9}) are true. From here on we will fix $\gamma$ and $\epsilon_0$ and we will always assume that $\epsilon<\epsilon_0$.\par
This last observation gives us a Doeblin condition for the Markov chain. Therefore,  we can introduce a coupling of the Markov chain with a sequence of i.i.d. Bernoulli random variables $(\eta_{j})_{j\in\mathbb{N}}$. This coupling  goes as follows:\par
We can write
$$\hat{\boldsymbol{{\pi}}}^{(\epsilon)}(z,dy)=\gamma(W_1\times W_1)(dy)+(1-\gamma)\frac{\hat{\boldsymbol{{\pi}}}^{(\epsilon)}(z,dy)-\gamma(W_1\times W_1)(dy)}{1-\gamma}.$$
This process establishes a regenerative structure for the Markov chain. More specifically, whenever a "success" ($\eta_j=1$) occurs, the chain forgets all prior steps and begins anew. This regenerative property allows us to decompose the trajectory of the Markov chain into a sequence of independent cycles.\par 
{Now let $((X_j,Y_j))_{j\geq1}\subseteq\Omega_1\times\Omega_1$ be a trajectory of $\boldsymbol{\hat{\pi}}^{(\epsilon)}$  with an initial segment $(X_0,Y_0)\sim W_1\times W_1$. We consider  the pair of paths $(\omega_{X_0}, \omega_{Y_0})$ in $C([0,\infty))^2$, each constructed by gluing $(X_j)_{j\geq0}$, $(Y_j)_{j\geq0}$ respectively, according to \eqref{eq3.1}. }We define the notion of the total path increment between two regeneration times. First, define  the associated regeneration times:
\begin{equation}\label{eq:4.4}
    T_0:=0\textbf{, }T_{i}:=\inf\{j>T_{i-1}:\eta_j=1\}.
\end{equation}
The times between two regenerations, $T_{k+1}-T_k$, $k=0,1,...$ are i.i.d. and distributed as geometric random variables with parameter $\gamma$. The corresponding total path increment for $\omega_{X_0}$ and $\omega_{Y_0}$ between two regenerations is defined by the equations:
\begin{equation}\label{pathincr}
    \mathbf{X}^{(\epsilon)}_j:=\sum_{k=T_j}^{T_{j+1}-1}X_k(1)\textit{ and }\mathbf{Y}^{(\epsilon)}_j:=\sum_{k=T_j}^{T_{j+1}-1}Y_k(1).
\end{equation}
By construction, the random variables $((\mathbf{X}_j^{(\epsilon)},\mathbf{Y}_j^{(\epsilon)}))_{j\geq0}$ are i.i.d.. Also, from \eqref{eq3.1}
\begin{equation}\label{Randomwalks}
    \omega_{X_0}(T_j)=\sum_{k=0}^{j-1}\mathbf{X}^{(\epsilon)}_{k} \textit{, }\omega_{Y_0}(T_j)=\sum_{k=0}^{j-1}\mathbf{Y}^{(\epsilon)}_{k},
\end{equation}
and
\begin{equation}\label{cycles}
    \mathcal{C}_j:=\biggl(T_{j+1}-T_j,(\omega_{X_0}(T_j+s)-\omega_{X_0}(T_{j}),\omega_{Y_0}(T_j+s)-\omega_{Y_0}(T_{j}))_{s\leq T_{j+1}-T_j}\biggr)
\end{equation}
is equal in distribution to 
\begin{equation}\label{eq4.10}
    \biggl(\theta,(\Tilde{\omega}_{B^1}(s),\Tilde{\omega}_{B^2}(s))_{s\leq \theta}\biggr),
\end{equation}
where:
\begin{itemize}
    \item $\theta\sim Geo(\gamma)$.
    \item $(B^1,B^2)\sim W_1\times W_1$ independent from $\theta$.
    \item $(\Tilde{\omega}_{B^1},\Tilde{\omega}_{B^2})$ is a pair of paths, each built according to \eqref{eq3.1} using segments sampled from $\frac{\hat{\boldsymbol{{\pi}}}^{(\epsilon)}-\gamma W_1\times W_1}{1-\gamma}$ with $(B^1,B^2)$ as initial steps. Also, $(\Tilde{\omega}_{B^1},\Tilde{\omega}_{B^2})$ is independent from $\theta$.
\end{itemize}
All of these random variables mentioned above are also independent from $(\omega_{X_0}(s),\omega_{Y_0}(s))_{s\leq T_j}$.\par
{The random variables $\mathcal{C}_j$ are called cycles. From our previous observations, the collection of all cycles $(\mathcal{C}_j)_{j\geq1}$ is a collection of i.i.d. random variables with \eqref{eq4.10} as the common underlying distribution. This proves that the process $(\omega_{X_0}(s), \omega_{Y_0}(s))_{s\geq0}$ can be decomposed into i.i.d. cycles. In other words,  the process
 $(\omega_{X_0}(s), \omega_{Y_0}(s))_{s\geq0}$ is a regenerative process in the sense of \cite{Rege_pro_Tal}. }

\subsection{The Mixing property}
Here, we detail the mixing mechanism of the Markov chain and some of its basic consequences. This is based on the observation that there are constants ${A}_{\epsilon}, {B}_{\epsilon}$, which converge to 1 as $\epsilon\rightarrow0$, and such that for all Borel sets of $\Omega_1$ and $x\in\Omega_1$,
\begin{equation}\label{eq:4.2}
 {A}_{\epsilon}W_1(A)\leq\hat{\pi}^{(\epsilon)}(x,A)\leq {B}_{\epsilon}W_1(A).   
\end{equation}
{Indeed, recall that from \eqref{transition}, $\hat{\pi}^{(\epsilon)}$ is absolutely continuous with respect to the Wiener measure
$$\hat{\pi}^{(\epsilon)}(x,dy)=\frac{e^{I^{(\epsilon)}(x,y)}\Psi^{(\epsilon)}(y)}{\rho^{(\epsilon)}\Psi^{(\epsilon)}(x)}\cdot e^{D_1^{(\epsilon)}(y,y)}W_1(dy),$$
with $D^{(\epsilon)}$ as in \eqref{selfinteration} and $I^{(\epsilon)}$ as in \eqref{iotepdef}. Since
$$e^{-||I^{(\epsilon)}||_{\infty}}\leq\rho^{(\epsilon)}\leq e^{||I^{(\epsilon)}||_{\infty}}.$$
and
$$e^{-2||I^{(\epsilon)}||_{\infty}}\leq\Psi^{(\epsilon)}(x)\leq e^{2||I^{(\epsilon)}||_{\infty}},$$
we see that there is a constant $C>0$ such that
$$e^{-C/\log\frac{1}{\epsilon}}W_1(A)\leq\hat{\pi}^{(\epsilon)}(x,A)\leq e^{C/\log\frac{1}{\epsilon}}W_1(A),$$
for all Borel sets $A\subseteq\Omega_1$ and $x\in\Omega_1$.}\par
{A similar inequality is true for $\boldsymbol{\hat{\pi}}^{(\epsilon)}$ and $\frac{\hat{\boldsymbol{{\pi}}}^{(\epsilon)}-\gamma W_1\times W_1}{1-\gamma}$ (only with different constants $\mathcal{A}_{\epsilon}, \mathcal{B}_{\epsilon}$) since both measures are absolutely continuous with respect to the Wiener measure, with density converging to $1$ uniformly as $\epsilon\rightarrow0$. With this observation, we can prove the following proposition:}
\begin{mylem}\label{lemm4.2}\label{thm:lem4.1}
    Let F be a nonegative integrable function on $(\Omega_1\times\Omega_1)^{p+1}$ and let $((X_0,Y_0),...,(X_p,Y_p))$ be a sequence of random variables generated by $\boldsymbol{\hat{\pi}}^{(\epsilon)}$ or $\frac{\hat{\boldsymbol{{\pi}}}^{(\epsilon)}-\gamma W_1\times W_1}{1-\gamma}$ with $(X_0,Y_0)\sim W_1\times W_1$. Then there are constants $\mathcal{A}_{\epsilon},\mathcal{B}_{\epsilon}$ such that they converge to $1$ as $\epsilon\rightarrow0$ and such that for any $1\leq k\leq m$
        $$\mathcal{A}_{\epsilon}\E\biggl[F((X_0,Y_0),...,(\Tilde{X}_k,\tilde{Y}_k),...,(X_p,Y_p))\biggr]$$ 
        $$\leq\E\biggl[F((X_0,Y_0),...,(X_p,Y_p))\biggr]\leq$$
        \begin{equation}\label{eq:4.3}
            \mathcal{B}_{\epsilon}\E\biggl[F((X_0,Y_0),...,(\Tilde{X}_k,\tilde{Y}_k),...,(X_p,Y_p))\biggr],
        \end{equation}
     where $(X_0,Y_0),...,(\Tilde{X}_k,\tilde{Y}_k),...,(X_p,Y_p)$ is generated  as before only at step $k$ we sample, $(\Tilde{X}_k,\tilde{Y}_k)\sim W_1\times W_1$ independently from all previous steps.    
    \end{mylem}
\begin{proof}
    First, assume that $F$ is bounded. We write $Z_k=(X_k,Y_k)$ for $k\geq0$. We prove the proposition in the case where $(Z_k)_k$ is generated using $\boldsymbol{\hat{\pi}}^{(\epsilon)}$. We prove this by writing the expectation of $F$ under the Markov chain explicitly:
    $$\E\biggl[F(Z_0,..., Z_k)\biggr]=\int_{\Omega_1}W_1\times W_1(dz_0)...\int_{\Omega_1}\boldsymbol{\hat{\pi}}^{(\epsilon)}(z_{k-1},dz_k)...\int_{\Omega_1}\boldsymbol{\hat{\pi}}^{(\epsilon)}(z_{m-1},dz_m)F(z_0,...,z_m).$$
    We aim to replace $\boldsymbol{\hat{\pi}}^{(\epsilon)}(z_{k-1},dz_k)$ by $W_1(dz_k)$:
    From \eqref{eq:4.2} there are constants $\mathcal{A}_{\epsilon}'$  and $\mathcal{B}_{\epsilon}'$ converging to $1$ as $\epsilon\rightarrow0$, such that 
    $$\mathcal{A}_{\epsilon}'W_1(A)\leq\hat{\boldsymbol{\pi}}^{(\epsilon)}(x,A)\leq \mathcal{B}_{\epsilon}'W_1(A),$$ 
    for all Borel sets $A\subseteq\Omega_1\times\Omega_1$. This inequality extends to all simple positive simple functions. In particular, there are constants $\mathcal{A}_{\epsilon},\mathcal{B}_{\epsilon}$ converging to $1$ as $\epsilon\rightarrow0$
    such that
    \begin{equation}\label{premixing}
        \mathcal{A}_{\epsilon}W_1\times W_1(f)\leq\boldsymbol{\hat{\pi}}^{(\epsilon)}(x,f)\leq \mathcal{B}_{\epsilon}W_1\times W_1(f),
    \end{equation}
    where, for a measure $\mu$ and a function $f$, $\mu(f)$ is the integral of $f$ with respect to $\mu$. By the definition of the Lebesgue integral, \eqref{premixing} extends to all positive measurable functions $f$.\par
    Therefore, from \eqref{premixing} we get the inequality:
    $$\E\biggl[F(Z_0,...,Z_m)\biggr]\leq \mathcal{B}_{\epsilon}\int_{\Omega_1}W_1\times W_1(dz_0)...\int_{\Omega_1}W_1(dz_k)...\int_{\Omega_1}\boldsymbol{\hat{\pi}}^{(\epsilon)}(z_{m-1},dz_m)F(z_0,...,z_m),$$
    which is the right-hand side of \eqref{eq:4.3}. Similarly, we prove the left-hand side as well. When $(Z_k)_{k\in\mathbb{N}_0}$ is generated by $\frac{\hat{\boldsymbol{{\pi}}}^{(\epsilon)}-\gamma W_1\times W_1}{1-\gamma}$, \eqref{eq:4.3} is proved in the same way.\par
    Once we have \eqref{eq:4.3} for all bounded measurable positive $F$ we may argue via a cutoff argument and monotone convergence to get the full result. 
\end{proof}
So when we have a positive functional of the Markov chain and we want to calculate the limit $\epsilon\rightarrow0$ of its expectation, we may replace a finite number of steps by standard Brownian motions that are independent of all previous steps. As $\epsilon\rightarrow0$  this will not change the limiting expectation since both upper and lower bounds will match.\newline\par

\begin{remark}\label{Mixingremark}
    \begin{itemize}
    \item Using this result we can prove the same double inequality (with different constants $\mathcal{A}_{\epsilon},\mathcal{B}_{\epsilon}$, still converging to 1, as $\epsilon\rightarrow0$) for a general $F\in L^1$ by splitting $F$ into its positive and negative parts.
    \item This lemma can also be applied in the case where $F$ is a function of a finite number of independent trajectories of $\boldsymbol{\hat{\pi}}^{(\epsilon)}$  (or trajectories generated by  $\frac{\hat{\boldsymbol{{\pi}}}^{(\epsilon)}-\gamma W_1\times W_1}{1-\gamma}$ with $(X_0, Y_0)\sim W_1\times W_1$) and we want to replace parts of these trajectories by Brownian motions. In this case we will get  bounds similar to $(\ref{eq:4.3})$ (again with a different set of constants $\mathcal{A}_{\epsilon},\mathcal{B}_{\epsilon}$,  converging to $1$, as $\epsilon\rightarrow0$).
    \item  Finally, this lemma generalizes in the case where we consider the edge effects and we sample the first and the final steps of the Markov chain using $\hat{\pi}_{0,1}\times\hat{\pi}_{0,1}$ and $\hat{\pi}_{p-1,p}\times\hat{\pi}_{p-1,p}$ respectively. Indeed, it is easy to see that $\hat{\pi}_{0,1}$ and $\hat{\pi}_{p-1,p}$ satisfy a version of $(\ref{eq:4.2})$ which allow us to replace $(X_1,Y_1)$ or $(X_p,Y_p)$ by Brownian motions and get the bounds presented above. If we want to replace any other step, the proof of \textbf{Lemma \ref{thm:lem4.1}} works. A similar argument works in the case where we start the Markov chain by  $\hat{\P}^{(\epsilon)}_{\tau}(dx_0)$ for some $\tau<1$ or the (normalized) measure $f^{(\epsilon)}_{0,1}(x_0)\hat{\P}^{(\epsilon)}_{\tau}(dx_0)$ and we want to replace the starting point by a Brownian motion.
\end{itemize} 
\end{remark} 
By conditioning on the regeneration length $T_{k+1}-T_k$, \textbf{Lemma \ref{thm:lem4.1}} has the following consequence for functionals of the total path increments:
\begin{mycor}
    There are constants $\mathcal{A}_{\epsilon},\mathcal{B}_{\epsilon}$, converging to $1$, as $\epsilon\rightarrow0$, such that for any non-negative $F\in L^1(\mathbb{R}^{2p})$: 
    $$\sum_{N_1=1,...,N_{p}=1}^{\infty}\biggl(\prod_{i=1}^{p}\mathcal{A}_{\epsilon}^{N_i}\gamma(1-\gamma)^{N_i-1}\biggr)\E\biggr[F((\mathbf{X}^{(N_1)},\mathbf{Y}^{(N_1)}),...,(\mathbf{X}^{(N_p)},\mathbf{Y}^{(N_p)}))\biggl]$$
$$\leq\E\biggr[F((\mathbf{X}^{(\epsilon)}_1,\mathbf{Y}^{(\epsilon)}_1),...,(\mathbf{X}^{(\epsilon)}_p,\mathbf{Y}^{(\epsilon)}_p))\biggl]\leq$$        
\begin{equation}\label{eq4.11}
    \sum_{N_1=1,...,N_{p}=1}^{\infty}\biggl(\prod_{i=1}^{p}\mathcal{B}_{\epsilon}^{N_i}\gamma(1-\gamma)^{N_i-1}\biggr)\E\biggr[F((\mathbf{X}^{(N_1)},\mathbf{Y}^{(N_1)}),...,(\mathbf{X}^{(N_p)},\mathbf{Y}^{(N_p)}))\biggl],
\end{equation}
where $(\mathbf{X}^{N_i})_{i=1,...,p}, (\mathbf{Y}^{N_i})_{i=1,...,p}$ are two independent sets of independent mean zero Gaussian random variables with covariance matrix $N_iI_{2\times2}$ respectively.
\end{mycor}

\begin{proof}
    {We prove this for $p=1$, with the general proof following the same lines, since $((\mathbf{X}^{(\epsilon)}_i,\mathbf{Y}^{(\epsilon)}_i))_{i\in\mathbb{N}}$ are i.i.d.. Recall the definition of $(\mathbf{X}^{(\epsilon)}_i,\mathbf{Y}^{(\epsilon)}_i)$ from \eqref{pathincr}. By construction, $(\mathbf{X}^{(\epsilon)}_i,\mathbf{Y}^{(\epsilon)}_i)$ is equal in distribution to 
    $$\biggl(\sum_{k=0}^{\theta-1}X_{k}(1),\sum_{k=0}^{\theta-1}Y_k(1)\biggr),$$
    where:
    \begin{itemize}
        \item $(X_0,Y_0)\sim W_1\times W_1(dx_0,dy_0)$ and $(X_{k+1},Y_{k+1})\sim\frac{\boldsymbol{\hat{\pi}}^{(\epsilon)}((X_k,Y_k),dx_{k+1},dy_{k+1})-\gamma W_1\times W_1(dx_{k+1},dy_{k+1})}{1-\gamma}$ for $k\in\mathbb{N}$.
        \item $\theta\sim Geo(\gamma)$, independent from $(X_{k},Y_k)_{k\in\mathbb{N}_0}$.
    \end{itemize}
    By conditioning on $\theta$, we have
    $$\E[F(\mathbf{X}^{(\epsilon)}_1,\mathbf{Y}^{(\epsilon)}_1)]=\sum_{N=1}\gamma(1-\gamma)^{N-1}\E\biggl[F\biggl(\sum_{k=0}^{N-1}X_{k}(1),\sum_{k=0}^{N-1}Y_k(1)\biggr)\biggr].$$
    From \eqref{eq:4.3}, and \textbf{Remark \ref{Mixingremark}}, applied to the function $\tilde{F}:(\Omega_1\times\Omega_1)^N\rightarrow\mathbb{R}$,
    $$\tilde{F}((x_1,y_1),...,(x_N,y_N))=F\biggl(\sum_{k=0}^{N-1}x_{k}(1),\sum_{k=0}^{N-1}y_k(1)\biggr),$$
     we have the upper bound
    $$\E[F(\mathbf{X}^{(\epsilon)}_1,\mathbf{Y}^{(\epsilon)}_1)]\leq\sum_{N=1}\gamma(1-\gamma)^{N-1}\mathcal{B}_{\epsilon}^N\E\biggl[F\biggl(\sum_{j=0}^{N-1}B_j^1(1),\sum_{j=0}^{N-1}B_j^2(1)\biggr)\biggr],$$
    where $(B_j^1,B_j^2)$ are pairs of i.i.d.  standard Brownian motions, and $\mathcal{B}_{\epsilon}\rightarrow1$ as $\epsilon\rightarrow0$. This proves the right-hand side of \eqref{eq4.11}. Similarly, we prove the left-hand side of \eqref{eq4.11} and this concludes the proof.}    
\end{proof}

\textbf{Proposition \ref{thm:lem4.1}} has the following corollary. It is the analog of \textbf{Lemma A.2} from \cite{Gu_2018}:

\begin{mylem}\label{thm:cor4.1}
    Let $(X_{k}, Y_k)_{k\in\mathbb{N}_0}$ be a sequence of random variables generated by the transition probability $\frac{\boldsymbol{\hat{\pi}}^{(\epsilon)}-\gamma W_1\times W_1}{1-\gamma}$, where $(X_{0}, Y_{0})\sim\hat{\P}_1^{(\epsilon)}(dx_0)\times\hat{\P}_1^{(\epsilon)}(dy_0)$. Then $\E[X_k(1)]=0$. Moreover, there is a $c_1>0$ such that
    \begin{equation}\label{subg}
        \P[\sup_{s\in[0,1]}|X_{k}(s)|>t]\lesssim e^{-c_1t^2}.
    \end{equation}
     Additionally, if $\theta\sim Geo(\gamma)$ is independent from $(X_{k})_{k\in\mathbb{N}}$, then for some $c_2>0$:
    \begin{equation}\label{subexp}
        \P\biggl[\sum_{k=1}^{\theta}\sup_{s\in[0,1]}|X_{k}(s)|>t\biggr]\lesssim e^{-c_2t}.
    \end{equation}
    Therefore, $\sum_{k=1}^{\theta}\sup_{s\in[0,1]}|X_{k}(s)|$ has exponential tails, uniformly in $\epsilon$. Finally
    \begin{equation}\label{anotherfestimate}
        \P[\sup_{s\leq\theta}|B(s)|>t]\lesssim e^{-c_3t},
    \end{equation}
    for some $c_3>0$, where $B\in C([0,\infty))$ is the path built from  $(X_{k}, Y_k)_{k\in\mathbb{N}_0}$.
\end{mylem}
\begin{remark}\label{Remarkfornikos}
    Observe that the marginal law of $(X_k)_{k\in\mathbb{N}_0}$ is the same as the law of $(Z_k)_{\mathbb{N}_0}$, where this sequence is generated by  $\frac{{\hat{\pi}}^{(\epsilon)}-\gamma W_1}{1-\gamma}$ with $Z_0\sim\hat{\P}_1^{(\epsilon)}(dz_0)$.
\end{remark}
\begin{proof}
    The proof of $\E[X_k(1)]=0$ is the same as in \cite{Gu_2018} (look at the proof of Lemma A.2).\par
    Inequality \eqref{subg} is a straightforward application of \textbf{Lemma \ref{thm:lem4.1}} with $F(x)=\sup_{s\in[0,1]}|x(s)|$ and the corresponding bound
     $$\P[\sup_{s\in[0,1]}|B(s)|>t]\lesssim e^{-c_1t^2},$$
     where $B$ is a Brownian motion.\par
     For \eqref{subexp}, we condition on $\theta$ and  apply \textbf{Lemma \ref{thm:lem4.1}} to $F(x)=\sum_{k=1}^{N}\sup_{s\in[0,1]}|x_i(s)|$. If $B^1,...,B^{N}$ are independent Brownian motions then
     $$\P\biggr[\sum_{k=1}^{\theta}\sup_{s\in[0,1]}|X_{k}(s)|>t\biggl]\lesssim\sum_{N=1}^{\infty}\gamma(1-\gamma)^{N-1}\P\biggr[\sum_{k=1}^{N}\sup_{s\in[0,1]}|B^k(s)|>t\biggl].$$
     We have that
     $$\P\biggr[\sum_{k=1}^{N}\sup_{s\in[0,1]}|B^k(s)|>t\biggl]\leq e^{-ct}\E\biggr[\exp\biggr(c\sum_{k=1}^{N}\sup_{s\in[0,1]}|B^k(s)|\biggl)\biggl].$$
     Therefore,
     $$\P\biggr[\sum_{k=1}^{N}\sup_{s\in[0,1]}|X_{k}(s)|>t\biggl]\lesssim\sum_{N=1}^{\infty}\gamma(1-\gamma)^{N-1}e^{-ct} C^{N},$$
     for  $C=\E[\exp(c\sup_{s\in[0,1]}|B^1(s)|)]$. We can conclude by taking $c$ small enough.\par
     Finally, \eqref{anotherfestimate} is proved similarly.
\end{proof}

\subsection{Intersections of independent paths}\label{sec:5}
For the proof of \textbf{Proposition \ref{thm:prop2.3}} we will need to know the limit of
\begin{equation}
    \E\biggl[\exp\biggl(\frac{\hat{\beta}^2}{\log\frac{1}{\epsilon}}\int_{[0,M(\epsilon)/\epsilon^2]^2}R(s-u,x+\omega_{X_0}(s)-\omega_{Y_0}(u))dsdu\biggr)\biggr],
\end{equation}
where $R$ is the correlation function of the space-time mollified white noise, $\omega_{X_0},\omega_{Y_0}$ are two paths built from $\boldsymbol{\hat{\pi}}^{(\epsilon)}$, with $(X_0,Y_0)\sim W_1\times W_1$ as an initial segment,  and $M(\epsilon)$ is such that $\log M(\epsilon)/\log\frac{1}{\epsilon}\rightarrow0$, as $\epsilon\rightarrow0$. The main theorem is the following (recall that $\hat{\beta}_c(R)=\sqrt{2\pi/||R|||_1}$):

\begin{mythm}\label{thm:5.1}
    Let $X_0,Y_0\in\Omega_1$, $(\omega_{X_0},\omega_{Y_0})$ and $M(\epsilon)$  be as above. Then, as $\epsilon\rightarrow0$,
    $$ \E\biggl[\exp\biggl(\frac{\hat{\beta}^2}{\log\frac{1}{\epsilon}}\int_{[0,M(\epsilon)/\epsilon^2]^2}R(s-u,x+\omega_{X_0}(s)-\omega_{Y_0}(u))dsdu\biggr)\biggr]\rightarrow\biggr(1-\frac{\hat{\beta}^2}{\hat{\beta}_c(R)^2}\biggl)^{-1},$$
    for all $x\in\mathbb{R}^2$ and for all $\hat{\beta}<\hat{\beta}_c(R)$.
\end{mythm}

The proof of \textbf{Theorem \ref{thm:5.1}} follows by establishing two intermediate estimates:
\begin{itemize}
    \item Using the fact that $(\omega_{X_0},\omega_{Y_0})$ is built by the Markov chain on $\Omega_1\times\Omega_1$, and  estimating the first moment of \eqref{eq:add_funct} we prove that
    \begin{equation}\label{eq:exponential_bound}
    \limsup_{\epsilon\rightarrow0}\E\biggl[\exp\biggl(\frac{\hat{\beta}^2}{\log\frac{1}{\epsilon}}\int_{[0,T_{[M(\epsilon)/\epsilon^2]}]^2}R(s-u,x+\omega_{X_0}(s)-\omega_{Y_0}(u))dsdu\biggr)\biggr]\leq\biggr(1-\frac{\hat{\beta}^2}{\hat{\beta}_c(R)^2}\biggl)^{-1},
    \end{equation}
    where we recall that $(T_{n})_{n\in\N}$ are the regeneration times of $(\omega_{X_0},\omega_{Y_0})$. This is done in \textbf{Proposition \ref{cor5.1}}.
    \item We use the regenerative structure of $(\omega_{X_0},\omega_{Y_0})$ and the previous exponential bound to prove 
    \begin{equation}\label{eq:mom_liminf}
        \liminf_{\epsilon\rightarrow0}\E\biggl[\biggl(\frac{\hat{\beta}^2}{\log\frac{1}{\epsilon}}\int_{[0,M(\epsilon)/\epsilon^2]^2}R(s-u,x+\omega_{X_0}(s)-\omega_{Y_0}(u))dsdu\biggr)^p\biggr]\geq\biggl(\frac{\hat{\beta}^2}{\hat{\beta}_c(R)^2}\biggr)^pp!,
    \end{equation}
    for all $p\in\N$.
\end{itemize}

\begin{proof}[Proof of \textbf{Theorem \ref{thm:5.1}}]
    Since for all $n\in\N$, $n\leq T_n$, the bound \eqref{eq:exponential_bound}, yields
    $$\limsup_{\epsilon\rightarrow0}\E\biggl[\exp\biggl(\frac{\hat{\beta}^2}{\log\frac{1}{\epsilon}}\int_{[0,M(\epsilon)/\epsilon^2]^2}R(s-u,x+\omega_{X_0}(s)-\omega_{Y_0}(u))dsdu\biggr)\biggr]\leq\biggr(1-\frac{\hat{\beta}^2}{\hat{\beta}_c(R)^2}\biggl)^{-1}.$$
    On the other hand, from \eqref{eq:mom_liminf}, we have
    $$\liminf_{\epsilon\rightarrow0}\E\biggl[\exp\biggl(\frac{\hat{\beta}^2}{\log\frac{1}{\epsilon}}\int_{[0,M(\epsilon)/\epsilon^2]^2}R(s-u,x+\omega_{X_0}(s)-\omega_{Y_0}(u))dsdu\biggr)\biggr]\geq\sum_{p=0}^M\biggl(\frac{\hat{\beta}^{2}}{\hat{\beta}_c(R)}\biggr)^p,$$
    for all $M\in\N$. These two observations prove the theorem.
\end{proof}

The rest of this section is dedicated to establishing \eqref{eq:exponential_bound} and \eqref{eq:mom_liminf}. To prove these two estimates, we need to control moments of 
\begin{equation}\label{eq:add_funct}
    \int_{[0,M(\epsilon)/\epsilon^2]^2}R(s-u,x+\omega_{X_0}(s)-\omega_{Y_0}(u))dsdu,
\end{equation}
and we do this using the regenerative structure of the paths. Let us explain the main idea. By splitting the integral in \eqref{eq:add_funct} over the regeneration times, and using the fact that $R(s,\cdot)=0$ for $|s|\geq1$, we write the pth moment of \eqref{eq:add_funct} as
$$\E\biggl[\biggl(\int_{[0,M(\epsilon)/\epsilon^2]^2}R(s-u,\omega_{X_0}(s)-\omega_{Y_0}(u))dsdu\biggr)^p\biggr]=\E\biggl[\biggl(\sum_{k=1}^{[M(\epsilon)/\epsilon^{2}]-1}d_{k,k-1}+d_{k,k}+d_{k,k+1}\biggr)^p\biggr]=$$
$$=p!\sum\limits_{1\leq k_1<k_2<...<k_p\leq [M(\epsilon)/\epsilon^{2}]-1}\E\biggl[\prod_{i=1}^p\biggl(d_{k_i,k_i-1}+d_{k_i,k_i}+d_{k_i,k_i+1}\biggr)\biggr]+\Ecal(p;\epsilon)=$$
\begin{equation}\label{eq5.6}
    =p!\sum_{\boldsymbol{\delta}\in\{-1,0,+1\}^p}\sum\limits_{\substack{1\leq k_1<k_2<...<k_p\leq [M(\epsilon)/\epsilon^{2}]-1}}\E\biggl[\prod_{i=1}^p{d}_{k_i,k_i+\delta_i}\biggr]+\Ecal(p;\epsilon),
\end{equation}
where $\Ecal(p;\epsilon)$ is an  error term, and for $\delta\in\{-1,0,1\}$, we defined
\begin{equation}\label{dkterms}
    d_{k,k+\delta}:=\int_{T_{k}}^{T_{k+1}}\int_{T_{k+\delta}}^{T_{k+\delta+1}}R(s-u,\textit{ }x+\omega_{X_0}(s)-\omega_{Y_0}(u))dsdu,
\end{equation}
with $d_{1,0}=d_{[M(\epsilon)/\epsilon^2],[M(\epsilon)/\epsilon^2]+1}=d_{[M(\epsilon)/\epsilon^2]-1,[M(\epsilon)/\epsilon^2]+1}=0$. By the change of variables $(s,u)\rightarrow(s+T_{k_i+\delta_i}, u+T_{k_i})$ and by adding and subtracting the terms $\omega_{X_0}(T_{k_i+\delta_i})$ and $\omega_{Y_0}(T_{k_i})$, $d_{k_i,k_i+\delta_i}$ is equal to:
\begin{align*}
    \int_{0}^{T_{k_i+1}-T_{k_i}}\int_{0}^{T_{k_i+\delta_i+1}-T_{k_i+\delta_i}}R\biggr(s-u-(T_{k_i}-T_{k_i+\delta_i}),\textit{ }x+\omega_{X_0}(T_{k_i+\delta_i})-\omega_{Y_0}(T_{k_i})+\\
    \omega_{X_0}(s+T_{k_i+\delta_i})-\omega_{X_0}(T_{k_i+\delta_i})-(\omega_{Y_0}(u+T_{k_i})-\omega_{Y_0}(T_{k_i}))\biggl)dsdu.
\end{align*}
We focus on  
\begin{equation}\label{firstdominantterm}
    \sum_{\boldsymbol{\delta}\in\{-1,0,+1\}^p}\sum\limits_{\substack{1\leq k_1<k_2<...<k_p\leq [M(\epsilon)/\epsilon^{2}]-1}}\E\biggl[\prod_{i=1}^p{d}_{k_i,k_i+\delta_i}\biggr].
\end{equation}
Recall that, from \textbf{Section \ref{sec:4}}, the process $(\omega_{X_0},\omega_{Y_0})$ (and therefore the processes $\omega_{X_0}$ and $\omega_{Y_0}$, considered separately) can be split into i.i.d. 'cycles'
\begin{equation}\label{thecycles}
    (\mathcal{C}_j)_{j\geq0}=\biggl(T_{j+1}-T_j,(\omega_{X_0}(T_j+s)-\omega_{X_0}(T_{j}),\omega_{Y_0}(T_j+s)-\omega_{Y_0}(T_{j}))_{s\leq T_{j+1}-T_j}\biggr)_{j\geq0}
\end{equation}
Observe that $d_{k_i,k_i}$ is a functional of the cycle $\mathcal{C}_{k_i}$ and of the point that the cycle initiates i.e. of $\omega_{X_0}(T_{k_i})$ and $\omega_{Y_0}(T_{k_i})$. From the regenerative structure of $(\omega_{X_0},\omega_{Y_0})$, we see that $\mathcal{C}_{k_i}$ is independent from
$$(\omega_{X_0}(T_{k_i}),\omega_{Y_0}(T_{k_i}))$$
Therefore, we can write the expectation of $d_{k_i,k_i}$ as
$$\frac{1}{\log\frac{1}{\epsilon}}\E[d_{k_i,k_i}]=\frac{1}{\log\frac{1}{\epsilon}}\E[\Gcal_{\epsilon}(\omega_{X_0}(T_{k_i})-\omega_{Y_0}(T_{k_i}))],$$
with 
$$\Gcal_\epsilon(z)=\E_{cycle\textbf{ }law}\biggl[\int_0^\theta\int_{0}^\theta R(s-u,\textit{ }x+z+\Tilde{\omega}_{B^1}(s)-\Tilde{\omega}_{B^2}(u))dsdu\biggr],$$
where $(\theta,(\Tilde{\omega}_{B^1}(s),\Tilde{\omega}_{B^2}(s))_{s\leq \theta})$ is as in \eqref{eq4.10}. So, in theory, if we want to control \eqref{firstdominantterm}, we can average over the law of the cycles first. This will give us a functional of $(\omega_{X_0}(T_{k_i})-\omega_{Y_0}(T_{k_i}))_{i=1,...,p}$. From \eqref{Randomwalks} these random variables are sums of i.i.d. random variables, which guarantees their 'good behavior' (see \textbf{Proposition \ref{prop4.4}}). Then the resulting additive functional of $(\omega_{X_0}(T_{k_i})-\omega_{Y_0}(T_{k_i}))_{i=1,...,p}$ can be handled by standard results from \cite{Kallianpur1954TheSO}.\par
{There are two obstacles to implementing this strategy:
\begin{enumerate}[label=\textbf{P.\arabic*}]
    \item\label{Problem1} First, it is obvious that we have to consider the terms $d_{k_i,k_i+\delta_i}$, for $\delta_i\in\{-1,0,1\}$. When $\delta_i\neq0$, the term $d_{k_i,k_i+\delta_i}$ compares the $k_i+\delta_i-$th cycle of $\omega_{X_0}$ to the $k_i-$th cycle of $\omega_{Y_0}$. So when we average over the law of these two cycles, we will get a function of $\omega_{X_0}(T_{k_i+\delta_i})-\omega_{Y_0}(T_{k_i})$. This will lead us to a 'non-directed' additive functional of $(\omega_{X_0}(T_{k_i})-\omega_{Y_0}(T_{k_i}))_{i=1,...,p}$ (see \eqref{nondirfun} below). Nevertheless, this can be handled with \textbf{Lemma \ref{lemm4.3}}.
    \item\label{Problem2} Secondly, we cannot average over all cycles in the product 
    \begin{equation}\label{productterm}
        \E\biggl[\prod_{i=1}^p{d}_{k_i,k_i+\delta_i}\biggr]
    \end{equation}
    to get an expectation of a function of $(\omega_{X_0}(T_{k_i+\delta_i})-\omega_{Y_0}(T_{k_i}))_{i=1,...,p}$. The reason for this becomes apparent even in the case of $p=2$, for the expectation
    \begin{equation*}
        \E[d_{k_1,k_1}d_{k_2,k_2}],
    \end{equation*}
    with $k_1<k_2$. As mentioned, $d_{k_i,k_i}$, is a function of $\mathcal{C}_{k_i}$
    and of $\omega_{X_0}(T_{k_i})-\omega_{Y_0}(T_{k_i}).$ 
    However, $\omega_{X_0}(T_{k_2})-\omega_{Y_0}(T_{k_2})$ is not independent from $\mathcal{C}_{k_1}$. Indeed, consider the paths $(\omega_{X_0}(T_{k_i}+s)-\omega_{X_0}(T_{k_i}),\omega_{Y_0}(T_{k_i}+u)-\omega_{Y_0}(T_{k_i}))_{s,u\in[0, T_{k_i+1}-T_{k_i}]}$. Then for $s=u=T_{k_i+1}-T_{k_i}$, these paths take the value:
    $$(\omega_{X_0}(T_{k_1+1})-\omega_{X_0}(T_{k_1}),\omega_{Y_0}(T_{k_1+1})-\omega_{Y_0}(T_{k_1})).$$
    From \eqref{Randomwalks}, this is equal to $(\mathbf{X}^{(\epsilon)}_{k_1},\mathbf{Y}^{(\epsilon)}_{k_1})$. Since $k_2>k_1$,
    $$\mathbf{X}^{(\epsilon)}_{k_1}-\mathbf{Y}^{(\epsilon)}_{k_1}$$ is one of the summands of $\omega_{X_0}(T_{k_2})-\omega_{Y_0}(T_{k_2})$.
    This problem becomes more complicated in the general case. In particular,  the product in  \eqref{productterm} is a function of 
    \begin{align}\label{boldzetaterm}
    \mathbf{Z}_{k_i}:=\biggr(T_{k_i+1}-T_{k_i},T_{k_i+\delta_i+1}-T_{k_i+\delta_i},(\omega_{X_0}(s+T_{k_i+\delta_i})-\omega_{X_0}(T_{k_i+\delta_i}))_{s\leq T_{k_i+\delta_i+1}-T_{k_i+\delta_i}},\nonumber \\
    (\omega_{Y_0}(u+T_{k_i})-\omega_{Y_0}(T_{k_i}))_{u\leq T_{k_i+1}-T_{k_i}}\biggl),
\end{align}
    for $i=1,...,p$ and of 
    \begin{equation}\label{eqforinvrw}
        (\omega_{X_0}(T_{k_1+\delta_1})-\omega_{Y_0}(T_{k_1}),...,\omega_{X_0}(T_{k_p+\delta_p})-\omega_{Y_0}(T_{k_p})).
    \end{equation}
    For the same reasons as in the previous example,  the cycles 
    $$\biggl(T_{k_i+\delta_i+1}-T_{k_i+\delta_i}, (\omega_{X_0}(s+T_{k_i+\delta_i})-\omega_{X_0}(T_{k_i+\delta_i}), \omega_{Y_0}(s+T_{k_i+\delta_i})-\omega_{Y_0}(T_{k_i+\delta_i}))_{s\in[0, T_{k_i+\delta_i+1}-T_{k_i+\delta_i}]}\biggr)$$ 
    and 
    $$\biggl(T_{k_i+1}-T_{k_i}, (\omega_{X_0}(s+T_{k_i})-\omega_{X_0}(T_{k_i}), \omega_{Y_0}(s+T_{k_i})-\omega_{Y_0}(T_{k_i}))_{s\in[0, T_{k_i+1}-T_{k_i}]}\biggr),$$
    are not independent from $\omega_{X_0}(T_{k_j+\delta_j})-\omega_{Y_0}(T_{k_j})$ for $i<j$. However, the latter depends on these cycles only through $\mathbf{X}^{(\epsilon)}_{k_i}-\mathbf{Y}^{(\epsilon)}_{k_i}$ and $\mathbf{X}^{(\epsilon)}_{k_i+\delta_i}-\mathbf{Y}^{(\epsilon)}_{k_i+\delta_i}$.
\end{enumerate}
Now, we explain how we deal with item \ref{Problem2}. The idea is to write \eqref{eqforinvrw} as $\mathcal{S}^{\epsilon}+\mathcal{Y}^\epsilon$, for appropriate, independent random variables, so that $\mathcal{S}^{\epsilon}$ is independent from the cycles \eqref{thecycles} and equal in distribution to a shifted version of \eqref{eqforinvrw}. This is accomplished by 'throwing out' the terms $\mathbf{X}^{(\epsilon)}_{k_i}-\mathbf{Y}^{(\epsilon)}_{k_i}$ and $\mathbf{X}^{(\epsilon)}_{k_i+\delta_1}-\mathbf{Y}^{(\epsilon)}_{k_i+\delta_i}$ from the sum \eqref{Randomwalks}, defining $\omega_{X_0}(T_{k_j+\delta_j})-\omega_{Y_0}(T_{k_j})$, $1\leq i\leq j\leq p$.  With such an expression in hand, and by averaging over the joint law of $\mathcal{Y}^\epsilon$ and of \eqref{boldzetaterm}, we can write  \eqref{productterm} as an expectation of a function of a shifted version of \eqref{eqforinvrw}. Then, as mentioned, the resulting additive functional can be handled with \textbf{Lemma \ref{lemm4.3}}. Because of this shifting, it is necessary to consider \eqref{firstdominantterm} with the restriction $k_{i+1}>k_i+3$. In particular, we have the following lemma, which is sufficient for our purposes:

\begin{mylem}\label{techlemma}
    Let $(X_0,Y_0)\sim W_1\times W_1$ and $(\omega_{X_0},\omega_{Y_0})$ be the path built from $\boldsymbol{\hat{\pi}}^{(\epsilon)}$ with $(X_0,Y_0)$ as the initial segment.  Then for any $M(\epsilon)$, as in \textbf{Theorem \ref{thm:5.1}} we have
    \begin{equation}\label{eq5.8}
    \biggl(\frac{2\pi}{||R||_1\log\frac{1}{\epsilon}}\biggr)^p\sum_{\mathbf{\delta}\in\{-1,0,+1\}^p}\sum\limits_{\substack{3\leq k_1<k_2<...<k_p\leq [M(\epsilon)/\epsilon^2]-1,\\ k_{i+1}>k_i+3}}\mathbb{E}\biggl[\prod_{i=1}^p{d}_{k_i,k_i+\delta_i}\biggr]\rightarrow1,
\end{equation}
for any $p\in\N$, as $\epsilon\rightarrow0$.
\end{mylem}

\begin{proof}
We follow the outline explained above. In particular, for $i=2,...,p$, we write:
$$\omega_{Y_0}(T_{k_i})=\bar{\mathbf{S}}_{k_i}+\sum_{j=1}^{i-1}(\mathbf{Y}^{(\epsilon)}_{k_{j}}+\mathbf{Y}^{(\epsilon)}_{k_{j}+\delta_j})+\mathbf{Y}^{(\epsilon)}_{k_{i}-1}.$$
Here, we have 'thrown out' the variables $\mathbf{Y}^{(\epsilon)}_{k_{j}-1},\mathbf{Y}^{(\epsilon)}_{k_{j}}, \mathbf{Y}^{(\epsilon)}_{k_{j}+1}$ to account for every possible value of $\delta_j$.  We included $\mathbf{Y}^{(\epsilon)}_{k_{i}-1}$ as well since, in the case where $\delta_i=-1$, $\omega_{Y_0}(T_{k_i})$ also depends on $T_{k_i}-T_{k_i-1}$ (which appears in $d_{k_i,k_i-1}$).  Similarly, we write:
$$\omega_{X_0}(T_{k_i+\delta_i})={\mathbf{S}}_{k_i+\delta_i}+\sum_{j=1}^{i-1}(\mathbf{X}^{(\epsilon)}_{k_{j}}+\mathbf{X}^{(\epsilon)}_{k_{j}+\delta_j})+\mathbf{X}^{(\epsilon)}_{k_{i}-1},$$
where $\mathbf{S}_{k_i},\bar{\mathbf{S}}_{k_i}$ are defined by the equations above and are independent from $\mathbf{X}^{(\epsilon)}_{k_{j}-1}$, $\mathbf{X}^{(\epsilon)}_{k_{j}}$, $\mathbf{X}^{(\epsilon)}_{k_{j}+1}$, $\mathbf{Y}^{(\epsilon)}_{k_{j}-1}$, $\mathbf{Y}^{(\epsilon)}_{k_{j}}$, $\mathbf{Y}^{(\epsilon)}_{k_{j}+1}$, $j=0,...,i-1$ and from $\mathbf{X}^{(\epsilon)}_{k_{i}-1},\mathbf{Y}^{(\epsilon)}_{k_{i}-1}$. Observe that $\mathbf{S}_{k_i+\delta_i}-\bar{\mathbf{S}}_{k_i}$ is equal in distribution to\footnote{Since $k_{i+1}>k_{i}+3$ for all $i=1,...,p$, with $k_1\geq2$, we have $k_{i}>3(i-1)+1$.} 
$$\omega_{X_0}(T_{k_i+\delta_i-3(i-1)-1})-\omega_{Y_0}(T_{k_i-3(i-1)-1}),$$ 
and furthermore, the vector:
$$({\mathbf{S}}_{k_1+\delta_1}-\bar{\mathbf{S}}_{k_1},...,{\mathbf{S}}_{k_p+\delta_p}-\bar{\mathbf{S}}_{k_p}),$$
is independent from $(\mathbf{Z}_{k_l})_{l=1,...,p}$, defined in \eqref{boldzetaterm}, and it is equal in distribution to the vector:
$$(\omega_{X_0}(T_{k_1-1+\delta_1})-\omega_{Y_0}(T_{k_1-1}),\dots,\omega_{X_0}(T_{k_p+\delta_p-3(p-1)-1})-\omega_{Y_0}(T_{k_p-3(p-1)-1})).$$
These observations show that we can write
\begin{align}\label{pulledoutformula}
    d_{k_i,k_i+\delta_i}=\int_{0}^{T_{k_i+1}-T_{k_i}}\int_{0}^{T_{k_i+\delta_i+1}-T_{k_i+\delta_i}}R\biggr(s-u-(T_{k_i}-T_{k_i+\delta_i}),\quad \mathbf{S}_{k_i+\delta_i}-\bar{\mathbf{S}}_{k_i}+\mathbf{Y}(i)+
    \nonumber\\ +\omega_{X_0}(s+T_{k_i+\delta_i})-\omega_{X_0}(T_{k_i+\delta_i})-(\omega_{Y_0}(u+T_{k_i})-\omega
    _{Y_0}(T_{k_i}))\biggl)dsdu,
\end{align}
where 
\begin{equation}\label{wtf0}
    \mathbf{Y}(i)=\sum_{j=1}^{i-1}\sum_{\delta\in\{-1,0,1\}}(\mathbf{X}^{(\epsilon)}_{k_{j}+\delta}-\mathbf{Y}^{(\epsilon)}_{k_{j}+\delta})+\mathbf{X}^{(\epsilon)}_{k_{i}-1}-\mathbf{Y}^{(\epsilon)}_{k_{i}-1}.
\end{equation}
From \eqref{pulledoutformula}, we observe that
\[\prod_{i=1}^pd_{k_i,k_i+\delta_i}\]
is a function of 
\begin{equation}\label{variablevector}
    ({\mathbf{S}}_{k_1+\delta_1}-\bar{\mathbf{S}}_{k_1},...,{\mathbf{S}}_{k_p+\delta_p}-\bar{\mathbf{S}}_{k_p})
\end{equation}
and of 
\begin{equation}\label{eq5.9}
    (\mathbf{Z}_{k_i};\mathbf{Y}(i))_{i=1,...p}.
\end{equation}
We proved that \eqref{variablevector} and \eqref{eq5.9} are independent. Therefore, if we prove that the distribution of \eqref{eq5.9} is independent of $k_1,...,k_p$, in \eqref{eq5.8}, we can average over the law of \eqref{eq5.9} first and get a functional with respect to \eqref{variablevector}. Then, \textbf{Lemma \ref{lemm4.3}} will give us the limit of \eqref{eq5.8}.\par 
Looking at  \eqref{eq5.9}, we see that it depends on the cycles $((\mathcal{C}_{k_i-1},\mathcal{C}_{k_i},\mathcal{C}_{k_i+1}))_{i=1,...,p}$. Since we have the restriction of $k_{i+1}>k_i+3$, the triplets $(\mathcal{C}_{k_i-1},\mathcal{C}_{k_i},\mathcal{C}_{k_i+1})$ and $(\mathcal{C}_{k_j-1},\mathcal{C}_{k_j},\mathcal{C}_{k_j+1})$ are comprised by distinct cycles for $i\neq j$, and are therefore independent. The distribution of $(\mathcal{C}_{k_i-1},\mathcal{C}_{k_i},\mathcal{C}_{k_i+1})$ is equal to the distribution of 
$$\biggl(\biggl(\theta^i_{-1},(\Tilde{\omega}_{B^i_{-1}}(s),\Tilde{\omega}_{b^i_{-1}}(s))_{s\leq \theta^i_{-1}}\biggr),\biggl(\theta^i_{0},(\Tilde{\omega}_{B^i_{0}}(s),\Tilde{\omega}_{b^i_{0}}(s))_{s\leq \theta^i_{0}}\biggr),\biggl(\theta^i_{1},(\Tilde{\omega}_{B^i_{1}}(s),\Tilde{\omega}_{b^i_{1}}(s))_{s\leq \theta^i_{1}}\biggr)\biggr),$$
where:
\begin{itemize}
    \item $(B_{-1}^{i},b_{-1}^{i})_{i=1,...,p}$, $(B_{0}^{i},b_{0}^{i})_{i=1,...,p}$, $(B_{1}^{i},b_{1}^{i})_{i=1...,p}$ are independent collections of independent Brownian motions in $\Omega_1\times\Omega_1$.
    \item $(\theta_{-1}^{i})_{i=1,...,p}$, $(\theta_{0}^{i})_{i=1,...,p}$, $(\theta_{1}^{i})_{i=1,...,p}$  are independent collections of independent geometric random variables with parameter $\gamma$. 
    \item $((\Tilde{\omega}_{B_{-1}^i},\Tilde{\omega}_{b_{-1}^{i}}))_{i=1,...,p}$, $((\Tilde{\omega}_{B_{0}^i},\Tilde{\omega}_{b_{0}^{i}}))_{i=1,...,p}$, $((\Tilde{\omega}_{B_{1}^i},\Tilde{\omega}_{b_{1}^{i}}))_{i=1,...,p}$ are collections of continuous paths built from the transition probability kernel $\frac{\hat{\boldsymbol{\pi}}-\gamma W_1\times W_1}{1-\gamma}$ with $(B_{-1}^i$, $b_{-1}^{i})$, $(B_{0}^i$, $b_{0}^{i})$, $(B_{1}^i$, $b_{1}^{i})$  as initial steps, respectively. These three collections of paths are independent and each of these paths is also independent from $(\theta^i_{-1},\theta^i_{0},\theta^i_{1})_{i=1,...p}$.    
   \end{itemize}
Furthermore, we observe that 
$$\mathbf{X}^{(\epsilon)}_{k_i+\delta}\overset{def}{=}\omega_{X_0}(T_{k_i+\delta+1})-\omega_{X_0}(T_{k_i+\delta})\overset{d}{=}\tilde{\omega}_{B^i_\delta}(\theta^i_\delta),$$
where $i=1,...,p$ and the last equality is in distribution. Similarly $\mathbf{Y}^{(\epsilon)}_{k_i+\delta}\overset{d}{=}\tilde{\omega}_{b^i_\delta}(\theta^i_\delta)$.
This proves that \eqref{eq5.9} is equal in distribution to:
\begin{equation}\label{eq5.10}
    \Zcal^{(\epsilon)}:=\biggr(\theta_0^i,\theta_{\delta_i}^i,(\Tilde{\omega}_{B^{i}_{\delta_i}}(s))_{s\leq\theta_{i+\delta_i}},(\Tilde{\omega}_{b_0^{i}}(u))_{u\leq\theta_{i}};\mathcal{Y}^{(\epsilon)}_i\biggl)_{i=1,...p},
\end{equation}
where
$$\mathcal{Y}^{(\epsilon)}_i=\sum_{j=1}^{i-1}\sum_{\delta\in\{-1,0,1\}}(\Tilde{\omega}_{B_{\delta_j}^{j}}(\theta_{\delta_j}^j)-\Tilde{\omega}_{b_{\delta_j}^{j}}(\theta_{\delta_j}^j))+\Tilde{\omega}_{B_{-1}^{i}}(\theta_{-1}^{i})-\Tilde{\omega}_{b_{-1}^{i}}(\theta_{-1}^{i}).$$
Furthermore $\Zcal^{(\epsilon)}$ is independent from $(\mathbf{S}_{k_i}-\bar{\mathbf{S}}_{k_i})_{i=1,...,p}$. This implies that \eqref{eq5.8} is equal to
\begin{align}\label{nondirfun}
    \sum_{\boldsymbol{\delta}\in\{-1,0,+1\}^p}\sum\limits_{\substack{1\leq k_1<k_2<...<k_p\leq [M(\epsilon)/\epsilon^{2}]-1,\\ k_{i+1}>k_i+3}}\E\biggl[\Gcal_{\epsilon,\boldsymbol{\delta}}\biggl(\omega_{X_0}(T_{k_1+\delta_1-1})-\omega_{Y_0}(T_{k_1-1}),...,\omega_{X_0}(T_{k_p+\delta_p-3(p-1)-1})-\nonumber\\-\omega_{Y_0}(T_{k_p-3(p-1)-1})\biggr)\biggr],
\end{align}
where
$$\Gcal_{\epsilon,\boldsymbol{\delta}}(z_1,...,z_p)=$$
\begin{equation}\label{functionwnoep}
\E_{\Zcal^{(\epsilon)}}\biggr[\prod_{i=1}^p\int_{0}^{\theta_0^{i}}\int_{0}^{\theta_{\delta_i}^{i}}R\biggl(s-u+\textbf{1}_{\delta_i=1}\theta_0^{i}-\textbf{1}_{\delta_i=-1}\theta_{-1}^{i},\textit{ }x+z_i+\mathcal{Y}^{(\epsilon)}_i+\Tilde{\omega}_{B_{\delta_i}^i}(s)-\Tilde{\omega}_{b_{0}^i}(u)\biggr)dsdu\biggl].
\end{equation}
We change the variables $k_{i}-3(i-1)-1\rightarrow\tilde{k}_i$. Then, \eqref{nondirfun} is equal to 
\begin{align*}
    \sum_{\mathbf{\delta}\in\{-1,0,+1\}^p}\sum\limits_{\substack{1\leq \tilde{k}_1<\tilde{k}_2<...<\tilde{k}_p\leq [M(\epsilon)/\epsilon^2]-3(p-1)-2}}\mathbb{E}\biggl[\mathcal{G}_{\epsilon,\boldsymbol \delta}\biggl(\omega_{X_0}(T_{\tilde{k}_1+\delta_1})-\omega_{Y_0}(T_{\tilde{k}_1}),...,\omega_{X_0}(T_{\tilde{k}_p+\delta_p})-\omega_{Y_0}(T_{\tilde{k}_p})\biggr)\biggr],
\end{align*}
We want to apply \textbf{Lemma \ref{lemm4.3}} to \eqref{nondirfun}, so we seek to prove that $(\Gcal_{\epsilon,\boldsymbol{\delta}})_{\epsilon>0}$ satisfies the assumptions of the lemma, namely that:
\begin{enumerate}[label=\textbf{L.\arabic*}]
    \item\label{Lass1} The $L^1$ norm of $\Gcal_{\epsilon,\boldsymbol{\delta}}$ does not depend on $\epsilon$.
    \item\label{Lass2} We have
          \begin{equation*}
              \lim_{M\rightarrow\infty}\sup_{\epsilon\in(0,1)}\int_{|\mathbf{x}|\geq M}\Gcal_{\epsilon,\boldsymbol{\delta}}(\mathbf{x})d\mathbf{x}=0.
          \end{equation*}
    \item\label{Lass3} For any $J\subseteq\{1,\dots,p\}$ we have\footnote{The function $G^J:\mathbb{R}^{2(p-|J|)}\rightarrow\mathbb{R}$ is defined as 
                      $$G^J((x_i)_{i\in\{1,...,p\}\backslash J})=\sup_{x_i\in\mathbb{R}^2,i\in J}G(x_1,...,x_p),$$
                       and we adopt the convention that when $J=\emptyset$, $G^J=G$.}  \[G^J_\epsilon(x)\leq \Hcal_\epsilon(x),\]
                       where the collection of functions $(\Hcal_\epsilon)_{\epsilon\in\N}$ satisfies \ref{Lass1} and \ref{Lass2}.
\end{enumerate}
Observe that 
$$||\Gcal_{\epsilon,\boldsymbol{\delta}}||_1=||\psi\star\psi||_1\E\biggr[\prod_{i=1}^p\int_{0}^{\theta_0^{i}}\int_{0}^{\theta_{\delta_i}^{i}}\phi\star\tilde{\phi}(s-u+\textbf{1}_{\delta_i=1}\theta_0^{i}-\textbf{1}_{\delta_i=-1}\theta_{-1}^{i})dsdu\biggl],$$
which is indeed an $\epsilon-$independent constant, so \ref{Lass1} is true. To prove item \ref{Lass2} we observe that, from \textbf{Proposition \ref{thm:prop4.2}} and \textbf{Corollary \ref{thm:cor4.1}}, 
    $$\mathbb{P}[|k^i_\epsilon|\geq M_0]\rightarrow0,$$
    as $M_0\rightarrow\infty$, uniformly in $\epsilon$, where
$$k_{\epsilon}^i=|\mathcal{Y}^{(\epsilon)}_i|+\sup_{s\leq\theta_0^{i}}|\Tilde{\omega}_{B_0^i}(s)|+\sup_{s\leq\theta_{\delta_i}^{i}}|\Tilde{\omega}_{b_{\delta_i}^i}(s)|,$$
for $i=1,\dots,p$. Arguing similarly as in  \textbf{Remark \ref{Remofapplicability}}, $(\Gcal_{\epsilon,\boldsymbol{\delta}})_{\epsilon>0}$ satisfies \ref{Lass2}. 
For \ref{Lass3}, we have $\Gcal_{\epsilon,\mathbf{\delta}}^J\leq\tilde{\Gcal}_{\epsilon,\mathbf{\delta}}$, for any $J\subset\{1,...,p\}$, where
$$\tilde{\Gcal}_{\epsilon,\mathbf{\delta}}((z_i)_{i\in\{1,..,p\}\backslash J})=$$
\begin{align*}
||R||_{\infty}^{|J|}\mathbb{E}_{\mathcal{Z}^{(\epsilon)}}\biggr[\prod_{i\in\{1,..,p\}\backslash J}\int_{0}^{\theta_0^{i}}\int_{0}^{\theta_{\delta_i}^{i}}R\biggl(s-u+\textbf{1}_{\delta_i=1}\theta_0^{i}-\textbf{1}_{\delta_i=-1}\theta_{-1}^{i},\textit{ }x+z_i+\mathcal{Y}^{(\epsilon)}_i+\Tilde{\omega}_{B_{\delta_i}^i}(s)-\Tilde{\omega}_{b_{0}^i}(u)\biggr)dsdu\cdot\\\cdot\prod_{i\in J}\theta_0^{i}\theta_{\delta_i}^{i}\biggl].
\end{align*}
By arguing similarly to above, we get 
$$\lim_{M\rightarrow\infty}\sup_{\epsilon\in(0,1)}\int_{|x|\geq M}\tilde{\Gcal}_{\epsilon,\mathbf{\delta}}(x)dx=0.$$
Therefore, \textbf{Lemma \ref{lemm4.3}} implies that \eqref{nondirfun} is asymptotic to:
\begin{equation}\label{asym}
    \biggr((\gamma/2\pi)\log \frac{1}{\epsilon}\biggl)^p\sum_{\boldsymbol{\delta}\in\{-1,0,+1\}^p}||\Gcal_{\epsilon,\boldsymbol{\delta}}||_1,
\end{equation}
as $\epsilon\rightarrow0$.\par
We need to show that this expression is equal to $(||R||_1\log \frac{1}{\epsilon}/2\pi)^p$. Since $R(t,x)=\phi\star\Tilde{\phi}(t)\psi\star\psi(x)$, it suffices to show that
$$\sum_{\boldsymbol{\delta}\in\{-1,0,+1\}^p}\E\biggr[\prod_{i=1}^p\int_{0}^{\theta_0^{i}}\int_{0}^{\theta_{\delta_i}^{i}}\phi\star\tilde{\phi}(s-u+\textbf{1}_{\delta_i=1}\theta_0^{i}-\textbf{1}_{\delta_i=-1}\theta_{-1}^{i})dsdu\biggl]=||\phi\star\tilde\phi||_1=||\phi||^2_1$$
By writing the above expectation explicitly, we have to calculate
$$\sum_{\boldsymbol{\delta}\in\{-1,0,+1\}^p}\sum_{\textbf{N}\in\mathbb{N}^{3p}}^{\infty}\prod_{i=1}^p\gamma^3(1-\gamma)^{N_{-1}^i+N_{0}^i+N_{1}^i-3}\prod_{i=1}^p\int_{0}^{N_0^{i}}\int_{0}^{N_{\delta_i}^i}\phi\star\Tilde{\phi}(s-u-\textbf{1}_{\delta_i=1}N_0^{i}+\textbf{1}_{\delta_i=-1}N_{-1}^i)dsdu.$$
We then define:
$$L_0(N):=N\cdot\int_0^1\int_0^1\phi(s-u)dsdu+2\cdot(N-1)\int_0^1\int_0^1\phi(s+u)dsdu,$$
$$L_-=L_+:=\int_0^1\int_0^{1}\phi(s+u)dsdu.$$
Since $\phi$ is symmetric we have, for $\delta_i=-1,1$
\[\int_{0}^{N_0^{i}}\int_{0}^{N_{\delta_i}^i}\phi(s-u-\textbf{1}_{\delta_i=1}N_0^{i}+\textbf{1}_{\delta_i=-1}N_{-1}^i)dsdu=\int_0^1\int_0^{1}\phi(s+u)dsdu.\]
Similarly, since $\phi$ is symmetric and supported in $[-1,1]$ we have:
\[\int_{0}^{N_0^{i}}\int_{0}^{N_{0}^i}\phi(s-u)dsdu=\sum_{k=0}^{N_0^i-1}\int_k^{k+1}\int_k^{k+1}\phi(s-u)dsdu+2\sum_{k=0}^{N_0^i-2}\int_k^{k+1}\int_{k+1}^{k+2}\phi(s-u)dsdu=\]
\[=N_0^i\cdot\int_0^1\int_0^1\phi(s-u)dsdu+2\cdot(N_0^i-1)\int_0^1\int_0^1\phi(s+u)dsdu.\]
Thus, the above expression can be rewritten as
\begin{equation}\label{wtf}
    \sum_{\{A_0,A_-,A_+\}\vdash\{1,...,p\}}\sum_{\textbf{N}\in\mathbb{N}^{3p}}^{\infty}\prod_{i=1}^p\gamma^3(1-\gamma)^{N_{-1}^i+N_{0}^i+N_{1}^i-3}\prod_{i\in A_0}L_0(N_0^{i})L_+^{|A_+|}L_-^{|A_-|},
\end{equation}
where $\{A_0,A_-,A_+\}\vdash\{1,...,p\}$ means that the sets $A_0,A_-,A_+$ partition $\{1,...,p\}$, and we sum over all such partitions. By summing over all $\textbf{N}$, \eqref{wtf} simplifies to 
$$\sum_{\{A_0,A_-,A_+\}\vdash\{1,...,p\}}L_{\gamma}^{|A_0|}L_+^{|A_+|}L_-^{|A_-|},$$
where
$$L_{\gamma}:=\frac{1}{\gamma}\biggl(\int_0^1\int_0^1\phi\star\Tilde{\phi}(s-u)dsdu+2\int_0^1\int_0^1\phi\star\Tilde{\phi}(s+u)dsdu\biggr)-2\int_0^1\int_0^1\phi\star\Tilde{\phi}(s+u)dsdu.$$
Therefore, \eqref{wtf} is equal to
$$\biggl(L_{\gamma}+L_++L_-\biggr)^p=(A(R)/\gamma)^p,$$
where
$$A(R):=\biggl(\int_0^1\int_0^1\phi\star\Tilde{\phi}(s-u)dsdu+2\int_0^1\int_0^1\phi\star\Tilde{\phi}(s+u)dsdu\biggr)||\psi||_1^2.$$
This implies that \eqref{eq5.8} is asymptotic to 
$$\biggr((\gamma/2\pi)\log \frac{1}{\epsilon}\biggl)^p(A(R)/\gamma)^p=\biggr(\frac{A(R)}{2\pi\log \frac{1}{\epsilon}}\biggl)^p.$$
We will show that $A(R)=||R||_1$. Since, $||R||_1=||\phi||_1^2||\psi||_1^2$ it is sufficient to prove that
$$\int_0^1\int_0^1\phi\star\Tilde{\phi}(s-u)dsdu+2\int_0^1\int_0^1\phi\star\Tilde{\phi}(s+u)dsdu=||\phi||_1^2.$$
By a change of variables $s\rightarrow s+u$, the first term on the left-hand side is equal to
$$\int_{0}^1\int_{-u}^{1-u}\phi\star\Tilde{\phi}(s)dsdu,$$
while the second term is equal to
$$\int_0^1\int_u^1\phi\star\Tilde{\phi}(s)dsdu+\int_0^1\int_{1-u}^1\phi\star\Tilde{\phi}(s)dsdu,$$
where we also used the fact that $\phi\star\tilde\phi(s)=0$ for $s\geq1$.
By adding these terms we get
$$\int_{0}^1\biggr(\int_{-u}^{1-u}\phi\star\Tilde{\phi}(s)dsdu+\int_u^1\phi\star\Tilde{\phi}(s)dsdu+\int_{1-u}^1\phi\star\Tilde{\phi}(s)ds\biggl)du.$$
In the second term inside the integral, we use  the fact that $\phi\star\Tilde{\phi}$ is symmetric, and we get that the above integral is equal to
$$\int_{0}^1\biggr(\int_{-u}^{1-u}\phi\star\Tilde{\phi}(s)dsdu+\int_{-1}^{-u}\phi\star\Tilde{\phi}(s)dsdu+\int_{1-u}^1\phi\star\Tilde{\phi}(s)ds\biggl)du=\int_{0}^1\int_{-1}^1\phi\star\Tilde{\phi}(s)dsdu=||\phi||_1^2,$$
and therefore $A(R)=||R||_1$. This concludes the proof.
\end{proof}

Now, we move on to prove the bound \eqref{eq:exponential_bound}:

\begin{myprop}\label{cor5.1}
    Let $X_0,Y_0,R,M(\epsilon)$ be as before. Then for every $\hat{\beta}<\hat{\beta}_c(R)$:
    \begin{equation}\label{eq:exp_unif_bound}
        \sup_{x\in\mathbb{R}^2}{\mathbb{E}}\biggr[\exp\biggr(\frac{\hat{\beta}^2}{\log \frac{1}{\epsilon}}\int_{[0,T_{[M(\epsilon)/\epsilon^{2}]}]^2}R(s-u,x+\omega_{X_0}(s)-\omega_{Y_0}(u))dsdu\biggl)\biggl]\lesssim1,
    \end{equation}
    for all $\epsilon$ small enough. Furthermore
   \begin{equation}\label{expunifbound}
       \limsup_{\epsilon\rightarrow0}\sup_{x\in\mathbb{R}^2}{\mathbb{E}}\biggr[\exp\biggr(\frac{\hat{\beta}^2}{\log \frac{1}{\epsilon}}\int_{[0,T_{[M(\epsilon)/\epsilon^{2}]}]^2}R(s-u,x+\omega_{X_0}(s)-\omega_{Y_0}(u))dsdu\biggl)\biggl]\leq\biggl(1-\frac{\hat{\beta}^2}{\hat{\beta}_c(R)^2}\biggr)^{-1}.
   \end{equation}
    \end{myprop}

\begin{proof}
     To prove this, we will argue similarly as in \textbf{Lemma \ref{techlemma}} and prove a uniform in $x\in\R^2$ bound for the first moment of \eqref{eq:add_funct}. Then, we will use a Khas’minskii-type lemma for Markov chains (\textbf{Proposition \ref{kamsinskilemma}}) to 'upgrade' this bound to \eqref{eq:exp_unif_bound} and \eqref{expunifbound}.\par
     First, we write  
    \[\frac{1}{\log\frac{1}{\epsilon}}\E\biggl[\int_{[T_3,T_{[M(\epsilon)/\epsilon^{2}]}]^2}R(s-u,x+\omega_{X_0}(s)-\omega_{Y_0}(u))dsdu\biggr]=\sum_{k=3}^{[M(\epsilon)/\epsilon^2]-1}\E[d_{k,k-1}+d_{k,k}+d_{k,k+1}]\]
    with $d_{k,k+\delta}$ as in \eqref{dkterms}.  By following the same arguments as in \textbf{Lemma \ref{techlemma}}, there exists  a non-negative integrable function $\Gcal_\delta$, $\delta\in\{-1,0,1\}$, such that $\sum_{\delta\in\{-1,0,1\}}||\Gcal_{\delta}||_1=||R||_1/\gamma$ and such that  
    \begin{align}\label{exampletermagain}
        \frac{1}{\log\frac{1}{\epsilon}}\sum_{k=3}^{[M(\epsilon)/\epsilon^2]-1}\E[d_{k,k-1}+d_{k,k}+d_{k,k+1}]=\frac{1}{\log\frac{1}{\epsilon}}\sum_{k=2}^{[M(\epsilon)/\epsilon^2]-2}\E[\Gcal_{\epsilon,\delta}(\omega_{X_0}(T_{k+\delta})-\omega_{Y_0}(T_{k}))].
    \end{align}
    Similar arguments to the proof of \textbf{Lemma \ref{lemm4.3}} show that this sum is equal to 
    \[\frac{1}{\log\frac{1}{\epsilon}}\sum_{k=1}^{[M(\epsilon)/\epsilon^2]-3}\E\biggl[\sum_{\delta\in\{-1,0,1\}}\tilde{\Gcal}_{\delta}(x+\omega_{X_0}(T_{k})-\omega_{Y_0}(T_k))\biggr],\]
    for a new non-negative integrable function $\tilde{\Gcal}_\delta$, such that $\sum_{\delta\in\{-1,0,1\}}||\tilde{\Gcal}_{\delta}||_1=||R||_1/\gamma$. Now we use the local limit theorem \eqref{local}, as well as the bound $p_{ak}(x)\leq\frac{1}{2\pi a k}$, where $p_{s}(x)$ is the two-dimensional heat kernel. We get that \eqref{exampletermagain} is bounded above by
    \[||R||_1\sum_{k=1}^{[M(\epsilon)/\epsilon^2]-3}\frac{1}{4\pi  k}+\sum_{k=1}^{[M(\epsilon)/\epsilon^2]-3}\frac{\sup_{x\in\R^2}|\Rcal_{\epsilon,k}(x)|}{ k}.\]
    Therefore, for $\hat{\beta}^2<2\pi/||R||_1=\hat{\beta}_c(R)^2$, and for every $\eta<1$ such that $\frac{\hat{\beta}^2}{\hat{\beta}_c(R)^2}+\eta<1$ we have 
    \[\frac{\hat{\beta}}{\log\frac{1}{\epsilon}}\E\biggl[\int_{[T_3,T_{[M(\epsilon)/\epsilon^{2}]}]^2}R(s-u,x+\omega_{X_0}(s)-\omega_{Y_0}(u))dsdu\biggr]\leq\frac{\hat{\beta}^2}{\hat{\beta}_c(R)^2}+\eta,\]
    for all $x\in\R^2$, and $\epsilon$ small enough. Since $T_3$ is the sum of three independent geometric random variables we get that, for all  $\hat{\beta}^2<\hat{\beta}_c(R)^2$, and all $\eta'<1$ such that $\frac{\hat{\beta}^2}{\hat{\beta}_c(R)^2}+\eta'<1$, we have
    \begin{equation}\label{elemboundtouse1}
        \frac{\hat{\beta}}{\log\frac{1}{\epsilon}}\E\biggl[\int_{[0,T_{[M(\epsilon)/\epsilon^{2}]}]^2}R(s-u,x+\omega_{X_0}(s)-\omega_{Y_0}(u))dsdu\biggr]\leq\frac{\hat{\beta}^2}{\hat{\beta}_c(R)^2}+\eta',
    \end{equation}
     for all $x\in\R^2$, and $\epsilon$ small enough. This inequality shows that for all  $\hat{\beta}^2<2\pi/||R||_1=\hat{\beta}_c(R)^2$, and all $\eta''<1$ such that $\frac{\hat{\beta}^2}{\hat{\beta}_c(R)^2}+\eta''<1$, we have
    \begin{equation}\label{elemboundtouse2}
        \frac{\hat{\beta}}{\log\frac{1}{\epsilon}}\E\biggl[\int_{[0,T_{[M(\epsilon)/\epsilon^{2}]}]^2}R(s-u,x+\omega_{X_0}(s)-\omega_{Y_0}(u))dsdu\biggl|(X_0,Y_0)\biggr]\leq\frac{\hat{\beta}^2}{\hat{\beta}_c(R)^2}+\eta'',
    \end{equation}
     for all $x\in\R^2$, $(X_0,Y_0)\in\Omega_1^2$, and $\epsilon$ small enough. Indeed, from the regeneration property 
     \begin{align*}
        \E\biggl[\int_{[T_3,T_{[M(\epsilon)/\epsilon^2]}]^2}R(s-u,x+\omega_{X_0}(s)-\omega_{Y_0}(u))\biggl|(X_0,Y_0)\biggr]\leq\E\biggl[\E_{(\bar{\omega}_{X_{T_3}},\bar{\omega}_{Y_{T_3}})}\biggl[\int_{[0,T_{[M(\epsilon)/\epsilon^2]}]^2}R(s-u,x+\\+\omega_{Y_0}(T_3)-\omega_{X_0}(T_3)+\bar{\omega}_{X_{T_3}}(s)-\bar{\omega}_{Y_{T_3}}(u))dsdu\biggr]\biggl|(X_0,Y_0)\biggr],
     \end{align*}
     where $(\bar{\omega}_{X_{T_3}},\bar{\omega}_{Y_{T_3}})$ is a new path built from the Markov chain, with $({X_{T_3}},{Y_{T_3}})\sim W_1\times W_1$ as an initial condition and is independent from $(\omega_{X_0}(s),\omega_{Y_0}(s))_{s\leq T_3}$. With this inequality, and \eqref{elemboundtouse1}, we get
     \[
        \frac{\hat{\beta}}{\log\frac{1}{\epsilon}}\E\biggl[\int_{[T_3,T_{[M(\epsilon)/\epsilon^{2}]}]^2}R(s-u,x+\omega_{X_0}(s)-\omega_{Y_0}(u))dsdu\biggl|(X_0,Y_0)\biggr]\leq\frac{\hat{\beta}^2}{\hat{\beta}_c(R)^2}+\eta',
     \]
     Again, since $R$ is bounded and $T_3$ has a finite mean, this bound implies \eqref{elemboundtouse2}. Now we apply \textbf{Proposition \ref{kamsinskilemma}}. We observe that there is a constant $c$ such that  
     $$e^{\frac{\hat{\beta}}{\log\frac{1}{\epsilon}}(p_{k,k-1}(x)+p_{k,k}(x)+p_{k,k+1}(x))}-1\leq\frac{\hat{\beta}}{\log\frac{1}{\epsilon}} e^{\frac{c}{\log\frac{1}{\epsilon}}}(p_{k,k-1}(x)+p_{k,k}(x)+p_{k,k+1},(x)),$$
      where $p_{k,k+\delta}$, $\delta\in\{-1,0,1\}$ as in \eqref{pkterm}. From this inequality and \eqref{elemboundtouse2}, we get that \eqref{assumptioninprop} is satisfied for $\hat{\beta}<\hat{\beta}_c(R)$.   Therefore,  \textbf{Proposition \ref{kamsinskilemma}} yields that, for all $x\in\R^2$ and $\epsilon$ small enough,
     \[{\mathbb{E}}\biggr[\exp\biggr(\frac{\hat{\beta}^2}{\log \frac{1}{\epsilon}}\int_{[0,T_{[M(\epsilon)/\epsilon^{2}]}]^2}R(s-u,x+\omega_{X_0}(s)-\omega_{Y_0}(u))dsdu\biggl)\biggl]\leq\biggl(1-\frac{\hat{\beta}^2}{\hat{\beta}_c(R)^2}-\eta''+\alpha\biggr)^{-1},\]
     where $\alpha=2(\exp(3||R||_{\infty}\hat{\beta}/\log\frac{1}{\epsilon})-1)$. This proves the first claim. Now sending $\epsilon$ to $0$ concludes the proof since $\eta''$ is arbitrary and $\alpha\rightarrow0$.
\end{proof}

Finally,  \eqref{eq:mom_liminf} follows easily from \textbf{Proposition \ref{cor5.1}} and \textbf{Lemma \ref{techlemma}}. Indeed, we set
    \[\Mcal(n)=\frac{\hat{\beta}^2}{\log\frac{1}{\epsilon}}\int_{[0,T_{n}]^2}R(s-u,x+\omega_{X_0}(s)-\omega_{Y_0}(u))dsdu.\]
    We also define 
     $$N_{\epsilon}:=\max\{i\textit{/ }T_i<M(\epsilon)/\epsilon^2\},$$
     and
     $$A_{\epsilon,\delta}=\{|N_{\epsilon}/([M(\epsilon)/\epsilon^{2}])-\gamma|\geq\delta\}.$$
     Observe that $T_i-T_{i-1}\sim Geo(\gamma)$ and are independent, which implies  that $T_{[M_2(\epsilon)/\epsilon^{2}]}/[M(\epsilon)/\epsilon^{2}]\rightarrow1/\gamma$ a.s.. This implies that $N_{\epsilon}/(M(\epsilon)/\epsilon^{2})\rightarrow\gamma$ in probability.\par
     We have:
    $$\E\biggl[\biggl(\frac{\hat{\beta}^2}{\log\frac{1}{\epsilon}}\int_{[0,M(\epsilon)/\epsilon^2]^2}R(s-u,x+\omega_{X_0}(s)-\omega_{Y_0}(u))dsdu\biggr)^p\biggr]\geq\E[\mathcal{M}(N_{\epsilon})^p]\geq\E[\textbf{1}_{A_{\epsilon,\delta}^c}\mathcal{M}(N_{\epsilon})^p].$$
    On $A_{\epsilon,\delta}^c$ we have $N_{\epsilon}\geq(\gamma-\delta)(M(\epsilon)/\epsilon^{2})$, and therefore
    $$\E[\mathcal{M}(N_{\epsilon})^p]\geq\E[\textbf{1}_{A_{\epsilon,\delta}^c}\mathcal{M}((\gamma-\delta)M(\epsilon)/\epsilon^{2})^p]=\E[(\textbf{1}_{A_{\epsilon,\delta}^c}-1)\mathcal{M}((\gamma-\delta)M(\epsilon)/\epsilon^{2})^p]+\E[\mathcal{M}((\gamma-\delta)M(\epsilon)/\epsilon^{2})^p].$$
    The first term on the right-hand side is estimated above by
    $$\E[(\textbf{1}_{A_{\epsilon,\delta}^c}-1)^2]^{1/2}\textbf{ }\E[\mathcal{M}((\gamma-\delta)M(\epsilon)/\epsilon^{2})^{2p}]^{1/2},$$
    which goes to $0$ as $\epsilon\rightarrow0$, since $\P(A_{\epsilon,\delta}^c)\rightarrow1$ and $\E[\mathcal{M}((\gamma-\delta)M(\epsilon)/\epsilon^{2})^{2p}]\lesssim1$, from \textbf{Proposition \ref{cor5.1}}. On the other hand, we observe that $\Mcal([(\gamma-\delta)M(\epsilon)/\epsilon^2])^p$ is bounded below by
    \[p!\biggl(\frac{\hat{\beta}}{\log\frac{1}{\epsilon}}\biggr)^p\sum_{\mathbf{\delta}\in\{-1,0,+1\}^p}\sum\limits_{\substack{3\leq k_1<k_2<...<k_p\leq [(\gamma-\delta)M(\epsilon)/\epsilon^2],\\ k_{i+1}>k_i+3}}\mathbb{E}\biggl[\prod_{i=1}^p{d}_{k_i,k_i+\delta_i}\biggr],\]
    where $d_{k,k+\delta}$ are as in \eqref{dkterms}. Now, \textbf{Lemma \ref{techlemma}} shows that this converges to $p!(\hat{\beta}^2/\hat{\beta}_c(R)^2)^p$. These two observations show \eqref{eq:mom_liminf}.

\begin{remark}\label{sanitycheck}
    It is not hard to show that \textbf{Theorem \ref{thm:5.1}} implies that
    $$\mathbf{E}[u_{\epsilon}(t,x)^2]\rightarrow\biggr(1-\frac{\hat{\beta}^2}{\hat{\beta}_c(R)^2}\biggl)^{-1},$$
    as $\epsilon\rightarrow0$.
\end{remark}

\section{Proof of Proposition \ref{thm:prop2.4}}\label{sec:6}

In this section, we will prove \textbf{Proposition \ref{thm:prop2.4}}. Recall the formula for $\Fcal_{\epsilon}(r,y,M_1(\epsilon),M_2(\epsilon))$:
    \begin{equation}\label{eq6.1}
        \Fcal_{\epsilon}(r,y,M_1(\epsilon),M_2(\epsilon))=\int_{\mathbb{R}^{4}}\int_{[0,1]^2}\Hat{\E}_{B,t/\epsilon^2}[I_{\epsilon}e^{J_{\epsilon}(M_1(\epsilon),M_2(\epsilon))}]\prod_{i=1}^2\phi(s_i)\psi(x_i)d\Bar{s}d\Bar{x},
    \end{equation}
    where
\begin{equation}\label{eqforI}
   I_{\epsilon}(x_1,x_2,y,s_1,s_2,r)=\prod_{i=1}^2g(\epsilon x_i-\epsilon B^i_{(t-r)/\epsilon^2-s_i}+y)
\end{equation}
    and

    \begin{align}\label{6.2}
    J_{\epsilon}(M_1(\epsilon),M_2(\epsilon))=\frac{\hat{\beta}^2}{{\log\frac{1}{\epsilon}}}\int_{-1}^{M_1(\epsilon)}\int_{-1}^{M_2(\epsilon)}R_{\phi}(u_1,u_2)\psi\star\psi(x_1-x_2+\Delta B^1_{(t-r)/\epsilon^2-s_1,(t-r)/\epsilon^2+u_1}-\nonumber \\
    -\Delta B^2_{(t-r)/\epsilon^2-s_2,(t-r)/\epsilon^2+u_2})du_1du_2,
\end{align}
    {with $R_{\phi}$ as in \eqref{Rphi}. Finally, recall that $M_1(\epsilon),M_2(\epsilon)$ satisfy the following assumptions:
   \begin{enumerate}[label=\textbf{A.\arabic*}]
       \item\label{Ass1} $0\leq M_1(\epsilon),M_2(\epsilon)\leq r$.
       \item\label{Ass2} For $i=1,2$, $\log M_i(\epsilon)/\log\frac{1}{\epsilon}\rightarrow0$, as $\epsilon\rightarrow0$.
       \item\label{Ass3} Either $M_1(\epsilon)=M_2(\epsilon)$ or for all $\epsilon$ small enough $M_1(\epsilon)-M_2(\epsilon)\geq c>0$, for some $c$.
   \end{enumerate}
    First, we argue that it is sufficient to prove \textbf{Proposition \ref{thm:prop2.4}} in the case $M_1(\epsilon)=M_2(\epsilon)$. Indeed, recall from \textbf{Remark \ref{Nikosremark}} that if $M_1(\epsilon)-M_2(\epsilon)\geq c>0$, then for all $\epsilon$ small enough
    $$\Fcal_{\epsilon}(r,y,M_1(\epsilon)/\epsilon^2,M_2(\epsilon)/\epsilon^2)=\Fcal_{\epsilon}(r,y,M_2(\epsilon)/\epsilon^2+1,M_2(\epsilon)/\epsilon^2).$$
    From \eqref{6.2},  we get that 
    \begin{equation}\label{errordecomp} J_\epsilon(M_2(\epsilon)/\epsilon^2+1,M_2(\epsilon)/\epsilon^2)=\mathcal{R}_{\epsilon}+J_\epsilon(M_2(\epsilon)/\epsilon^2,M_2(\epsilon)/\epsilon^2),
    \end{equation}
    with 
    \begin{align}\label{anothererrorterm}
        \mathcal{R}_{\epsilon}:=\frac{\hat{\beta}^2}{{\log\frac{1}{\epsilon}}}\int_{M_2(\epsilon)/\epsilon^2}^{M_2(\epsilon)/\epsilon^2+1}\int_{-1}^{M_2(\epsilon)/\epsilon^2}R_{\phi}(u_1,u_2)\psi\star\psi(x_1-x_2+\Delta B^1_{(t-r)/\epsilon^2-s_1,(t-r)/\epsilon^2+u_1}-\nonumber \\
    -\Delta B^2_{(t-r)/\epsilon^2-s_2,(t-r)/\epsilon^2+u_2})du_1du_2.
    \end{align}
    Since $\phi$ is supported on $[0,1]$,  $R_\phi(u_1,u_2)=0$ when $|u_1-u_2|\geq1$. Therefore, $\mathcal{R}_{\epsilon}$ is equal to
    \begin{align*}
        \frac{\hat{\beta}^2}{{\log\frac{1}{\epsilon}}}\int_{M_2(\epsilon)/\epsilon^2}^{M_2(\epsilon)/\epsilon^2+1}\int_{M_2(\epsilon)/\epsilon^2-1}^{M_2(\epsilon)/\epsilon^2}R_{\phi}(u_1,u_2)\psi\star\psi(x_1-x_2+\Delta B^1_{(t-r)/\epsilon^2-s_1,(t-r)/\epsilon^2+u_1}-\nonumber \\
    -\Delta B^2_{(t-r)/\epsilon^2-s_2,(t-r)/\epsilon^2+u_2})du_1du_2.
    \end{align*}
    Since $R_\phi$ and $\psi\star\psi$ are bounded, we conclude that this term (and therefore $\mathcal{R}_\epsilon$) converges to $0$ as $\epsilon\rightarrow0$, uniformly over the paths $B^1, B^2$ and over $r,y,s,x_1,x_2$. Now, assume for the moment that \textbf{Proposition \ref{thm:prop2.4}} is true for $\Fcal_{\epsilon}(r,y,M_2(\epsilon)/\epsilon^2,M_2(\epsilon)/\epsilon^2)$. Then, from our observations about $\mathcal{R}_{\epsilon}$,  \eqref{eq6.1}, \eqref{errordecomp} and from Hölder's inequality 
    $$\Fcal_{\epsilon}(r,y,M_1(\epsilon)/\epsilon^2,M_2(\epsilon)/\epsilon^2)-\Fcal_{\epsilon}(r,y,M_2(\epsilon)/\epsilon^2,M_2(\epsilon)/\epsilon^2)\rightarrow0,$$
    as $\epsilon\rightarrow0$ and 
    $$\Fcal_{\epsilon}(r,y,M_1(\epsilon)/\epsilon^2,M_2(\epsilon)/\epsilon^2)\lesssim\Fcal_{\epsilon}(r,y,M_2(\epsilon)/\epsilon^2,M_2(\epsilon)/\epsilon^2),$$
    for all $\epsilon$. Hence, \textbf{Proposition \ref{thm:prop2.4}} is true for $\Fcal_{\epsilon}(r,y,M_1(\epsilon)/\epsilon^2,M_2(\epsilon)/\epsilon^2)$ as well.}\par
    {Therefore, we will focus on 
    $$\Fcal_{\epsilon}(r,y,M(\epsilon)/\epsilon^2,M(\epsilon)/\epsilon^2),$$
    where $M(\epsilon)$ satisfies the assumptions \ref{Ass1} and \ref{Ass2} described above. For this case, we prove \textbf{Proposition \ref{thm:prop2.4}}  in \textbf{Lemma \ref{lem:bound}} and \textbf{Lemma \ref{lem:limit}}.} First, we will calculate the limit of $\Hat{\E}_{B,t/\epsilon^2}[I_{\epsilon}\exp(J_{\epsilon})]$. Our main strategy will be to try to decouple $I_{\epsilon}$ and $J_{\epsilon}$. This means that, asymptotically, we should have
    $$\Hat{\E}_{B,t/\epsilon^2}\biggl[I_{\epsilon}\exp(J_{\epsilon})\biggr]\approx\Hat{\E}_{B,t/\epsilon^2}[I_{\epsilon}]\Hat{\E}_{B,t/\epsilon^2}[\exp(J_{\epsilon})].$$
    Then the expectation
    $$\Hat{\E}_{B,t/\epsilon^2}[\exp(J_{\epsilon})]$$
    should converge to a constant, giving us the effective variance and
    $$\Hat{\E}_{B,t/\epsilon^2}[I_{\epsilon}]$$
    should  converge to $p_{t-r}\star g(y)^2$, where $p_t$ is the two-dimensional heat kernel. This heuristic is made formal in \textbf{Lemma \ref{lem:prep1}} and \textbf{Lemma \ref{lem:prep2}}.\par

    \begin{mylem}\label{lem:prep1}
        For all $\hat{\beta}<\hat{\beta}_c(R)$  we have that
        $$\Hat{\E}_{B,t/\epsilon^2}[I_{\epsilon}]\rightarrow p_{t-r}\star g(y)^2,$$
        and
        $$\Hat{\E}_{B,t/\epsilon^2}[\exp(J_{\epsilon}(M(\epsilon)/\epsilon^2,M(\epsilon)/\epsilon^2))]\rightarrow\biggl(1-\frac{\hat{\beta}^2}{\hat{\beta}_c(R)^2}\biggr)^{-1},$$
         as $\epsilon\rightarrow0$.
    \end{mylem}
   
    \begin{proof}
      We calculate the limit of $\Hat{\E}_{B,t/\epsilon^2}[I_{\epsilon}]$ first.  Observe that $I_{\epsilon}$ is a functional of two independent paths 
      distributed according to the exponentially tilted measure (\ref{expmeasure}). We will emphasize this dependence by writing it as $I_{\epsilon}(B^1, B^2)$. We 
      replace the paths that appear in (\ref{eqforI}) with paths built from the Markov chain constructed in \textbf{Section \ref{sec:3}}. More specifically, set 
      $$T=t/\epsilon^2\textit{, } \tau=(t-r)/\epsilon^2-[(t-r)/\epsilon^2] \textit{ and } N_{\epsilon}=[t/\epsilon^2-\tau].$$ 
      We sample independently  two random vectors
      $(\tilde{X}_0^1,..., \tilde{X}^1_{N_{\epsilon}+1})$, $(\tilde{X}_0^2,..., \tilde{X}^2_{N_{\epsilon}+1})$ as described in \textbf{Section \ref{sec:3}},  with  $\tilde{X}_0^1, \tilde{X}_0^2\in\Omega_{\tau}$
      and $\tilde{X}^1_{N_{\epsilon}+1},\tilde{X}^2_{N_{\epsilon}+1}\in\Omega_{t/\epsilon^2-\tau-N_{\epsilon}}$. Then from \eqref{Markovcon}, we have that:
      $$\hat{\E}_{B,t/\epsilon^2}[I_{\epsilon}(B^1,B^2)]=\E\biggl[\Gcal_{\epsilon}(X^1_{N_{\epsilon}},X^2_{N_{\epsilon}})I_{\epsilon}
      ([\tilde{X}_0^1,...,\tilde{X}^1_{N_{\epsilon}+1}],[\tilde{X}_0^2,...,\tilde{X}^2_{N_{\epsilon}+1}])\biggr],$$
      with $\Gcal_{\epsilon}$ as in \eqref{eqedge}. From \textbf{Remark \ref{Remedge}} $\Gcal_{\epsilon}$ 
      converges to $1$, uniformly as $\epsilon\rightarrow0$. Therefore, it is sufficient to calculate:
     $$\E\biggl[I_{\epsilon}([\tilde{X}_0^1,...,\tilde{X}^1_{N_{\epsilon}+1}],[\tilde{X}_0^2,...,\tilde{X}^2_{N_{\epsilon}+1}])\biggr],$$
      as $\epsilon\rightarrow0$.\par
      We see that the term in the above expectation is a functional of $(\tilde{X}_0^1,..., \tilde{X}^1_{N_{\epsilon}+1})$ and $(\tilde{X}_0^2,..., \tilde{X}^2_{N_{\epsilon}+1})$. {We will apply \textbf{Lemma \ref{thm:lem4.1}} and replace the steps of these two trajectories of the Markov chain at $[(t-r)/\epsilon^2]$ and $[(t-r)/\epsilon^2]+1$ by independent Brownian motions that are also independent of the previous steps. More specifically, consider the sequences $({X}_0^1,...,{X}^1_{N_{\epsilon}+1})$ and $({X}_0^2,...,{X}^2_{N_{\epsilon}+1})$, which are generated in  the same way as $(\tilde{X}_0^1,..., \tilde{X}^1_{N_{\epsilon}+1})$ and $(\tilde{X}_0^2,..., \tilde{X}^2_{N_{\epsilon}+1})$ respectively, but at the steps $k=[(t-r)/\epsilon^2],[(t-r)/\epsilon^2]+1$ we sample
      $({X}^i_{k},{X}^i_{k})\sim W_1\times W_1$, $i=1,2$ independently from the previous steps (also $({X}^i_{k},{X}^i_{k}),
      ({X}^j_{k},{X}^j_{k})$ are independent for $i\neq j$). Let $\boldsymbol{\upomega}_{X_0^1}=[{X}_0^1,..., {X}^1_{N_{\epsilon}+1}]$
      and $\boldsymbol{\upomega}_{X_0^2}=[{X}_0^2,..., {X}^2_{N_{\epsilon}+1}]$ be the corresponding paths built according to \eqref{eq3.1}.} Then, from
      \textbf{Lemma \ref{thm:lem4.1}} and \textbf{Remark \ref{Mixingremark}}, there are constants $\mathcal{A}_{\epsilon},\mathcal{B}_{\epsilon}\rightarrow1$, as $\epsilon\rightarrow0$, such that
      $$\mathcal{A}_{\epsilon}\E\biggl[I_{\epsilon}(\boldsymbol{\upomega}_{X_0^1},\boldsymbol{\upomega}_{X_0^2})\biggr]\leq\E\biggl[I_{\epsilon}([\tilde{X}_0^1,...,\tilde{X}^1_{N_{\epsilon}+1}],[\tilde{X}_0^2,...,\tilde{X}^2_{N_{\epsilon}+1}])\biggr]\leq\mathcal{B}_{\epsilon}\E\biggl[I_{\epsilon}(\boldsymbol{\upomega}_{X_0^1},\boldsymbol{\upomega}_{X_0^2})\biggr].$$
      Therefore, it is sufficient to calculate the limit of:
      $$\E\biggl[I_{\epsilon}(\boldsymbol{\upomega}_{X_0^1},\boldsymbol{\upomega}_{X_0^2})\biggr].$$
      By adding and subtracting the term $\epsilon\boldsymbol{\upomega}_{X_0^i}([(t-r)/\epsilon^2]+\tau-1)$ in \eqref{eqforI} we have:
      $$I_{\epsilon}(\boldsymbol{\upomega}_{X_0^1},\boldsymbol{\upomega}_{X_0^2})=$$
      $$\prod_{i=1}^2g(\epsilon x_i-\epsilon\boldsymbol{\upomega}_{X_0^i}([(t-r)/\epsilon^2]+\tau-1)-\epsilon (\boldsymbol{\upomega}_{X_0^i}({[(t-r)/\epsilon^2]+\tau-s_i})-\boldsymbol{\upomega}_{X_0^i}([(t-r)/\epsilon^2]+\tau-1))+y_i).$$
    {Recall that $(\omega_X,\omega_Y)$ denotes a pair of paths built from $\boldsymbol{\hat{\pi}}^{(\epsilon)}=\hat{\pi}^{(\epsilon)}\times\hat{\pi}^{(\epsilon)}$, with $(X,Y)\sim W_1\times W_1$ as an initial step. For $s_i\in[0,1]$, from \eqref{eq3.1} and the fact that $(X^i_{[(t-r)/\epsilon^2]})_{i=1,2}$ are independent Brownian motions in $\Omega_1$, independent from the previous steps of the Markov chain, we get that}
    $$(\boldsymbol{\upomega}_{X_0^i}({[(t-r)/\epsilon^2]+\tau-s_i})-\boldsymbol{\upomega}_{X_0^i}([(t-r)/\epsilon^2]+\tau-1))_{i=1,2}=$$
    $$(\boldsymbol{\upomega}_{X_0^i}({[(t-r)/\epsilon^2]+\tau-1+1-s_i})-\boldsymbol{\upomega}_{X_0^i}([(t-r)/\epsilon^2]+\tau-1))_{i=1,2}\overset{d}{=}(\omega_{X^i_{[(t-r)/\epsilon^2]}}({1-s_i}))_{i=1,2},$$
    where the last equality is in distribution. The right hand side for $s_i\in[0,1]$ is equal to $(X^i_{[(t-r)/\epsilon^2]}(1-s_i))_{i=1,2}$  which are  independent from $\boldsymbol{{\omega}}_{X_0^i}({[(t-r)/\epsilon^2]+\tau-1})$, $i=1,2$. Since $(X^1_{[(t-r)/\epsilon^2]},X^2_{[(t-r)/\epsilon^2]})\sim W_1\times W_1$, $I_{\epsilon}(\boldsymbol{\upomega}_{X_0^1},\boldsymbol{\upomega}_{X_0^2})$ is equal in distribution  to:
    $$\prod_{i=1}^2g(\epsilon x_i-\epsilon\boldsymbol{\upomega}_{X^i_0}([(t-r)/\epsilon^2]+\tau-1)-\epsilon \tilde{B}^i_{(1-s_i)}+y_i)$$
    where $(\tilde{B}^i)_{i=1,2}$ are independent Brownian motions, independent from $(\boldsymbol{\upomega}_{X_0^i}({[(t-r)/\epsilon^2]+\tau-1}))_{i=1,2}$. By $\textbf{Lemma \ref{lemm6.1}}$,  proved below, and \textbf{Remark \ref{nikoswillcomplain}}, 
    $$(\epsilon\boldsymbol{\upomega}_{X_0^1}({[(t-r)/\epsilon^2]+\tau-1}),\epsilon\boldsymbol{\upomega}_{X_0^2}({[(t-r)/\epsilon^2]+\tau-1}))$$
    converges in distribution, as $\epsilon\rightarrow0$, to a centered Gaussian random variable with covariance $(t-r)I_{4\times4}$. Therefore 
    $$\E[I_{\epsilon}(\boldsymbol{\upomega}_{X_0^1},\boldsymbol{\upomega}_{X_0^2})]\rightarrow p_{t-r}\star g(y)^2,$$
    as $\epsilon\rightarrow0$ and therefore $\hat{\E}_{B,t/\epsilon^2}[I_{\epsilon}(B^1,B^2)]\rightarrow p_{t-r}\star g(y)^2,$ as $\epsilon\rightarrow0$.\par
    Now we look at $\hat{\E}_{B,t/\epsilon^2}[\exp(J_{\epsilon})]$.  Let $A_\epsilon=[-1,M(\epsilon)/\epsilon^2]^2\symbol{92}[0,[M(\epsilon)/\epsilon^2]-1]^2$. We write
    $$\hat{\E}_{B,t/\epsilon^2}[\exp(J_{\epsilon}(M(\epsilon)/\epsilon^2,M(\epsilon)/\epsilon^2)]=\hat{\E}_{B,t/\epsilon^2}[\exp(\Tilde{J}_{\epsilon}(B^1,B^2)+\mathcal{J}_{\epsilon}(B^1,B^2))]$$
    with
     \begin{align}\label{eq6.3}
    \Tilde{J}_{\epsilon}(B^1,B^2)=\frac{\hat{\beta}^2}{{\log\frac{1}{\epsilon}}}\int_{[0,[M(\epsilon)/\epsilon^2]-1]^2}R(u_1-u_2,x_1-x_2+ \Delta B^1_{(t-r)/\epsilon^1-s_1,(t-r)/\epsilon^2+u_1}-\nonumber \\ -\Delta B^2_{(t-r)/\epsilon^2-s_2,(t-r)/\epsilon^2+u_2})du_1du_2
\end{align}
and
\begin{align}\label{therestadditiveterm}
    \mathcal{J}_{\epsilon}(B^1,B^2)=\frac{\hat{\beta}^2}{{\log\frac{1}{\epsilon}}}\int_{A_\epsilon}R_\phi(u_1,u_2)\psi\star\psi(x_1-x_2+ \Delta B^1_{(t-r)/\epsilon^1-s_1,(t-r)/\epsilon^2+u_1}-\nonumber \\ \Delta B^2_{(t-r)/\epsilon^2-s_2,(t-r)/\epsilon^2+u_2})du_1du_2,
\end{align}
   where we also used the fact that $R_\phi(u_1,u_2)=\phi\star\tilde{\phi}(u_1-u_2)$ when $u_1,u_2\geq0$. By arguing similarly as we did for $\mathcal{R}_\epsilon$,
   defined in \eqref{anothererrorterm}, we can see that $\mathcal{J}_{\epsilon}(B^1,B^2)\rightarrow0$, as $\epsilon\rightarrow0$, uniformly over $B^1,B^2$. Therefore 
   to calculate $\hat{\E}_{B,t/\epsilon^2}[\exp(J_{\epsilon}(M(\epsilon)/\epsilon^2,M(\epsilon)/\epsilon^2)]$ as $\epsilon\rightarrow0$, it is sufficient to
   calculate 
   $$\hat{\E}_{B,t/\epsilon^2}[\exp(\Tilde{J}_{\epsilon}(B^1,B^2)],$$
   as $\epsilon\rightarrow0$. Using \textbf{Lemma \ref{thm:lem4.1}} and arguing similarly as we did in the calculation of $\Hat{\E}_{B,t/\epsilon^2}[I_{\epsilon}]$,  we see that it is enough to calculate the limit of:
   \begin{equation}\label{changeofpaths}
       {\E}[\exp(\Tilde{J}_{\epsilon}(\boldsymbol{\upomega}_{X_0^1},\boldsymbol{\upomega}_{X_0^2}))].
   \end{equation}
   By adding and subtracting $\boldsymbol{\upomega}_{X_0^1}({[(t-r)/\epsilon^2]+\tau})$ and $\boldsymbol{\upomega}_{X_0^2}({[(t-r)/\epsilon^2]+\tau})$ in  \eqref{eq6.3}, the paths in the spatial term inside the integral can be written as
    \begin{equation}\label{eq6.4}
        \sum_{i=1}^2 \biggl(\boldsymbol{\upomega}_{X_0^i}({(t-r)/\epsilon^2+u_1})-\boldsymbol{\upomega}_{X_0^i}({[(t-r)/\epsilon^2]+\tau})\biggr)+\biggr(\boldsymbol{\upomega}_{X_0^i}({[(t-r)/\epsilon^2]+\tau})-\boldsymbol{\upomega}_{X_0^i}({(t-r)/\epsilon^2-s_i})\biggl)
    \end{equation}
   {We deal with the two terms in the above sum, separately. Observe that for $i=1$ and for $i=2$ the corresponding summands are independent, by construction. We prove that
    \begin{enumerate}[label=\textbf{I.\arabic*}]
        \item\label{item1} For $i=1,2$ the term $(\boldsymbol{\upomega}_{X_0^i}({[(t-r)/\epsilon^2]+\tau})-\boldsymbol{\upomega}_{X_0^i}({(t-r)/\epsilon^2-s_i}))_{s_i\in[0,1]}$ is equal in distribution to a Brownian motion $(\tilde{B}^i_{s_i})_{s_i\in[0,1]}$.
        \item\label{item2} For $i=1,2$ the term $(\boldsymbol{\upomega}_{X_0^i}({(t-r)/\epsilon^2+u_1})-\boldsymbol{\upomega}_{X_0^i}({[(t-r)/\epsilon^2]+\tau}))_{u_i\in[0,[M(\epsilon)/\epsilon^2]-1]}$ is equal in distribution to $(\omega_{X^i_{[(t-r)/\epsilon^2]+1}}({u_i}))_{u_i\in[0,[M(\epsilon)/\epsilon^2]-1]}$. 
        \item\label{item3} The terms $$(\boldsymbol{\upomega}_{X_0^i}({[(t-r)/\epsilon^2]+\tau})-\boldsymbol{\upomega}_{X_0^i}({(t-r)/\epsilon^2-s_i}))_{s_i\in[0,1]}$$
        and 
        $$(\boldsymbol{\upomega}_{X_0^i}({(t-r)/\epsilon^2+u_1})-\boldsymbol{\upomega}_{X_0^i}({[(t-r)/\epsilon^2]+\tau}))_{u_i\in[0,[M(\epsilon)/\epsilon^2]-1]}$$ are independent. 
    \end{enumerate}
    Observe that these three items, along with \textbf{Theorem \ref{thm:5.1}} and \textbf{Proposition \ref{cor5.1}}, give us the desired limit for \eqref{changeofpaths}. Indeed,  $\Tilde{J}_{\epsilon}(\boldsymbol{\upomega}_{X_0^1},\boldsymbol{\upomega}_{X_0^2})$ is equal in distribution to
    $$\frac{\hat{\beta}^2}{{\log\frac{1}{\epsilon}}}\int_{[0,[M(\epsilon)/\epsilon^2]-1]^2}R\biggr(u_1-u_2, x_1-x_2+\sum_{i=1}^2 \tilde{B}^i_{s_i}+\omega_{X^1_{[(t-r)/\epsilon^2]+1}}({u_1})-\omega_{X^2_{[(t-r)/\epsilon^2]+1}}({u_2})\biggr)du_1du_2,$$
    where $(\tilde{B}^1_{s_1})_{s_1\in[0,1]}, (\tilde{B}^2_{s_2})_{s_1\in[0,1]}$ are two independent Brownian motions that are also independent from 
    $$(\omega_{X^1_{[(t-r)/\epsilon^2]+1}}({u_1}),\omega_{X^2_{[(t-r)/\epsilon^2]+1}}({u_2}))_{u_1,u_2\in[0,[M(\epsilon)/\epsilon^2]-1]}.$$
    Also, recall that $(X^1_{[(t-r)/\epsilon^2]+1},X^2_{[(t-r)/\epsilon^2]+1})\sim W_1\times W_1$. Now, if $b^1,b^2\in\Omega_1$ are two (generic) Brownian motions, this equality in distribution proves that
   \begin{align}\label{nikoswillcomplain}
         \nonumber\E\biggl[\exp(\tilde{J}_{\epsilon}(\boldsymbol{\upomega}_{X_0^1},\boldsymbol{\upomega}_{X_0^2}))\biggr]=\int_{\mathbb{R}^2}\int_{\mathbb{R}^2}\E\biggl[\exp\biggl(\frac{\hat{\beta}^2}{{\log\frac{1}{\epsilon}}}\int_{[0,[M(\epsilon)/\epsilon^2]-1]^2}R\biggr(u_1-u_2, 
         x_1-x_2+y_1+y_2+\\
         +\omega_{b^1}({u_1})-\omega_{b^2}({u_2})\biggl)du_1du_2\biggr)\biggr]p_{s_1}(y_1)p_{s_2}(y_2)dy_1dy_2.
     \end{align}
    Now we apply  \textbf{Theorem \ref{thm:5.1}} and \textbf{Proposition \ref{cor5.1}} and  get that:
     $$\E\biggl[\exp(\tilde{J}_{\epsilon}(\boldsymbol{\upomega}_{X_0^1},\boldsymbol{\upomega}_{X_0^2}))\biggr]\rightarrow\biggl(1-\frac{\hat{\beta}^2}{\hat{\beta}_c(R)^2}\biggr)^{-1},$$
    as $\epsilon\rightarrow0$. Finally, from our previous remarks, $\Hat{\E}_{B,t/\epsilon^2}[\exp(J_{\epsilon}(M(\epsilon)/\epsilon^2,M(\epsilon)/\epsilon^2))]$ has the same limit.}\par
    It is therefore sufficient to prove items \ref{item1}-\ref{item3}. To prove the item \ref{item1} we add and subtract $\boldsymbol{\upomega}_{X_0^i}([(t-r)/\epsilon^2]+\tau-1)$. We have that
    $$\boldsymbol{\upomega}_{X_0^i}({[(t-r)/\epsilon^2]+\tau})-\boldsymbol{\upomega}_{X_0^i}({[(t-r)/\epsilon^2]+\tau-s_i})=$$
    $$\biggr(\boldsymbol{\upomega}_{X_0^i}({[(t-r)/\epsilon^2]+\tau})-\boldsymbol{\upomega}_{X_0^i}([(t-r)/\epsilon^2]+\tau-1)\biggl)-\biggr(\boldsymbol{\upomega}_{X_0^i}({[(t-r)/\epsilon^2]+\tau-s_i})-\boldsymbol{\upomega}_{X_0^i}([(t-r)/\epsilon^2]+\tau-1)\biggl).$$ 
    Recall that $X^i_{[(t-r)/\epsilon^2]}$ is sampled from the Wiener measure $W_1$, independently from the previous steps. From this observation and \eqref{eq3.1}, the right-hand side, for $s_i\in[0,1]$, is equal in distribution to:
    \begin{equation}\label{eq6.5}
        \omega_{X^i_{[(t-r)/\epsilon^2]}}(1)-\omega_{X^i_{[(t-r)/\epsilon^2]}}(1-s_i).
    \end{equation}
    Now for $s_i\in[0,1]$ this is equal to $X^i_{[(t-r)/\epsilon^2]}(1)-{X^i_{[(t-r)/\epsilon^2]}}(1-s_i)$ which is distributed like a Brownian motion at $1-(1-s_i)=s_i$ independent from ${X^i_{[(t-r)/\epsilon^2]}}(1-s_i)$. This proves item \ref{item1}. \par
   {Now we focus on 
    \begin{equation}\label{mainpathcontr}
        (\boldsymbol{\upomega}_{X_0^i}({[(t-r)/\epsilon^2]+\tau+u_i})-\boldsymbol{\upomega}_{X_0^i}({[(t-r)/\epsilon^2]+\tau}))_{u_i\in[0,[M(\epsilon)/\epsilon^2]-1]}.
    \end{equation}
    Since $X^i_{[(t-r)/\epsilon^2]+1}$ is sampled from $W_1$, independently from the previous steps,  \eqref{mainpathcontr} is independent from $(\boldsymbol{\upomega}_{X_0^i}({[(t-r)/\epsilon^2]+\tau})-\boldsymbol{\upomega}_{X_0^i}({[(t-r)/\epsilon^2]+\tau-s_i}))_{s_i\in[0,1]}$. This proves item \ref{item3}.  Furthermore, since $M(\epsilon)\leq r$ we have
    $$\frac{t-r}{\epsilon^2}+[M(\epsilon)/\epsilon^2]-1\leq\tau+N_{\epsilon},$$
    This means that for $u_i\in[0,[M(\epsilon)/\epsilon^2]-1]$, the step $X_{N_{\epsilon}+1}^i$ (which is sampled using $\hat{\pi}^\epsilon_{N_{\epsilon},N_{\epsilon}+1}$, defined in \eqref{finaledgetran}) does not appear in the path 
    $$\boldsymbol{\upomega}_{X_0^i}({[(t-r)/\epsilon^2]+\tau+u_i})-\boldsymbol{\upomega}_{X_0^i}({[(t-r)/\epsilon^2]+\tau}))_{u_i\in[0,[M(\epsilon)/\epsilon^2]-1]}.$$
    Using \eqref{eq3.1}, this implies that \eqref{mainpathcontr} is equal in distribution to $(\omega_{X^i_{[(t-r)/\epsilon^2]+1}}({u_i}))_{u_i\in[0,[M(\epsilon)/\epsilon^2]-1]}$ (which is built using only $\hat{\pi}^{(\epsilon)}$). This proves item \ref{item2} and concludes the proof.}      
    \end{proof}

    The lemma that we used in the calculation of the limit of $\Hat{\E}_{B,t/\epsilon^2}[I_{\epsilon}]$ is the following:

\begin{mylem}\label{lemm6.1}
For $t>0$ and for $(X_0,Y_0)\sim W_1\times  W_1$ and $({\omega}_{X_0},\omega_{Y_0})$ built from $\boldsymbol{\hat{\pi}}^{(\epsilon)}$ we have that
$$\biggr(\epsilon\omega_{X_0}([t/\epsilon^2]),\epsilon\omega_{X_0}([t/\epsilon^2])\biggl)$$
converges in distribution to $(\mathcal{N}_1,\mathcal{N}_2)$ where $\mathcal{N}_i\in\mathbb{R}^2$ are two independent centered Gaussian random variables with covariance matrix $tI_{2\times2}$.
\end{mylem}

\begin{proof}
    Since $({\omega}_{X_0},\omega_{Y_0})$ is a pair of independent paths each built from $\hat{\pi}^{(\epsilon)}$ it is enough to prove this for $\omega_{X_0}$. Define:
    $$N_t^{\epsilon}=\max\{i/\textbf{ }T_i<[t/\epsilon^2]\}-1.$$
       From similar arguments as in $\textbf{Proposition \ref{thm:prop4.2}}$ and \textbf{\ref{prop4.4}}, the random variable 
    $$\epsilon^{}\sum_{k=0}^{[t/\epsilon^2]}\mathbf{X}_k^{(\epsilon)}$$
    converges in distribution, as $\epsilon\rightarrow0$, to a centered Gaussian random variable with covariance matrix $\frac{t}{\gamma}I_{2\times2}$. Furthermore, since $T_n/n\rightarrow1/\gamma$ a.s as $n\rightarrow\infty$ we see that:
    \begin{equation}\label{eq6.6}
        \epsilon^{}\sum_{k=0}^{N_t^{\epsilon}}\mathbf{X}_k^{(\epsilon)}
    \end{equation}
    converges in distribution to a centered Gaussian random variable with covariance matrix $tI_{2\times2}$. Finally,  from \textbf{Corollary \ref{thm:cor4.1}} it is easy to see that:
    \begin{equation}\label{eq6.7}
        \epsilon|\omega_{X_0}([t/\epsilon^2])-\sum_{k=0}^{N_t^{\epsilon}}\mathbf{X}_k^{(\epsilon)}|\lesssim\sup_{j=1,...,N_{t}^\epsilon}\sup_{s\in[T_{j-1},T_{j}]}\epsilon|\omega_{X_0}(s)-\omega_{X_0}(T_{j})|\rightarrow0,
    \end{equation}
    a.s. as $\epsilon\rightarrow0$. This completes the proof. 
\end{proof}

\begin{remark}
     Observe that this implies that $\epsilon\boldsymbol{\upomega}_{X_0}([t/\epsilon^2])$ (where $X_0\sim W_{\tau}$, $\tau<1$) converges in distribution to a centered Gaussian random variable with covariance $tI_{2\times2}$. Indeed,  observe that $\epsilon\boldsymbol{\upomega}_{X_0}(N_t^{\epsilon})$  is equal to $(\ref{eq6.6})$ where only $\mathbf{X}_0^{(\epsilon)}$ has a different distribution from $\mathbf{X}_k^{(\epsilon)}$, $k\geq1$. Then by following the same arguments as in the proof above, we see that $(\ref{eq6.7})$ is still valid for $\boldsymbol{\upomega}_{X_0}$. So it converges to a Gaussian random variable with covariance matrix $tI_{2\times2}$. Finally if $(\boldsymbol{\upomega}_{X_0},\boldsymbol{\upomega}_{Y_0})$ are as in the proof of \textbf{Lemma \ref{lem:prep1}} they are independent and therefore $(\epsilon\boldsymbol{\upomega}_{X_0}([t/\epsilon^2]),\epsilon\boldsymbol{\upomega}_{X_0}([t/\epsilon^2]))$ converges in distribution to  $(\mathcal{N}_1,\mathcal{N}_2)$.\par
\end{remark}

     Now we show that ${I}_{\epsilon}$ and $\exp({J}_{\epsilon})$ decouple asymptotically. This is done in the proof of the following lemma:

     \begin{mylem}\label{lem:prep2}
         For all $\hat{\beta}<\hat{\beta}_c(R)$  we have that
         $$\Hat{\E}_{B,t/\epsilon^2}\biggl[I_{\epsilon}\exp(J_{\epsilon})\biggr]\rightarrow\biggl(1-\frac{\hat{\beta}^2}{\hat{\beta}_c(R)^2}\biggr)^{-1}p_{t-r}\star g(y)^2,$$
         as $\epsilon\rightarrow0$.
     \end{mylem}

     \begin{proof}
         Using the same notation as in the proof of \textbf{Lemma \ref{lem:prep1}}, we have that
         $$\Hat{\E}_{B,t/\epsilon^2}\biggl[I_{\epsilon}\exp(J_{\epsilon})\biggr]=\Hat{\E}_{B,t/\epsilon^2}\biggl[I_{\epsilon}(B^1,B^2)\exp(\Tilde{J}_{\epsilon}(B^1,B^2)+\mathcal{J}_{\epsilon}(B^1,B^2))\biggr],$$
         where $\Tilde{J}_{\epsilon}(B^1,B^2)$ and $\mathcal{J}_{\epsilon}(B^1,B^2)$ are defined in \eqref{eq6.3} and \eqref{therestadditiveterm}, respectively. Since $\mathcal{J}_{\epsilon}(B^1,B^2)\rightarrow0$, as $\epsilon\rightarrow0$, uniformly over the paths $B^1,B^2$, it is sufficient to consider 
         $$\Hat{\E}_{B,t/\epsilon^2}\biggl[I_{\epsilon}(B^1,B^2)\exp(\Tilde{J}_{\epsilon}(B^1,B^2))\biggr],$$
         and calculate its limit as $\epsilon\rightarrow0$.\par       
         From \eqref{Markovcon} this expectation is equal to
         $${\E}\biggl[\Gcal_{\epsilon}(\tilde{X}^1_{N_{\epsilon}},\tilde{X}^2_{N_{\epsilon}})I_{\epsilon}([\tilde{X}_0^1,...,\tilde{X}^1_{N_{\epsilon}+1}],[\tilde{X}_0^2,...,\tilde{X}^2_{N_{\epsilon}+1}])\exp(\Tilde{J}_{\epsilon}([\tilde{X}_0^1,...,\tilde{X}^1_{N_{\epsilon}+1}],[\tilde{X}_0^2,...,\tilde{X}^2_{N_{\epsilon}+1}]))\biggr],$$
         where $(\tilde{X}_0^1,...,\tilde{X}^1_{N_{\epsilon}+1})$ and $(\tilde{X}_0^2,...,\tilde{X}^2_{N_{\epsilon}+1})$ are as in the proof of \textbf{Lemma \ref{lem:prep1}}. Combining the fact that  $\Gcal_{\epsilon}$ converges to $1$ uniformly and  \textbf{Lemma \ref{thm:lem4.1}} we see that there are constants $\mathcal{A}_{\epsilon},\mathcal{B}_{\epsilon}$ such that:
         $$\mathcal{A}_{\epsilon}{\E}\biggl[I_{\epsilon}(\boldsymbol{\upomega}_{X_0^1},\boldsymbol{\upomega}_{X_0^2})\exp(\Tilde{J}_{\epsilon}(\boldsymbol{\upomega}_{X_0^1},\boldsymbol{\upomega}_{X_0^2}))\biggr]$$
         $$\leq\Hat{\E}_{B,t/\epsilon^2}\biggl[I_{\epsilon}\exp(J_{\epsilon})\biggr]\leq$$
         \begin{equation}\label{inequality}
             \mathcal{B}_{\epsilon}{\E}\biggl[I_{\epsilon}(\boldsymbol{\upomega}_{X_0^1},\boldsymbol{\upomega}_{X_0^2})\exp(\Tilde{J}_{\epsilon}(\boldsymbol{\upomega}_{X_0^1},\boldsymbol{\upomega}_{X_0^2}))\biggr],
         \end{equation}
         where $\boldsymbol{\upomega}_{X_0^1},\boldsymbol{\upomega}_{X_0^2}$ are as in the proof of \textbf{Lemma \ref{lem:prep1}} and $\mathcal{A}_{\epsilon},\mathcal{B}_{\epsilon}\rightarrow1$, as $\epsilon\rightarrow0$.\par
         The terms $I_{\epsilon}(\boldsymbol{\upomega}_{X_0^1},\boldsymbol{\upomega}_{X_0^2})$ and $\Tilde{J}_{\epsilon}(\boldsymbol{\upomega}_{X_0^1},\boldsymbol{\upomega}_{X_0^2})$ are  independent. More specifically,  the term (\ref{eq6.4}), that appears in $\Tilde{J}_{\epsilon}(\boldsymbol{\upomega}_{X_0^1},\boldsymbol{\upomega}_{X_0^2})$, is independent from $I_{\epsilon}(\boldsymbol{\upomega}_{X_0^1},\boldsymbol{\upomega}_{X_0^2})$. Indeed, observe that $\boldsymbol{\upomega}_{X_0^i}({[(t-r)/\epsilon^2]+\tau+u_1})-\boldsymbol{\upomega}_{X_0^i}({[(t-r)/\epsilon^2]+\tau})$, $i=1,2$  is independent from the paths $\boldsymbol{\upomega}_{X_0^1},\boldsymbol{\upomega}_{X_0^2}$ at all times before $[(t-r)/\epsilon^2]+\tau$, and specifically from all times that appear in ${I}_{\epsilon}(\boldsymbol{\upomega}_{X_0^1},\boldsymbol{\upomega}_{X_0^2})$. The other terms $\boldsymbol{\upomega}_{X_0^i}({[(t-r)/\epsilon^2]+\tau})-\boldsymbol{\upomega}_{X_0^i}{([(t-r)/\epsilon^2]+\tau-s_i})$ as mentioned in the comments after $(\ref{eq6.5})$ are independent from $\boldsymbol{\upomega}_{X^i_0}([(t-r)/\epsilon^2]+\tau-1)$ and from $\boldsymbol{\upomega}_{X^i_0}({[(t-r)/\epsilon^2]+\tau-s_i})-\boldsymbol{\upomega}_{X^i_0}([(t-r)/\epsilon^2]+\tau-1)$ and therefore from $I_{\epsilon}(\boldsymbol{\upomega}_{X_0^1},\boldsymbol{\upomega}_{X_0^2})$ as well. Therefore, ${I}_{\epsilon}(\boldsymbol{\upomega}_{X_0^1},\boldsymbol{\upomega}_{X_0^2})$ and $\exp(\tilde{J}_{\epsilon}(\boldsymbol{\upomega}_{X_0^1},\boldsymbol{\upomega}_{X_0^2}))$ are indeed independent.\par
         Finally, from the proof of  \textbf{Lemma \ref{lem:prep1}} 
         $${\E}\biggl[I_{\epsilon}(\boldsymbol{\upomega}_{X_0^1},\boldsymbol{\upomega}_{X_0^2})\biggr]\rightarrow p_{t-r}\star g(y)^2$$
        and
        $${\E}\biggl[\exp(\Tilde{J}_{\epsilon}(\boldsymbol{\upomega}_{X_0^1},\boldsymbol{\upomega}_{X_0^2}))\biggr]\rightarrow\biggl(1-\frac{\hat{\beta}^2}{\hat{\beta}_c(R)^2}\biggr)^{-1},$$
         as $\epsilon\rightarrow0$. This, combined with \eqref{inequality}, concludes the proof.
     \end{proof}

     Now \textbf{Lemma \ref{lem:prep1}} and \textbf{Lemma \ref{lem:prep2}} combined gives us the following:

     \begin{mylem}\label{lem:limit}
     As $\epsilon\rightarrow0$ and for all $\hat{\beta}<\hat{\beta}_c(R)$:
         $$\Fcal_{\epsilon}(r,y,M(\epsilon)/\epsilon^2,M(\epsilon)/\epsilon^2)\rightarrow v^2_{eff}(\hat{\beta})p_{t-r}\star g(y)^2,$$
     where
     $$v_{eff}^2(\hat{\beta})=\biggl(1-\frac{\hat{\beta}^2}{\hat{\beta}_c(R)^2}\biggr)^{-1}\cdot||R||_1.$$
     \end{mylem}

     \begin{proof}
     Recall the definition of $\Fcal_{\epsilon}(r,y, M(\epsilon)/\epsilon^2,M(\epsilon)/\epsilon^2)$ in \eqref{eq6.1} and of $J_{\epsilon}(M(\epsilon)/\epsilon^2,M(\epsilon)/\epsilon^2;x_1, x_2, s_1, s_2)$ in \eqref{6.2}, where we emphasize the dependence of $J_\epsilon$ on $x_1,x_2,s_1,s_2$. From \textbf{Lemma \ref{lem:prep2}}, we see that if 
     \begin{equation}\label{uniformbound}
         \Hat{\E}_{B,t/\epsilon^2}[I_{\epsilon}e^{J_{\epsilon}(M(\epsilon)/\epsilon^2,M(\epsilon)/\epsilon^2;x_1,x_2,s_1,s_2)}]\lesssim1,
     \end{equation}
     for all sufficiently small $\epsilon$ and uniformly in $x_1,x_2,s_1,s_2$ then, from the dominated convergence theorem\footnote{Recall that $\phi$ and $\psi$ are compactly supported.}, $\Fcal_{\epsilon}(r,y)\rightarrow v^2_{eff}(\hat{\beta})p_{t-r}\star g(y)^2$ as $\epsilon\rightarrow0$.\par
     {To prove \eqref{uniformbound}, we start from \eqref{inequality}, to obtain
     $$\Hat{\E}_{B,t/\epsilon^2}[I_{\epsilon}e^{J_{\epsilon}(M(\epsilon)/\epsilon^2,M(\epsilon)/\epsilon^2;x_1,x_2,s_1,s_2)}]\lesssim{\E}\biggl[I_{\epsilon}(\boldsymbol{\upomega}_{X_0^1},\boldsymbol{\upomega}_{X_0^2})\exp(\Tilde{J}_{\epsilon}(\boldsymbol{\upomega}_{X_0^1},\boldsymbol{\upomega}_{X_0^2};x_1,x_2,s_1,s_2))\biggr].$$
     Since $g$ has compact support we see that
     $$\E\biggl[\tilde{I}_{\epsilon}(\boldsymbol{\upomega}_{X_0^1},\boldsymbol{\upomega}_{X_0^2})\exp(\Tilde{J}_{\epsilon}(\boldsymbol{\upomega}_{X_0^1},\boldsymbol{\upomega}_{X_0^2};x_1,x_2,s_1,s_2))\biggr]\lesssim\E\biggl[\exp(\tilde{J}_{\epsilon}(\boldsymbol{\upomega}_{X_0^1},\boldsymbol{\upomega}_{X_0^2};x_1,x_2,s_1,s_2))\biggr].$$
     Recall that, from \eqref{nikoswillcomplain} 
    $$\E\biggl[\exp(\tilde{J}_{\epsilon}(\boldsymbol{\upomega}_{X_0^1},\boldsymbol{\upomega}_{X_0^2};x_1,x_2,s_1,s_2))\biggr]=$$
     \begin{align}\label{repeateq}
         \nonumber\int_{\mathbb{R}^2}\int_{\mathbb{R}^2}\E\biggl[\exp\biggl(\frac{\beta^2}{{\log\frac{1}{\epsilon}}}\int_{[0,[M(\epsilon)/\epsilon^2]-1]^2}R\biggr(u_1-u_2, 
         x_1-x_2+ y_1+y_2+ \omega_{b^1}({u_1})- \omega_{b^2}({u_2})\biggl)du_1du_2\biggr)\biggr]\cdot\\ \cdot p_{s_1}(y_1)p_{s_2}(y_2)dy_1dy_2,      
     \end{align}
     where $b^1,b^2$ are two independent Brownian motions in $\Omega_1$.}\par
     {Consider the term
     $$\E\biggl[\exp\biggl(\frac{\hat{\beta}^2}{{\log\frac{1}{\epsilon}}}\int_{[0,[M(\epsilon)/\epsilon^2]-1]^2}R\biggr(u_1-u_2, 
         x_1-x_2+y_1+y_2+\\
         \omega_{b^1}({u_1})-\omega_{b^2}({u_2})\biggl)du_1du_2\biggr)\biggr].$$
    Then, from \textbf{Proposition \ref{cor5.1}}, we have
    $$\E\biggl[\exp\biggl(\frac{\hat{\beta}^2}{{\log\frac{1}{\epsilon}}}\int_{[0,[M(\epsilon)/\epsilon^2]-1]^2}R\biggr(u_1-u_2, 
         x_1-x_2+y_1+y_2+\\
         \omega_{b^1}({u_1})-\omega_{b^2}({u_2})\biggl)du_1du_2\biggr)\biggr]\lesssim1,$$
   uniformly over $x_1,x_2,y_1,y_2$. From \eqref{repeateq} we get that for all $\epsilon$ small enough,
   $$\E[\exp(\tilde{J}_{\epsilon}(\boldsymbol{\upomega}_{X_0^1},\boldsymbol{\upomega}_{X_0^2};x_1,x_2,y_1,y_2,s_1,s_2))]\lesssim1,$$
   uniformly over $x_1,x_2,y_1,y_2,s_1,s_2$, which in turn proves the uniform bound \eqref{uniformbound}. This concludes the proof.}
     \end{proof}

     Finally, we prove the uniform bound (\ref{eq:2.12}):

     \begin{mylem}\label{lem:bound}
         Let $\hat{\beta}<\hat{\beta}_c(R)$ and $k>0$. Then 
         $$|\Fcal_{\epsilon}(r,y,M_1(\epsilon)/\epsilon^2,M_2(\epsilon)/\epsilon^2)|\lesssim(1\wedge|y|^{-k}),$$
         where the implied constant depends only on k.
     \end{mylem}

     \begin{proof}
      From the proof of  \textbf{Lemma \ref{lem:limit}} and the observations at the start of this section, it is easy to infer that for any $p>1$ such that $p\hat{\beta}<\hat{\beta}_{c}(R)$ 
      $$\sup_{\epsilon\in(0,1)}\Hat{\E}_{B,t/\epsilon^2}[e^{p J_{\epsilon}((M_1(\epsilon)/\epsilon^2,M_2(\epsilon)/\epsilon^2;x_1,x_2,s_1,s_2)}]\lesssim1,$$
      Hence, from Hölder's inequality
     $$\Fcal_{\epsilon}(r,y,M_1(\epsilon)/\epsilon^2,M_2(\epsilon)/\epsilon^2)\lesssim\int_{\mathbb{R}^{4}}\int_{[0,1]^2}\Hat{\E}_{B,t/\epsilon^2}[I_{\epsilon}^q]^{1/q}\prod_{i=1}^2\phi(s_i)\psi(x_i)d\Bar{s}d\Bar{x},$$
     where $1/p+1/q=1$. This implies that:
     $$\Fcal_{\epsilon}(r,y,M_1(\epsilon)/\epsilon^2,M_2(\epsilon)/\epsilon^2)\lesssim\int_{\mathbb{R}^{2}}\int_{[0,1]}\Hat{\E}_{B,t/\epsilon^2}[g(\epsilon x-\epsilon B_{(t-r)/\epsilon^2-s}+y)^q]^{1/q}\phi(s)\psi(x)d{s}d{x},$$
     since $g$ is compactly supported. Observe that for all $k$:
     $$\Hat{\E}_{B,t/\epsilon^2}[\textbf{1}_{|\epsilon B_{(t-r)/\epsilon^2-s}|>M}]\lesssim\frac{1}{M^{2k}},$$
     where the implied constant depends on $k$. The proof of this inequality is the same as in \cite{Gu_2018} (look at \textbf{Lemma 5.3} of the same paper) only we use \textbf{Corollary \ref{thm:cor4.1}} whenever they use $\textbf{Lemma A.2}$. With this inequality and since $g$ and compactly supported:
     $$\Hat{\E}_{B,t/\epsilon^2}[g(\epsilon x-\epsilon B_{(t-r)/\epsilon^2-s}+y)^q]^{1/q}\lesssim1\wedge\frac{1}{|y|^k},$$ 
     which  implies $(\ref{eq:2.12})$.
     \end{proof}


As mentioned, \textbf{Lemma \ref{lem:limit}} and \textbf{Lemma \ref{lem:bound}} prove \textbf{Proposition \ref{thm:prop2.4}}. We also have the following corollary:

\begin{mycor}\label{cor6.1}
    For all $\hat{\beta}<\hat{\beta}_c(R)$ we have that 
    $$\mathbf{Var}\biggl(\sqrt{\log\frac{1}{\epsilon}}\int_{\mathbb{R}^2}e^{-\zeta_{t/\epsilon^2}^{(\epsilon)}}(u_{\epsilon}(t,x)-\mathbf{E}[u_{\epsilon}(t,x)])g(x)dx\biggr)$$
    converges, as $\epsilon\rightarrow0$, to
    $$\mathbf{Var}\biggl(\int_{\mathbb{R}^2}\mathcal{U}(t,x)g(x)dx\biggr),$$
    where $\mathcal{U}$ is the solution to the Edwards-Wilkinson equation \eqref{eq:1.2}, with effective diffusivity equal to $I_{2\times2}$ and effective variance $v_{eff}$ defined as in \eqref{eq:effvar}.
\end{mycor}

\begin{proof}
   We have
    $$\hat{\beta}^2v^2_{eff}(\hat{\beta})\int_0^t\int_{\mathbb{R}^2}p_{t-r}\star g(y)^2dydr=\mathbf{Var}\biggl(\int_{\mathbb{R}^2}\mathcal{U}(t,x)g(x)dx\biggr).$$
    We want to apply \textbf{Proposition \ref{thm:prop2.4}} to the formula for the variance that we got from \textbf{Proposition \ref{thm:prop2.3}}. As mentioned in the remark after \textbf{Proposition \ref{thm:prop2.3}}, that formula has some modifications.   Recall that $s_1,s_2$ in \eqref{eq6.1} will lie in $[0,(t-r)/\epsilon^2]$ and $u_1,u_2$ will lie in $[-(t-r)/\epsilon^2,r/\epsilon^2]^2$.  Let $\boldsymbol{\Fcal}_{\epsilon}(r,y,r/\epsilon^2,r/\epsilon
    ^2)$ be defined by $(\ref{eq6.1})$ but with these modifications. Then we have that
    \begin{equation}\label{varianceeq}
        \mathbf{Var}\biggl(\sqrt{\log\frac{1}{\epsilon}}\int_{\mathbb{R}^2}e^{-\zeta_{t/\epsilon^2}^{(\epsilon)}}(u_{\epsilon}(t,x)-\mathbf{E}[u_{\epsilon}(t,x)])g(x)dx\biggr)=\hat{\beta}\int_{0}^t\int_{\mathbb{R}^{2}}\boldsymbol{\Fcal}_{\epsilon}(r,y,r/\epsilon^2,r/\epsilon^2)dydr.
    \end{equation}
    As argued in \textbf{Section \ref{sec:2}} the differences between $\boldsymbol{\Fcal}_{\epsilon}(r,y,r/\epsilon^2,r/\epsilon
    ^2)$ and ${\Fcal}_{\epsilon}(r,y,r/\epsilon^2,r/\epsilon
    ^2)$ come up only when $t-r\leq\epsilon^2$. Therefore, for fixed $r,y$
    $$\boldsymbol{\Fcal}_{\epsilon}(r,y,r/\epsilon^2,r/\epsilon^2)=\Fcal_{\epsilon}(r,y,r/\epsilon^2,r/\epsilon^2),$$
    for all $\epsilon$ small enough. This proves that  $\boldsymbol{\Fcal}_{\epsilon}(r,y,r/\epsilon^2,r/\epsilon^2)\rightarrow v^2_{eff}(\hat{\beta})p_{t-r}\star g(y)^2$, as $\epsilon\rightarrow0$. From this observation and \eqref{varianceeq}, we see that, to prove the corollary, we need to bound $\boldsymbol{\Fcal}_{\epsilon}(r,y,r/\epsilon^2,r/\epsilon^2)$ by an integrable function so we can apply the dominated convergence theorem.  \par 
    More specifically, we argue that $\boldsymbol{\Fcal}_{\epsilon}(r,y,r/\epsilon^2,r/\epsilon^2)$ satisfies the uniform bound \eqref{eq:2.12}. Indeed, observe that if $\textbf{J}_{\epsilon}(r/\epsilon^2,r/\epsilon^2;x_1,x_2,s_1,s_2)$ is defined by \eqref{6.2} but with the range of $u_1,u_2$ changed to $[-(t-r)/\epsilon^2,r/\epsilon^2]$ then, since $r\in(0,t)$ we have 
    $$\textbf{J}_{\epsilon}(r/\epsilon^2,r/\epsilon^2;x_1,x_2,s_1,s_2)\leq{J}'_{\epsilon}(t/\epsilon^2,t/\epsilon^2;x_1,x_2,s_1,s_2),$$
    for all $\epsilon$, where ${J}'_{\epsilon}(t/\epsilon^2,t/\epsilon^2;x_1,x_2,s_1,s_2)$ is also defined by \eqref{6.2} but with the range of $u_1,u_2$ changed to $[-t/\epsilon^2,t/\epsilon^2]$. For all $\epsilon$ small enough ${J}'_{\epsilon}(t/\epsilon^2,t/\epsilon^2;x_1,x_2,s_1,s_2)={J}_{\epsilon}(t/\epsilon^2,t/\epsilon^2;x_1,x_2,s_1,s_2)$, since $R_{\phi}(u_1,u_2)=0$ when $u_1\leq-1$ or $u_2\leq-1$. This proves that for $p>1$ such that $p\hat{\beta}<\hat{\beta}_c(R)$ we have
    $$\Hat{\E}_{B,t/\epsilon^2}[e^{p\textbf{J}_{\epsilon}(r/\epsilon^2,r/\epsilon^2;x_1,x_2,s_1,s_2)}]\lesssim\Hat{\E}_{B,t/\epsilon^2}[e^{p{J}_{\epsilon}(t/\epsilon^2,t/\epsilon^2;x_1,x_2,s_1,s_2)}]\lesssim1,$$
    uniformly over $x_1,x_2,s_1,s_2$ and $r$.  From this estimate, and since $g,\phi$ are compactly supported, we get that
    $$\boldsymbol{\Fcal}_{\epsilon}(r,y,r/\epsilon^2,r/\epsilon^2)\lesssim\int_{\mathbb{R}^{2}}\int_{[0,(t-r)/\epsilon^2]^2}\Hat{\E}_{B,t/\epsilon^2}[g(\epsilon x-\epsilon B^i_{(t-r)/\epsilon^2-s}+y)^q]^{1/q}\phi(s)\psi(x)d{s}d{x},$$
    where $1/p+1/q=1$. Now, by arguing in the same way as we did in the proof of \textbf{Lemma \ref{lem:bound}} we can prove that the uniform bound \eqref{eq:2.12} is satisfied by $\boldsymbol{\Fcal}_{\epsilon}(r,y,r/\epsilon^2,r/\epsilon^2)$, for all $\epsilon$ small enough. This concludes the proof.
    
\end{proof}

\appendix

\section{Estimates on the Total Path Increments}
In this appendix, we study the properties of the path increments between regeneration times. The main result is a local central limit theorem for sums of the total path increments that makes it possible to estimate the limiting distribution of additive functionals of $\omega_{X_0}(T_k)-\omega_{Y_0}(T_k)$. The latter is done in the next appendix.

\begin{myprop}\label{thm:prop4.2}
    The random variable $\mathbf{X}_1^{(\epsilon)}-\mathbf{Y}^{(\epsilon)}_1$ has zero mean. Moreover, as $\epsilon\rightarrow0$, $\mathbf{X}^{(\epsilon)}_1-\mathbf{Y}^{(\epsilon)}_1$ converges in total variation to $\sum_{j=0}^{\theta-1}\mathcal{N}_j$ where:
    \begin{itemize}
        \item $(\mathcal{N}_j)_{j\in\mathbb{N}_0}$ are i.i.d. Gaussian random variables with covariance matrix $2\cdot I_{2\times2}$.
        \item $\theta\sim Geo(\gamma)$ and is independent from the $(\mathcal{N}_j)_{j\in\mathbb{N}}$.
    \end{itemize}
     Moreover, $\mathbf{X}^{(\epsilon)}_1-\mathbf{Y}^{(\epsilon)}_1$ has exponential moments uniformly in $\epsilon$. More specifically,  we have 
     \begin{equation}\label{expbound}
         \limsup_{\epsilon\rightarrow0}\E[\exp(\langle\boldsymbol{\lambda},\mathbf{X}^{(\epsilon)}_1-\mathbf{Y}^{(\epsilon)}_1\rangle)]\lesssim1,
     \end{equation}
     for all $\boldsymbol{\lambda}\in\R^2$, small enough.  Finally, the covariance matrix of $\mathbf{X}^{(\epsilon)}_1-\mathbf{Y}^{(\epsilon)}_1$ is bounded in $\epsilon$ and  converges to $\frac{2}{\gamma}I_{2\times2}$ as $\epsilon\rightarrow0$.
\end{myprop}

\begin{proof}
     By symmetry, $\mathbf{X}^{(\epsilon)}_1-\mathbf{Y}^{(\epsilon)}_1$  has zero mean. In fact, it is easy to see that $\mathbf{X}^{(\epsilon)}_1-\mathbf{Y}^{(\epsilon)}_1$ is a symmetric random variable.\par
    { First, we show that $\mathbf{X}^{(\epsilon)}_1-\mathbf{Y}^{(\epsilon)}_1$ converges to a Gaussian. Inequality \eqref{eq4.11} implies that
    \begin{equation}\label{mixingest}
        \sum_{N=1}^{\infty}\gamma(1-\gamma)^{N-1}\mathcal{A}_{\epsilon}^N\P\biggl(\sum_{j=0}^{N-1}\mathcal{N}_j\in A\biggr)\leq\P(\mathbf{X}^{(\epsilon)}_1-\mathbf{Y}^{(\epsilon)}_1\in A)\leq\sum_{N=1}^{\infty}\gamma(1-\gamma)^{N-1}\mathcal{B}_{\epsilon}^N\P\biggl(\sum_{j=0}^{N-1}\mathcal{N}_j\in A\biggr),
    \end{equation}
    where $\mathcal{N}_j$ are i.i.d. Gaussian random variables with covariance matrix $2\cdot I_{2\times2}$, and where $\mathcal{A}_\epsilon\textit{, }\mathcal{B}_\epsilon\rightarrow1$ as $\epsilon\rightarrow0$.
     From dominated convergence
    $$\lim_{\epsilon\rightarrow0}\P\biggl(\mathbf{X}^{(\epsilon)}_1-\mathbf{Y}^{(\epsilon)}_1\in A\biggr)= \P\biggl(\sum_{j=0}^{\theta-1}\mathcal{N}_j\in A\biggr).$$
    We prove that this convergence happens uniformly over $A$. Observe that \eqref{mixingest} implies that for $\mathcal{C}_\epsilon(N)=(\mathcal{B}_\epsilon^N-1)\vee(1-\mathcal{A}_\epsilon^N)$\footnote{Observe that $\mathcal{A}_{\epsilon}\leq1\leq\mathcal{B}_\epsilon$.} and for a generic $\theta\sim Geo(\gamma)$, independent from $(\mathcal{N}_j)_{j\in\mathbb{N}}$, we have
    $$\biggl|\P(\mathbf{X}^{(\epsilon)}_1-\mathbf{Y}^{(\epsilon)}_1\in A)-\P\biggl(\sum_{j=0}^{\theta-1}\mathcal{N}_j\in A\biggr)\biggr|\leq\sum_{N=1}^{\infty}\gamma(1-\gamma)^{N-1}\mathcal{C}_\epsilon(N)\P\biggl(\sum_{j=0}^{N-1}\mathcal{N}_j\in A\biggr).$$
   For all Borel sets $A\subseteq\mathbb{R}^2$, the right-hand side of the above inequality is bounded above by
   $$\sum_{N=1}^{\infty}\gamma(1-\gamma)^{N-1}\mathcal{C}_\epsilon(N),$$
   which converges to $0$ as $\epsilon\rightarrow0$. Therefore, $\mathbf{X}^{(\epsilon)}_1-\mathbf{Y}^{(\epsilon)}_1$ converges in total variation to $\sum_{j=0}^{\theta-1}\mathcal{N}_j$.}\par
    We use \eqref{eq4.11} to show \eqref{expbound}. We choose $F(\mathbf{x})=\exp(\langle\boldsymbol{\lambda},\mathbf{x}\rangle/c)$. Then, from \eqref{eq4.11}, there is a constant $\mathcal{B}_{\epsilon}$:
    $$\E[F(\mathbf{X}^{(\epsilon)}_1-\mathbf{Y}^{(\epsilon)}_1)]\leq\sum_{N=1}^{\infty}\mathcal{B}_{\epsilon}^N\gamma(1-\gamma)^{N-1}\E_{\mathbf{X}^{(N)},\mathbf{Y}^{(N)}}[F(\mathbf{X}^{(N)}-\mathbf{Y}^{(N)})],$$
    where $(\mathbf{X}^{(N)},\mathbf{Y}^{(N)})$ are as in \eqref{eq4.11}. Since $\mathbf{X}^{(N)},\mathbf{Y}^{(N)}$ are Gaussian, we get:
    $$\E[F(\mathbf{X}^{(\epsilon)}_1-\mathbf{Y}^{(\epsilon)}_1)]\leq\sum_{N=1}^{\infty}\mathcal{B}_{\epsilon}^N\gamma(1-\gamma)^{N-1}e^{N|\boldsymbol{\lambda}|^2},$$
    which is finite for all $\epsilon>0$, for  $\boldsymbol{\lambda}$ small enough (recall that $\mathcal{B}_{\epsilon}\rightarrow1$ as $\epsilon\rightarrow0$).\par
    The last part of the proposition follows from a similar calculation as above, and tightness. 
\end{proof}

\begin{remark}\label{RemarkFour}
   Inequality \eqref{mixingest} implies that $\mathbf{X}^{(\epsilon)}_1-\mathbf{Y}^{(\epsilon)}_1$ has a density with respect to the Lebesgue measure and from the Lebesgue differentiation lemma, this density is in $L^{\infty}(\R^2)$. Now, we observe that for two measures $\mu,\nu$ that are absolutely continuous with respect to the Lebesgue measure, their total variation distance can be written as 
   \[d_{TV}(\mu,\nu)=\frac{1}{2}\int_{\R^2}\biggl|\frac{d\mu}{dx}-\frac{d\nu}{dx}\biggr|dx.\]
   This observation, combined with \textbf{Proposition \ref{thm:prop4.2}}, implies  that the density of $\mathbf{X}^{(\epsilon)}_1-\mathbf{Y}^{(\epsilon)}_1$ with respect to the Lebesgue measure converges in $L^1$ to the corresponding density of $\sum_{j=0}^{\theta-1}\mathcal{N}_j$. This implies that the characteristic function of $\mathbf{X}^{(\epsilon)}_1-\mathbf{Y}^{(\epsilon)}_1$ converges uniformly to the characteristic function of $\sum_{j=0}^{\theta-1}\mathcal{N}_j$. It is also easy to see that all mixed moments of $\mathbf{X}^{(\epsilon)}_1-\mathbf{Y}^{(\epsilon)}_1$ are bounded as $\epsilon\rightarrow0$.  
\end{remark}

The observations in \textbf{Remark \ref{RemarkFour}} lead us to the following central limit theorem (and its local version):\par

\begin{myprop}\label{prop4.4}
    Let $(X_0,Y_0)\sim W_1\times W_1$. Then as $k\rightarrow\infty$ and $\epsilon\rightarrow0$,  the random variable
    $$\frac{1}{\sqrt{k}}(\omega_{X_0}(T_k)-\omega_{Y_0}(T_k))$$
    converges in distribution to a normal random variable in $\mathbb{R}^2$ with covariance matrix $\frac{2}{\gamma}I_{2\times2}$. Furthermore, if  $F_k^{(\epsilon)}$ is the density  of $\omega_{X_0}(T_k)-\omega_{Y_0}(T_k)$ with respect to the Lebesque measure, then
    \begin{equation}\label{local}
        \sup_{\mathbf{X}\in\mathbb{R}^2}|kF_k^{(\epsilon)}(\mathbf{X})-\frac{\gamma}{4\pi}\exp(-\gamma|\mathbf{X}|^2/4k)|\rightarrow0,
    \end{equation}
        as $k\rightarrow\infty$ and $\epsilon\rightarrow0$.\newline\par
\end{myprop}

\begin{proof}
     Recall that $\mathbf{X}^{(\epsilon)}_1-\mathbf{Y}^{(\epsilon)}_1$  has zero mean. The central limit theorem is proved via a standard calculation of the characteristic function. In particular, from \eqref{Randomwalks} we have
    $$\phi_{\frac{1}{\sqrt{k}}(\omega_{X_0}(T_k)-\omega_{Y_0}(T_k))}(\mathbf{\eta})=(\phi_{\mathbf{X}^{(\epsilon)}_1-\mathbf{Y}^{(\epsilon)}_1}(\mathbf{\eta}/\sqrt{k}))^k,$$
    where $\phi_X$ denotes the characteristic function of the random variable $X$.\par
    As we observed in \textbf{Remark \ref{RemarkFour}}, all mixed moments of the vector $\mathbf{X}_1^{(\epsilon)}-\mathbf{Y}^{(\epsilon)}_1$ are bounded as $\epsilon\rightarrow0$. Therefore, a Taylor expansion yields
    $$\phi_{\frac{1}{\sqrt{k}}(\omega_{X_0}(T_k)-\omega_{Y_0}(T_k))}(\mathbf{\eta})=\biggl(1-\frac{\langle\mathbf{\eta},\mathbb{E}[(\mathbf{X}_1^{(\epsilon)}-\mathbf{Y}_1^{(\epsilon)})\otimes(\mathbf{X}_1^{(\epsilon)}-\mathbf{Y}_1^{(\epsilon)})]\mathbf{\eta}\rangle}{2k}+\textbf{R}_{\epsilon,k}(\mathbf{\eta})\biggr)^k,$$
    where the error term $\textbf{R}_{\epsilon,k}(\mathbf{\eta})$ satisfies the bound
    $$\limsup_{\epsilon\rightarrow0}|\mathbf{R}_{\epsilon,k}(\mathbf{\eta})|=o(1/k)|\mathbf{\eta}|^2.$$
     From \textbf{Proposition \ref{thm:prop4.2}}
    $$\mathbb{E}[(\mathbf{X}_1^{(\epsilon)}-\mathbf{Y}_1^{(\epsilon)})\otimes(\mathbf{X}_1^{(\epsilon)}-\mathbf{Y}_1^{(\epsilon)})]\rightarrow2/\gamma I_{2\times2},$$
    as $\epsilon\rightarrow0$. Therefore:
    \begin{equation}\label{eq4.8}
        \phi_{\frac{1}{\sqrt{k}}(\omega_{X_0}(T_k)-\omega_{Y_0}(T_k))}(\mathbf{\eta})\rightarrow\exp(-|\mathbf{\eta}|^2/\gamma),
    \end{equation}
    as $k\rightarrow\infty$, $\epsilon\rightarrow0$.\par
    {From \textbf{Remark \ref{RemarkFour}}, the density of $\mathbf{X}_1^{(\epsilon)}-\mathbf{Y}^{(\epsilon)}_1$ is in $L^{\infty}(\R^2)$. From Plancherel's theorem, this implies that $\phi_{\omega_{X_0}(T_k)-\omega_{Y_0}(T_k)}$ is in $L^1(\R^2)$ for all $k\geq2$. Therefore, if $F_k^{(\epsilon)}(\mathbf{x})$ is the probability density of $\omega_{X_0}(T_k)-\omega_{Y_0}(T_k)$, then by Fourier inversion, 
    \begin{equation}\label{Fourierinv}
        F_k^{(\epsilon)}(\mathbf{x})=\frac{1}{(2\pi)^2}\int_{\mathbb{R}^2}\exp(-i\langle\mathbf{x},\mathbf{\eta}\rangle)\phi_{\omega_{X_0}(T_k)-\omega_{Y_0}(T_k)}(\mathbf{\eta})d\mathbf{\eta}.
    \end{equation}
    To prove \eqref{local}, we go through the following steps:
    \begin{enumerate}[label=\textbf{F.\arabic*}]
        \item $\mathbf{X}_1^{(\epsilon)}-\mathbf{Y}^{(\epsilon)}_1$ has a density  $f_{\mathbf{X}_1^{(\epsilon)}-\mathbf{Y}_1^{(\epsilon)}}\in L^2(\mathbb{R}^2)$. There is an $\epsilon_0'>0$, such that
        \begin{equation}\label{densitybound}
        \sup_{\epsilon\leq\epsilon_0'}\int_{\mathbb{R}^2}f_{\mathbf{X}_1^{(\epsilon)}-\mathbf{Y}_1^{(\epsilon)}}(\mathbf{x})^2d\mathbf{x}\lesssim1.
    \end{equation}
    \item There is an $\epsilon_0''$ and an $A$ such that
    \begin{equation}\label{firstred}
        F_k^{(\epsilon)}(\mathbf{x})=\frac{1}{(2\pi)^2k}\int_{|\mathbf{\eta}|\leq A\sqrt{k}}\exp(-i\langle\mathbf{x},\mathbf{\eta}\rangle/\sqrt{k})\phi_{\omega_{X_0}(T_k)-\omega_{Y_0}(T_k)}(\mathbf{\eta}/\sqrt{k})d\mathbf{\eta}+\mathcal{R}_{\epsilon,k},
    \end{equation}
    with the error term $\mathcal{R}_{\epsilon,k}$ satisfying the bound
    $$\sup_{\epsilon\leq\epsilon_0''}|\mathcal{R}_{\epsilon,k}|\lesssim2^{-k}.$$
    \item As $k\rightarrow\infty$ and $\epsilon\rightarrow0$
         \begin{equation}\label{secondred}
            \int_{|\mathbf{\eta}|\leq A\sqrt{k}}|\phi_{\omega_{X_0}(T_k)-\omega_{Y_0}(T_k)}(\mathbf{\eta}/\sqrt{k})-\exp(-|\mathbf{\eta}|^2/\gamma)|d\mathbf{\eta}\rightarrow0.
         \end{equation}
    \item As $k\rightarrow\infty$,
         \begin{equation}\label{Gaussianint}
             \frac{1}{(2\pi)^2}\int_{|\mathbf{\eta}|\leq A\sqrt{k}}\exp(-i\langle\mathbf{x},\mathbf{\eta}\rangle/\sqrt{k})\exp(-|\mathbf{\eta}|^2/\gamma)d\mathbf{\eta}=\frac{\gamma}{4\pi}\exp(-\gamma|\mathbf{x}|^2/4k)+o(1).
         \end{equation}      
    \end{enumerate}
    Equations \eqref{firstred}-\eqref{Gaussianint} prove that:
    $$F_k^{(\epsilon)}(\mathbf{x})=\frac{\gamma}{4\pi k}\exp(-\gamma|\mathbf{x}|^2/4k)+\mathcal{R}_{\epsilon,k}'(x)\cdot1/k,$$
    where the error term $\sup_{x\in\R^2}|\mathcal{R}_{\epsilon,k}'(x)|$ goes to $0$, as $k\rightarrow\infty$ and $\epsilon\rightarrow0$. This proves \eqref{local}.}\par  
    {To prove \eqref{densitybound}, recall that from \textbf{Remark \ref{RemarkFour}}, $\mathbf{X}_1^{(\epsilon)}-\mathbf{Y}_1^{(\epsilon)}$ has a density with respect to the Lebesgue measure, which we denote by $f_{\mathbf{X}_1^{(\epsilon)}-\mathbf{Y}_1^{(\epsilon)}}$. From \eqref{eq4.11}
    $$\int_{\mathbb{R}^2}f_{\mathbf{X}_1^{(\epsilon)}-\mathbf{Y}_1^{(\epsilon)}}(\mathbf{x})^2d\mathbf{x}=\mathbb{E}_{\mathbf{X}_1^{(\epsilon)},\mathbf{Y}_1^{(\epsilon)}}[f_{\mathbf{X}_1^{(\epsilon)}-\mathbf{Y}_1^{(\epsilon)}}(\mathbf{X}_1^{(\epsilon)}-\mathbf{Y}_1^{(\epsilon)})]\leq$$
    $$\sum_{N=1}^{\infty}B_{\epsilon}^N\gamma(1-\gamma)^{N-1}\mathbb{E}_{\mathbf{X}^{(N)},\mathbf{Y}^{(N)}}[f_{\mathbf{X}_1^{(\epsilon)}-\mathbf{Y}_1^{(\epsilon)}}(\mathbf{X}^{(N)}-\mathbf{Y}^{(N)})],$$
    where $(\mathbf{X}^{(N)},\mathbf{Y}^{(N)})$ are as in \eqref{eq4.11} and $\mathcal{B}_{\epsilon}\rightarrow1$ as $\epsilon\rightarrow0$. Since the random variables $\mathbf{X}^{(N)},\mathbf{Y}^{(N)}$ are Gaussian, we have:
    $$\int_{\mathbb{R}^2}f_{\mathbf{X}_1^{(\epsilon)}-\mathbf{Y}_1^{(\epsilon)}}(\mathbf{x})^2d\mathbf{x}\lesssim\sum_{N=1}^{\infty}\frac{1}{N}B_{\epsilon}^N\gamma(1-\gamma)^{N-1}\int_{\mathbb{R}^2}f_{\mathbf{X}_1^{(\epsilon)}-\mathbf{Y}_1^{(\epsilon)}}(\mathbf{x})d\mathbf{x}\lesssim1,$$
    for all $\epsilon\leq\epsilon_0'$, where $\epsilon_0'$ is such that
    $\sup_{\epsilon\leq\epsilon_0'}\mathcal{B}_{\epsilon}<\frac{1}{(1+\gamma)(1-\gamma)}$. This proves \eqref{densitybound}.}\par
    {To prove \eqref{firstred} we do a change of variables $\mathbf{\eta}\rightarrow\mathbf{\eta}/\sqrt{k}$ in the integral in \eqref{Fourierinv}. Then, we split the integral over the regions $|\mathbf{\eta}|> A\sqrt{k}$ and $|\mathbf{\eta}|\leq A\sqrt{k}$. The proof of \eqref{firstred} reduces to proving that there is an $A$ and an $\epsilon''_0$ such that:
    \begin{equation}\label{smallterm}
        \sup_{\epsilon\leq\epsilon_0''}\biggl|\frac{1}{k}\int_{|\mathbf{\eta}|> A\sqrt{k}}\phi_{\omega_{X_0}(T_k)-\omega_{Y_0}(T_k)}(\mathbf{\eta}/\sqrt{k})d\mathbf{\eta}\biggr|\lesssim2^{-k}.
    \end{equation}
    We have that:
    \begin{equation}\label{basicfourierbound}
        \frac{1}{k}\int_{|\mathbf{\eta}|> A\sqrt{k}}|\phi_{\omega_{X_0}(T_k)-\omega_{Y_0}(T_k)}(\mathbf{\eta}/\sqrt{k})|d\eta\leq\sup_{|\mathbf{\eta}|> A}|\phi_{\mathbf{X}_1^{(\epsilon)}-\mathbf{Y}_1^{(\epsilon)}}(\mathbf{\eta})|^{k-2}\frac{1}{k}\int_{\mathbb{R}^2}|\phi_{\mathbf{X}_1^{(\epsilon)}-\mathbf{Y}_1^{(\epsilon)}}(\mathbf{\eta}/\sqrt{k})|^{2}d\mathbf{\eta}.
    \end{equation}
    The second term in the right-hand side of \eqref{basicfourierbound}, after a change of variables $\mathbf{\eta}\rightarrow\sqrt{k}\mathbf{\eta}$ and Plancerel's theorem,  is equal to $(2\pi)^2\int_{\mathbb{R}^2}f_{\mathbf{X}_1^{(\epsilon)}-\mathbf{Y}_1^{(\epsilon)}}(\mathbf{x})^2d\mathbf{x}$ which, from \eqref{densitybound}, is bounded in $\epsilon\leq\epsilon_0'$.}\par
    {The first term on the right-hand side of \eqref{basicfourierbound} is exponentially small. Indeed,
    from \textbf{Proposition \ref{thm:prop4.2}}, for all $\mathbf{\eta}$
    $$\phi_{\mathbf{X}_1^{(\epsilon)}-\mathbf{Y}_1^{(\epsilon)}}(\mathbf{\eta})\rightarrow\phi_{\mathbf{Y}}(\mathbf{\eta}),$$ 
    where $\mathbf{Y}=\sum_{i=0}^{\theta-1}\mathcal{N}_i$, with $\theta$ and $(\mathcal{N}_{i})_{i\in\mathbb{N}_0}$ as in \textbf{Proposition \ref{thm:prop4.2}}. From \textbf{Remark \ref{RemarkFour}}, this convergence happens uniformly in $\mathbf{\eta}$.  A straightforward calculation shows that:
    $$\phi_{\mathbf{Y}}(\mathbf{\eta})=\frac{\gamma}{e^{|\mathbf{\eta}|^2/2}-(1-\gamma)}.$$
    Therefore, there is an $A$ such that
    $\sup_{|\mathbf{\eta}|> A}|\phi_{\mathbf{Y}}(\mathbf{\eta})|\leq1/4$. This implies that, there is an $\epsilon_0''\leq\epsilon_0'$ such that for all $\epsilon\leq\epsilon_0''$ 
    $$\sup_{|\mathbf{\eta}|\geq A}|\phi_{\mathbf{X}_1^{(\epsilon)}-\mathbf{Y}_1^{(\epsilon)}}(\mathbf{\eta})|\leq1/2,$$
    which proves \eqref{smallterm}.}\par    
     {The proof of \eqref{secondred} goes along the same lines as the proofs of local limit theorems found in \cite{lawler2010random}. We skip the details. Finally, Gaussian integration proves \eqref{Gaussianint} and this concludes the proof.}
\end{proof}

\section{Kallianpur-Robbins Estimates}\label{sec:app1}
{Let $(X_0,Y_0)\sim W_1\times W_1$ and $S_n=\omega_{X_0}(T_{n})-\omega_{Y_0}(T_{n})$. Recall that, from \textbf{Proposition \ref{prop4.4}},
 \begin{equation}\label{loaclest}
     F^{(\epsilon)}_n(x)=\frac{\gamma}{4\pi k}e^{-\gamma|x|^2/4k}+\mathcal{R}_{\epsilon,k}(x)\cdot1/k,
 \end{equation}
    for all $x\in\mathbb{R}^2$, where $\sup_{x\in\mathbb{R}^2}|\mathcal{R}_{\epsilon,k}(x)|\rightarrow0$ as $k\rightarrow\infty$, $\epsilon\rightarrow0$. These precise asymptotics for the density $F^{(\epsilon)}_n$ allow us to prove results on the occupation times of $S_n$.}\par
 We will make use of the following elementary estimates. Here, $M(\epsilon)$ is such that $\log M(\epsilon)/\log \frac{1}{\epsilon}\rightarrow0$ as $\epsilon\rightarrow0$. We have
\begin{equation}\label{elementary}
    \frac{1}{(\log \frac{1}{\epsilon})^p}\sum_{\substack{1\leq k_1<...<k_p\leq [M(\epsilon)/\epsilon^2]}}\frac{1}{k_1(k_2-k_1)...(k_p-k_{p-1})}\rightarrow2,
    \end{equation}
    as $\epsilon\rightarrow0$ and for $j=1,...,p$
\begin{equation}\label{elementary2}
    \frac{1}{(\log \frac{1}{\epsilon})^p}\sum_{\substack{1\leq k_1<...<k_p\leq [M(\epsilon)/\epsilon^2]}}\frac{|\mathcal{R}_{\epsilon, k_j-k_{j-1}}(x)|\prod_{m>j}^p(1+|\mathcal{R}_{\epsilon,k_m-k_{m-1}}(x)|)}{k_1(k_2-k_1)...(k_p-k_{p-1})}=o(1),
\end{equation}
as $\epsilon\rightarrow0$, where $\mathcal{R}_{\epsilon,k}$ is as in \eqref{loaclest}. Similarly, for $j=1,...,p$
\begin{equation}\label{elementary3}
    \frac{1}{(\log \frac{1}{\epsilon})^p}\sum_{\substack{1\leq k_1<...<k_p\leq [M(\epsilon)/\epsilon^2]}}\frac{(k_j-k_{j-1})^{-1}}{k_1(k_2-k_1)...(k_p-k_{p-1})}=o(1),
\end{equation}
as $\epsilon\rightarrow0$.\par
The proof of \textbf{Proposition \ref{propA.1}} is similar to the proofs of the results in \cite{Kallianpur1954TheSO}.

\begin{myprop}\label{propA.1}
Let $M(\epsilon)$ be such that $\log M(\epsilon)/\log \frac{1}{\epsilon}\rightarrow0$ as $\epsilon\rightarrow0$. Also let $(H_\epsilon)_{\epsilon>0}\subseteq L^1(\mathbb{R}^{2p})$ be a family of non-negative and bounded functions, such that for all $\epsilon$, $||H_\epsilon||_1=A$, for some constant $A$ and such that
\begin{equation}\label{fakeuniformint}
    \lim_{M\rightarrow\infty}\sup_{\epsilon\in(0,1)}\int_{|\mathbf{x}|\geq M}H_\epsilon(\mathbf{x})d\mathbf{x}\rightarrow0.
\end{equation}
For $S_n=\omega_{X_0}(T_{n})-\omega_{Y_0}(T_{n})$, 
    \begin{equation}\label{kalRob}
        \biggl(\frac{2\pi}{\gamma \log \frac{1}{\epsilon}}\biggr)^p\sum_{\substack{1\leq k_1<...<k_p\leq [M(\epsilon)/\epsilon^2]}}\E[H_{\epsilon}(S_{k_1},...,S_{k_p})]\rightarrow A,
    \end{equation}
    as $\epsilon\rightarrow0$. 
    \end{myprop}

\begin{proof}
   {To prove \eqref{kalRob} we write the expectation inside the sum explicitly. We have
    $$\sum_{\substack{1\leq k_1<...<k_p\leq [M(\epsilon)/\epsilon^2]}}\E[H_\epsilon(S_{k_1},...,S_{k_p})]=$$
    \begin{equation}\label{dircal}
        \sum_{\substack{1\leq k_1<...<k_p\leq [M(\epsilon)/\epsilon^2]}}\int_{\mathbb{R}^{2p}}H_\epsilon(x_1,x_2,...,x_p)F^{(\epsilon)}_{k_1}(x_1)F^{(\epsilon)}_{k_2-k_1}(x_2-x_1)...F^{(\epsilon)}_{k_{p}-k_{p-1}}(x_p-x_{p-1})dx_1...dx_p.
    \end{equation}
    We successively replace  $F^{(\epsilon)}_{k_{i}-k_{i-1}}$, $i=1,...,p$, by its  asymptotic expansion \eqref{loaclest}. Adopting the notation in \cite{Kallianpur1954TheSO}, \eqref{dircal} is equal to
    \begin{equation}\label{asexpansion}
        \biggl(\frac{\gamma}{4\pi}\biggr)^p\sum_{\substack{1\leq k_1<...<k_p\leq M(\epsilon)/\epsilon,\\ k_{i+1}>k_{i}+2}}\frac{\int_{\mathbb{R}^{2p}}H_\epsilon(x_1,x_2,...,x_p)e^{-Q_{k_1,...,k_p}(x_1,...,x_p)}dx_1...dx_p}{k_1(k_2-k_1)...(k_p-k_{p-1})}+\mathcal{E}_{\epsilon},
    \end{equation}
    where
    \begin{equation}\label{exponent}
        Q_{k_1,...,k_p}(x_1,...,x_p):=\gamma\sum_{i=1}^p\frac{|x_{i}-x_{i-1}|^2}{4(k_i-k_{i-1})},
    \end{equation}
    with $x_0=0$, $k_0=0$. The error term $\mathcal{E}_{\epsilon}$ is equal to
    \begin{align}\label{errorterm}
        \sum_{j=1}^p\biggl(\frac{\gamma}{4\pi}\biggr)^j\sum_{\substack{1\leq k_1<...<k_p\leq [M(\epsilon)/\epsilon^2]}}\mathcal{R}_{\epsilon,k_j-k_{j-1}}(x_j-x_{j-1})\int_{\mathbb{R}^{2p}}H_\epsilon(x_1,x_2,...,x_p)\frac{e^{-Q_{k_1,...,k_j}(x_1,...,x_j)}}{\prod_{i=1}^j(k_i-k_{i-1})}\cdot\nonumber\\\cdot\biggl(\prod_{m>j}^pF^{(\epsilon)}_{k_{m}-k_{m-1}}(x_m-x_{m-1})\biggr)d\mathbf{x}.
    \end{align}}
    {We need to estimate both terms in \eqref{asexpansion}. First, we deal with the error term  \eqref{errorterm}. Observe that \eqref{loaclest} implies that $F^{(\epsilon)}_n(x)\leq(1+\mathcal{R}_{\epsilon,n})/n$. Using this bound,  we see that the error term is
    \begin{align*}
        \lesssim\sum_{j=1}^p\sum_{\substack{1\leq k_1<...<k_p\leq [M(\epsilon)/\epsilon^2]}}\int_{\mathbb{R}^{2p}}\mathcal{R}_{\epsilon,k_j-k_{j-1}}(x_j-x_{j-1})\frac{H_\epsilon(x_1,x_2...,x_p)}{k_1(k_2-k_1)...(k_p-k_{p-1})}\prod_{m>j}^p(1+\\+\mathcal{R}_{\epsilon,k_m-k_{m-1}}(x_m-x_{m-1}))dx_1...dx_p.
    \end{align*}
    Since for all $\epsilon$, $||H_\epsilon||_1=A$, this sum is bounded above by
    $$A\sum_{j=1}^p\sum_{\substack{1\leq k_1<...<k_p\leq [M(\epsilon)/\epsilon^2]}}\frac{\sup_{x\in\mathbb{R}^2}|\mathcal{R}_{\epsilon,k_j-k_{j-1}}(x)|\prod_{m>j}^p(1+\sup_{x\in\mathbb{R}^2}|\mathcal{R}_{\epsilon,k_m-k_{m-1}}(x)|)}{k_1(k_2-k_1)...(k_p-k_{p-1})}.$$
    From \eqref{elementary2},  this term is $o((\log \frac{1}{\epsilon})^p)$ as $\epsilon\rightarrow0$. Therefore, $\mathcal{E}_{\epsilon}=o((\log \frac{1}{\epsilon})^p)$.}\par
    {Now, we focus on the first term in \eqref{asexpansion}:
    \begin{equation}\label{dominant}
        \biggl(\frac{\gamma}{4\pi}\biggr)^p\sum_{\substack{1\leq k_1<...<k_p\leq [M(\epsilon)/\epsilon^2]}}\frac{\int_{\mathbb{R}^{2p}}H_\epsilon(x_1,x_2,...,x_p)e^{-Q_{k_1,...,k_p}(x_1,...,x_p)}dx_1...dx_p}{k_1(k_2-k_1)...(k_p-k_{p-1})}.
    \end{equation}}
    {We write \eqref{dominant} as
    $${J}_{p,\epsilon}+ A\biggl(\frac{\gamma }{4\pi}\biggr)^p\sum_{\substack{1\leq k_1<...<k_p\leq [M(\epsilon)/\epsilon^2]}}\frac{1}{k_1(k_2-k_1)...(k_p-k_{p-1})},$$
    where
    \begin{equation}\label{Jterm}
        J_{p,\epsilon}= \biggl(\frac{\gamma}{4\pi}\biggr)^p\sum_{\substack{1\leq k_1<...<k_p\leq [M(\epsilon)/\epsilon^2]}}\frac{\int_{\mathbb{R}^{2p}}H_\epsilon(x_1,x_2,...,x_p)(e^{-Q_{k_1,...,k_p}(x_1,...,x_p)}-1)dx_1...dx_p}{k_1(k_2-k_1)...(k_p-k_{p-1})}.
    \end{equation}
    We aim to prove that $(\log\frac{1}{\epsilon})^{-p}J_{p,\epsilon}\rightarrow0$, as $\epsilon\rightarrow0$. This will prove that \eqref{dominant} is asymptotic to $A({\gamma }/{2\pi})^p$ as $\epsilon\rightarrow0$, which proves \eqref{kalRob}.}\par
    {From our assumption on \eqref{fakeuniformint}, for any $\delta$, there is a $M_0$ large enough  such that
    $$\sup_{\epsilon>0}\int_{|\mathbf{x}|\geq M_0}H_\epsilon(x_1,x_2,...,x_p)d\mathbf{x}\leq\delta.$$
    Therefore,
    \begin{equation}\label{unifintbound}
        \biggl(\frac{\gamma}{4\pi}\biggr)^p\sum_{\substack{1\leq k_1<...<k_p\leq [M(\epsilon)/\epsilon^2]}}\frac{\int_{|\mathbf{x}|\geq M_0}H_\epsilon(x_1,x_2,...,x_p)|e^{-Q_{k_1,...,k_p}(x_1,...,x_p)}-1|d\mathbf{x}}{k_1(k_2-k_1)...(k_p-k_{p-1})}
    \end{equation}
    is bounded above by
    $$2\delta\biggl(\frac{\gamma}{4\pi}\biggr)^p\sum_{1\leq k_1<...<k_p\leq [M(\epsilon)/\epsilon^2]}\frac{1}{k_1(k_2-k_1)...(k_p-k_{p-1})}.$$
    By dividing this term by $(\log\frac{1}{\epsilon})^p$ and sending $\epsilon\rightarrow0$, we see that it converges to $4\delta(\gamma/4\pi)^p$. Since $\delta$ is a arbitrary, \eqref{unifintbound} is negligible and the main contribution to \eqref{Jterm} comes from 
    $$J_{p,\epsilon}'':=\biggl(\frac{\gamma}{4\pi}\biggr)^p\sum_{\substack{1\leq k_1<...<k_p\leq [M(\epsilon)/\epsilon^2]}}\frac{\int_{|\mathbf{x}|\leq M_0}H_\epsilon(x_1,x_2,...,x_p)(e^{-Q_{k_1,...,k_p}(x_1,...,x_p)}-1)dx_1...dx_p}{k_1(k_2-k_1)...(k_p-k_{p-1})}.$$
    Observe that for all $|\mathbf{x}|\leq M_0$, we have
    \begin{equation*}\label{Gexpestimate}
        |e^{-Q_{k_1,...,k_p}(x_1,...,x_p)}-1|\lesssim Q_{k_1,...,k_p}(x_1,...,x_p)\lesssim\sum_{j=1}^p\frac{1}{k_{j}-k_{j-1}}.
    \end{equation*}
    Therefore,
    $$J_{p,\epsilon}''\lesssim A\biggl(\frac{\gamma}{4\pi}\biggr)^p\sum_{\substack{1\leq k_1<...<k_p\leq [M(\epsilon)/\epsilon^2]}}\sum_{j=1}^p\frac{1}{k_1(k_2-k_1)...(k_p-k_{p-1})}\cdot
    \frac{1}{k_{j}-k_{j-1}}.$$
    From \eqref{elementary3}, the sum in the right-hand side  is $o((\log\frac{1}{\epsilon})^p)$ as $\epsilon\rightarrow0$, and therefore $J_{p,\epsilon}=o((\log\frac{1}{\epsilon})^p)$. This concludes the proof of \eqref{kalRob}.}\par
    
\end{proof}

\begin{remark}\label{Remofapplicability}
    Let  $(G_\epsilon)_{\epsilon\in(0,1)}\subseteq L^1(\mathbb{R}^{2p})$ satisfy the assumptions of \textbf{Proposition \ref{propA.1}} and define  
    $$\Gcal_\epsilon(\mathbf{z})=\E_{Z^{(\epsilon)}}[G_\epsilon(\mathbf{z}+Z^{(\epsilon)})],$$
    where $(Z^{(\epsilon)})_{\epsilon\in(0,1)}$ is a tight collection of random variables in $\mathbb{R}^{2}$. The family $(\Gcal_\epsilon)_{\epsilon\in(0,1)}$ satisfies the assumptions of \textbf{Proposition \ref{propA.1}}. To prove this we only need to prove \eqref{fakeuniformint}. Since $(Z^{(\epsilon)})_{\epsilon\in(0,1)}$ is tight we have
    \begin{equation}\label{densityfakeunif}
        \sup_{\epsilon\in(0,1)}\P[{|Z^{(\epsilon)}}|\geq M]\rightarrow0,
    \end{equation}
    as $M\rightarrow\infty$. We also have
    $$\int_{|z|\geq  M}\Gcal_\epsilon(\mathbf{z})d\mathbf{z}\leq\E\biggr[\textbf{1}_{\{|Z^{(\epsilon)}|\geq M/2\}}\int_{|\mathbf{z}|\geq M-|Z^{(\epsilon)}|}{G}_\epsilon(\mathbf{z})d\mathbf{z}\biggl]+\E\biggr[\textbf{1}_{\{|Z^{(\epsilon)}|\leq M/2\}}\int_{|\mathbf{z}|\geq M-|Z^{(\epsilon)}|}{G}_\epsilon(\mathbf{z})d\mathbf{z}\biggl].$$
    Therefore, since $(G_\epsilon)_{\epsilon>0}$ satisfies the assumptions of \textbf{Proposition \ref{propA.1}} and from \eqref{densityfakeunif}
    $$\int_{|z|\geq M}\Gcal_\epsilon(\mathbf{z})d\mathbf{z}\leq\P[{|Z^{(\epsilon)}}|\geq M/2]\int_{\mathbb{R}^{2p}}{G}_\epsilon(\mathbf{z})d\mathbf{z}+\int_{|z|\geq M/2}{G}_\epsilon(\mathbf{z})d\mathbf{z}\rightarrow0,$$
    as $M\rightarrow\infty$, uniformly in $\epsilon$.
\end{remark}

We need a 'non-directed' version of \textbf{Proposition \ref{propA.1}}. First, we introduce the following notation: For any non-negative function $G:\mathbb{R}^{pd}\rightarrow\mathbb{R}$ and a nonempty subset $J\subset\{1,...,p\}$ we define the function $G^J:\mathbb{R}^{(p-|J|)d}\rightarrow\mathbb{R}$ as 
$$G^J((x_i)_{i\in\{1,...,p\}\backslash J})=\sup_{x_i\in\mathbb{R}^d,i\in J}G(x_1,...,x_p).$$
When $J=\emptyset$ we set $G^J=G$.\par 
With this notation, we now state the following lemma:\newline\par

\begin{mylem}\label{lemm4.3}
    Let $(G_\epsilon)_{\epsilon\in(0,1)}$ be a collection of functions as in \textbf{Proposition \ref{propA.1}}, satisfying the following additional condition: For any $J\subset\{1,...,p\}$, we have:
    \[G^J_\epsilon(x)\leq \Hcal_\epsilon(x),\]
    where the collection of functions $(\Hcal_\epsilon)_{\epsilon\in(0,1)}$ satisfies the conditions in \textbf{Proposition \ref{propA.1}}.\par
    Now, let $\boldsymbol{\delta}\in\{+1,-1,0\}^p$ and $M(\epsilon)$ be such that $\log M(\epsilon)/\log{\frac{1}{\epsilon}}\rightarrow0$ as $\epsilon\rightarrow0$. Then the mean of   
    \begin{equation}\label{nondirectedsum}
        \sum_{\substack{1\leq k_1<...<k_p\leq [M(\epsilon)/\epsilon^2]}}G_\epsilon(\omega_{X_0}(T_{k_1+\delta_1})-\omega_{Y_0}(T_{k_1}),...,\omega_{X_0}(T_{k_p+\delta_p})-\omega_{Y_0}(T_{k_p}))
    \end{equation}
    is asymptotic to
    $$\biggr((\gamma/2\pi)\log \frac{1}{\epsilon}\biggl)^pA,$$
    as $\epsilon\rightarrow0$, where $||G_\epsilon||_1=A$.\newline\par
\end{mylem}

\begin{proof}
    First, we will calculate the limit of the mean of
    \begin{equation}\label{eq4.12}
        \sum_{\substack{2\leq k_1<...<k_p\leq [M(\epsilon)/\epsilon^2],\\ k_{i+1}>k_i+2}}G_\epsilon(\omega_{X_0}(T_{k_1+\delta_1})-\omega_{Y_0}(T_{k_1}),...,\omega_{X_0}(T_{k_p+\delta_p})-\omega_{Y_0}(T_{k_p})),
    \end{equation}
    for any collection  $(G_\epsilon)_{\epsilon\in(0,1)}$ that satisfies the conditions of \textbf{Proposition \ref{propA.1}}.\par
    We aim to use \textbf{Proposition \ref{propA.1}} to calculate the asymptotic of \eqref{eq4.12}.  The challenge is the appearance of $\boldsymbol{\delta}$. We try to bring the functional in \eqref{eq4.12} to a functional involving only the random walk $\omega_{X_0}(T_k)-\omega_{Y_0}(T_k)$.  We will exploit the fact that both $\omega_{X_0}(T_k)$ and $\omega_{Y_0}(T_k)$, as defined in \eqref{Randomwalks}, are sums of independent random variables. We can write the random vector appearing in $(\ref{eq4.12})$ as
    \begin{equation}\label{eq4.14}
        (\omega_{X_0}(T_{k_1+\delta_1})-\omega_{Y_0}(T_{k_1}),...,\omega_{X_0}(T_{k_p+\delta_p})-\omega_{Y_0}(T_{k_p}))=(\mathbf{S}^{(\epsilon)}_{k_1},...,\mathbf{S}^{(\epsilon)}_{k_p})+\mathbf{Z}^{(\epsilon)},
    \end{equation}
    where the vectors $(\mathbf{S}^{(\epsilon)}_{k_1},...,\mathbf{S}^{(\epsilon)}_{k_p})$ and $\mathbf{Z}^{(\epsilon)}$ are independent. The distribution of the first vector will "resemble" the distribution of $(\omega_{X_0}(T_{k_1})-\omega_{Y_0}(T_{k_1}),...,\omega_{X_0}(T_{k_p})-\omega_{Y_0}(T_{k_p}))$. Moreover, the distribution of $\mathbf{Z}^{(\epsilon)}$ will be independent from $k_1,...,k_p$. Averaging with respect to $\mathbf{Z}^{(\epsilon)}$ we obtain 
    $$\sum_{\substack{1\leq k_1<...<k_p\leq [M(\epsilon)/\epsilon^2],\\ k_{i+1}>k_{i}+2}}\E[H(\mathbf{S}^{(\epsilon)}_{k_1},...,\mathbf{S}^{(\epsilon)}_{k_p})],$$
    for an appropriate $H$. Then, we can apply $\textbf{Proposition \ref{propA.1}}$ and conclude.\par
    Let us make this more precise. From \eqref{Randomwalks},  we express the increment for each $i=1,..,p$: 
    $$\omega_{X_0}(T_{k_i+\delta_i})-\omega_{Y_0}(T_{k_i})=\mathbf{S}^{}_{k_i}+\sum_{1\leq j<i}(\mathbf{X}_{k_j-1}^{(\epsilon)}-\mathbf{Y}_{k_j-1}^{(\epsilon)}+\mathbf{X}_{k_{j}}^{(\epsilon)}-\mathbf{Y}_{k_{j}}^{(\epsilon)})-\mathbf{Y}_{k_i-1}^{(\epsilon)}+\sum_{j=k_i-1}^{k_i+\delta_i-1}\mathbf{X}_j^{(\epsilon)},$$
    where we make the convention that the last sum is $0$ when $\delta_i=-1$. From \eqref{Randomwalks}, $\mathbf{S}^{(\epsilon)}_{k_i}$ is a sum of independent random variables. Moreover, $\mathbf{S}^{}_{k_i}$ is equal in distribution to $\omega_{X_0}(T_{k_i-2i+1})-\omega_{Y_0}(T_{k_i-2i+1})$ (observe that since $k_{i+1}>k_i+1$ we have $k_i>2i-1$). Furthermore, the vector
    $$(\mathbf{S}^{}_{k_1},...,\mathbf{S}^{}_{k_p})$$
    is equal in distribution to
    $$\biggr(\omega_{X_0}(T_{k_1-1})-\omega_{Y_0}(T_{k_1-1 }),\omega_{X_0}(T_{k_2-3})-\omega_{Y_0}(T_{k_2-3 }),...,\omega_{X_0}(T_{k_p-2p+1})-\omega_{Y_0}(T_{k_p-2p+1 })\biggl)$$
    and is independent of:
    $$(\mathbf{X}_{k_1-1}^{(\epsilon)}-\mathbf{Y}_{k_1-1}^{(\epsilon)},\mathbf{X}_{k_1}^{(\epsilon)}-\mathbf{Y}_{k_1}^{(\epsilon)},...,\mathbf{X}_{k_p-1}^{(\epsilon)}-\mathbf{Y}_{k_p-1}^{(\epsilon)},\mathbf{X}_{k_p}^{(\epsilon)}-\mathbf{Y}_{k_p}^{(\epsilon)}),$$
   since these random variables do not appear inside the sum of any of the $\mathbf{S}^{}_{k_i}$. \par
    Now, define the vector $\mathbf{Z}^{(\epsilon)}\in\mathbb{R}^{2p}$, where the components $\mathbf{Z}^{(\epsilon)}(i)\in\mathbb{R}^{2}$ are given by:
    $$\mathbf{Z}^{(\epsilon)}(i):=\sum_{1\leq j<i}(\mathbf{X}_{k_j-1}^{(\epsilon)}-\mathbf{Y}_{k_j-1}^{(\epsilon)}+\mathbf{X}_{k_{j}}^{(\epsilon)}-\mathbf{Y}_{k_{j}}^{(\epsilon)})-\mathbf{Y}_{k_i-1}^{(\epsilon)}+\sum_{j=k_i-1}^{k_i+\delta_i-1}\mathbf{X}_j^{(\epsilon)}.$$
    This defines the contribution from the terms not included in $\mathbf{S}_{k_i}$. Importantly, the vector $\mathbf{Z}^{(\epsilon)}_p$ captures the dependence introduced by the shifts $\delta_i$ and is independent from $(\mathbf{S}_{k_1},...,\mathbf{S}_{k_p})$.\par
    To prove that the distribution of $\mathbf{Z}^{(\epsilon)}=(\mathbf{Z}^{(\epsilon)}(1),...,\mathbf{Z}^{(\epsilon)}(p))$ is independent of $k_1,...,k_p$ recall that the random variables $((\mathbf{X}_{k_i-1}^{(\epsilon)},\mathbf{Y}_{k_i-1}^{(\epsilon)}),(\mathbf{X}_{k_i}^{(\epsilon)},\mathbf{Y}_{k_i}^{(\epsilon)}))_{i=1,...,p}$ are i.i.d.\footnote{Here, we made use of the restriction $k_i>k_{i-1}+1$, so that $k_j<k_i-1$ for all $j<i$
    .}. Therefore, the vector
     $$\biggr((\mathbf{X}_{k_1-1}^{(\epsilon)},\mathbf{Y}_{k_1-1}^{(\epsilon)}),(\mathbf{X}_{k_1}^{(\epsilon)},\mathbf{Y}_{k_1}^{(\epsilon)}),...,(\mathbf{X}_{k_p-1}^{(\epsilon)},\mathbf{Y}_{k_p-1}^{(\epsilon)}),(\mathbf{X}_{k_p}^{(\epsilon)},\mathbf{Y}_{k_p}^{(\epsilon)})\biggl)$$
    is equal in distribution to
    $$\biggr((\mathbf{X}_{1}^{(\epsilon)},\mathbf{Y}_{1}^{(\epsilon)}),(\mathbf{X}_{2}^{(\epsilon)},\mathbf{Y}_{2}^{(\epsilon)}),...,(\mathbf{X}_{2p-1}^{(\epsilon)},\mathbf{Y}_{2p-1}^{(\epsilon)}),(\mathbf{X}_{2p}^{(\epsilon)},\mathbf{Y}_{2p}^{(\epsilon)})\biggl).$$
    This means that the distribution of $\mathbf{Z}^{(\epsilon)}$ is equal to the  distribution of $Z^{(\epsilon)}=(Z_{1}^{(\epsilon)},...,Z_{p}^{(\epsilon)})$ where
    $$Z_{i}^{(\epsilon)}:=\sum_{1\leq j<i}(\mathbf{X}_{2j-1}^{(\epsilon)}-\mathbf{Y}_{2j-1}^{(\epsilon)}+\mathbf{X}_{2{j}}^{(\epsilon)}-\mathbf{Y}_{2{j}}^{(\epsilon)})-\mathbf{Y}_{2i-1}^{(\epsilon)}+\sum_{j=2i-1}^{2i+\delta_i-1}\mathbf{X}_j^{(\epsilon)}.$$
    Therefore, the distribution of $\mathbf{Z}^{(\epsilon)}$ is independent from $k_1,...,k_p$. These observations prove \eqref{eq4.14}. In particular, \eqref{eq4.12} is equal to
    \begin{equation}\label{averocctime}
        \sum_{\substack{1\leq k_1<...<k_p\leq [M(\epsilon)/\epsilon^2],\\ k_{i+1}>k_i+2}}\E\biggl[\Gcal^{(\epsilon)}\biggr(\omega_{X_0}(T_{k_1-1})-\omega_{Y_0}(T_{k_1-1 }),...,\omega_{X_0}(T_{k_p-2p+1})-\omega_{Y_0}(T_{k_p-2p+1 })\biggl)\biggr],
    \end{equation}
    where $\Gcal^{(\epsilon)}:\mathbb{R}^{2p}\rightarrow\mathbb{R}$ is given by:
    \begin{equation}\label{dirfunction}
        \Gcal^{(\epsilon)}(\mathbf{z}):=\E_{Z^{(\epsilon)}}\biggr[G_\epsilon(\mathbf{z}+Z^{(\epsilon)})\biggl].
    \end{equation}
    By changing the variables $k_i-2i+1=\tilde{k}_i$, \eqref{averocctime} becomes 
    \[
        \sum_{\substack{1\leq \tilde{k}_1<...<\tilde k_p\leq [M(\epsilon)/\epsilon^2]-2p+1}}\E\biggl[\Gcal^{(\epsilon)}\biggr(\omega_{X_0}(T_{\tilde k_1})-\omega_{Y_0}(T_{\tilde{k}_1}),...,\omega_{X_0}(T_{\tilde k_p})-\omega_{Y_0}(T_{\tilde k_p})\biggl)\biggr],
    \]
    From \textbf{Proposition \ref{thm:prop4.2}} and \textbf{Remark \ref{Remofapplicability}}, $(\Gcal^{(\epsilon)})_{\epsilon\in(0,1)}$ satisfies the assumptions of \textbf{Proposition \ref{propA.1}} with $||\Gcal^{(\epsilon)}||_1=A$.  Thus, we deduce that \eqref{averocctime} is  asymptotic to 
    $$\biggr((\gamma/2\pi)\log \frac{1}{\epsilon}\biggl)^pA,$$
    as $\epsilon\rightarrow0$. Since \eqref{eq4.12} and \eqref{averocctime} are equal, we get 
    \begin{equation}\label{firstasymptotic}
        \sum_{\substack{1\leq k_1<...<k_p\leq [M(\epsilon)/\epsilon^2],\\ k_{i+1}>k_{i}+2}}\E[G_\epsilon(\omega_{X_0}(T_{k_1+\delta_1})-\omega_{Y_0}(T_{k_1}),...,\omega_{X_0}(T_{k_p+\delta_p})-\omega_{Y_0}(T_{k_p}))]\sim A(\gamma/2\pi)^p(\log \frac{1}{\epsilon})^p,
    \end{equation}
    as $\epsilon\rightarrow0$, for any collection  $(G_\epsilon)_{\epsilon\in(0,1)}$ that satisfies the conditions of \textbf{Proposition \ref{propA.1}}.\newline\par
   
    Finally, we argue that \eqref{eq4.12} is the main contribution to \eqref{nondirectedsum}. For $r=1,...,p$ and for $1\leq i_1<i_2<...<i_r\leq p$,  we define $A_{i_1,...,i_r}$ to be the following set:
    \begin{equation}\label{setsoffail}
        \{(k_i)_{i\in\{1,...,p\}}\in\mathbb{N}^p/\textit{ }1\leq k_1<...<k_p\leq [M(\epsilon)/\epsilon^2],\textbf{  } k_{i_j}\leq k_{i_j-1}+2, \textbf{  } k_i>k_{i-1}+2, i\neq i_1,...,i_r\}.
    \end{equation}
    The set $A_{i_1,...,i_r}$ specifies where the restriction $k_{i+1}>k_{i}+2$ fails. We now consider the sum
    \begin{equation}\label{residue}
        \sum\limits_{k_1,...,k_p\in A_{i_1,...,i_r}}\mathbb{E}\biggl[G_\epsilon(\omega_{X_0}(T_{k_1+\delta_1})-\omega_{Y_0}(T_{k_1}),...,\omega_{X_0}(T_{k_p+\delta_p})-\omega_{Y_0}(T_{k_p}))\biggr].
    \end{equation}
    Using \eqref{firstasymptotic}, we can show that this sum is $O((\log\frac{1}{\epsilon})^{p-r})=o((\log\frac{1}{\epsilon})^p)$. Indeed, we write $\{1\leq j_1<j_2<....<j_{p-r}\leq[M(\epsilon)/\epsilon^2]\}=\{1,...,p\}\backslash\{i_1,i_2,...,i_r\}$ for the complement set. Then \eqref{residue} is bounded above by
    \[\sum\limits_{k_1,...,k_p\in A_{i_1,...,i_r}}\mathbb{E}\biggl[G^{\{i_1,...,i_r\}}_\epsilon(\omega_{X_0}(T_{k_{j_1}+\delta_{j_1}})-\omega_{Y_0}(T_{k_{j_1}}),...,\omega_{X_0}(T_{k_{j_{p-r}}+\delta_{p-r}})-\omega_{Y_0}(T_{k_{p-r}}))\biggr].\]
    Observe that $k_{i_j}\leq k_{i_j-1}+2$ implies $k_{i_j}=k_{i_{j}-1}+1$ or $k_{i_j}=k_{i_{j}-1}+2$.  Moreover, note that for $(k_i)_{i\in\{1,...,p\}}\in A_{i_1,...,i_r}$ we have $k_{j_m}>k_{j_{m-1}}+2$.  Therefore, the last term is
    \[\lesssim\sum_{\substack{1\leq k_{j_1}<k_{j_2}<...<k_{j_{p-r}}\leq\alpha M\\ k_{j_m}>k_{j_{m-1}}+2}}\mathbb{E}\biggl[G^{\{i_1,...,i_r\}}_\epsilon(\omega_{X_0}(T_{k_{j_1}+\delta_{j_1}})-\omega_{Y_0}(T_{k_{j_1}}),...,\omega_{X_0}(T_{k_{j_{p-r}}+\delta_{p-r}})-\omega_{Y_0}(T_{k_{p-r}}))\biggr].\]
    By assumption, $G^{\{i_1,...,i_r\}}_\epsilon\leq\Hcal_\epsilon^{\{i_1,...,i_r\}}$, where $(\Hcal_\epsilon^{\{i_1,...,i_r\}})_{\epsilon\in(0,1)}$ satisfies the assumptions of \textbf{Proposition \ref{propA.1}}. From \eqref{firstasymptotic} this sum is $O((\log\frac{1}{\epsilon})^{p-r})$, which concludes the proof.
\end{proof}

\section{A  Khas’minskii-type Lemma}\label{App:pre_exp_bound}

Here, we prove an exponential bound for additive functionals of $(\omega_{X_0},\omega_{Y_0})$, that we needed in  \textbf{Section \ref{sec:5}}. We prove it using the fact that the path $(\omega_{X_0},\omega_{Y_0})$ is built from the Markov chain in $\Omega_1\times\Omega_1$.  A similar theorem for finite-state Markov chains can be found in \cite{cosco2023moments}, Appendix C, and is based on the Khas'minskii lemma found in \cite{sznitman2013brownian}. See also Portenko's lemma in  \cite{Port}.

\begin{mylem}\label{kamsinskilemma}
    Recall that we denote by $(\omega_{X_0},\omega_{Y_0})$ the path built from the transition probability kernel $\hat{\boldsymbol{\pi}}$, with $(X_0,Y_0)$ as the initial condition. We set $n_\epsilon=[M(\epsilon)/\epsilon^2]$, where $M(\epsilon)$, is as in \textbf{Theorem \ref{thm:5.1}}.  For $\delta\in\{-1,0,1\}$, we define
    \begin{equation}\label{pkterm}
        p_{k,k+\delta}(x):=\int_k^{k+1}\int_{k+\delta}^{k+\delta+1}R(s-u, x+\omega_{X_0}(s)-\omega_{Y_0}(u))dsdu,
    \end{equation}
    for $k=1,...,n_\epsilon-1$. Additionally, we set $p_{1,0}(x)=p_{n_\epsilon-1,n_\epsilon}(x)=p_{n_\epsilon-1,n_\epsilon+1}(x)=0$. Assume that there exists  a constant $\hat{\beta}>0$, such that
    \begin{equation}\label{assumptioninprop}
        \eta:=\sup_{x\in\R^d,(X_0,Y_0)\in\Omega_1^2}\E\biggl[\sum_{k=1}^{n_\epsilon-1}e^{\frac{\hat{\beta}}{\log\frac{1}{\epsilon}}(p_{k,k-1}(x)+p_{k,k}(x)+p_{k,k+1}(x))}-1\biggr|X_0,Y_0\biggr]<1,
    \end{equation}
    and such that $\eta+2(\exp(3||R||_{\infty}\hat{\beta}/\log\frac{1}{\epsilon})-1)<1$. Then, we have: 
    $$\sup_{x\in\R^d,(X_0,Y_0)\in\Omega_1^2}\E\biggl[\exp\biggl(\frac{\hat{\beta}}{\log\frac{1}{\epsilon}}\int_{[1,n_\epsilon]^2}R(s-u, x+\omega_{X_0}(s)-\omega_{Y_0}(u))dsdu\biggr)\biggr|(X_0,Y_0)\biggr]\leq\frac{1}{1-\eta+\alpha},$$
    where $\alpha:=2(\exp(3||R||_{\infty}\hat{\beta}/\log\frac{1}{\epsilon})-1)$.
\end{mylem}

\begin{proof}
    We set $D_k(x)=e^{\frac{\hat{\beta}}{\log\frac{1}{\epsilon}}(p_{k,k-1}(x)+p_{k,k}(x)+p_{k,k+1}(x))}-1$. Observe that
    $$\int_{[1,n_\epsilon]^2}R(s-u, x+\omega_{X_0}(s)-\omega_{Y_0}(u))dsdu=\sum_{k=1}^{n_\epsilon-1}p_{k,k-1}(x)+p_{k,k}(x)+p_{k,k+1}(x),$$
    and therefore we get
    \[\E\biggl[\exp\biggl(\frac{\hat{\beta}}{\log\frac{1}{\epsilon}}\int_{[1,n_\epsilon]^2}R(s-u, x+\omega_{X_0}(s)-\omega_{Y_0}(u))dsdu\biggr)\biggr]=\E\biggl[\prod_{k=1}^{n_\epsilon-1}\biggl(1+D_k\biggr)\biggr].\]
    The right-hand side is equal to
    \begin{equation}\label{productexp}
        \sum_{p=0}^{\infty}\sum_{1\leq k_1<...<k_p\leq n_\epsilon-1}\E\biggl[\prod_{i=1}^pD_{k_i}(x)\biggr].
    \end{equation}
    We set
    \begin{equation}\label{Apterm}
        A_p:=\sum_{1\leq k_1<...<k_p\leq n_\epsilon-1}\E\biggl[\prod_{i=1}^pD_{k_i}(x)\biggr],
    \end{equation}
    and we will show that $A_p\leq(\eta+2(\exp(3||R||_{\infty}\hat{\beta}/\log\frac{1}{\epsilon})-1))^p$. To prove this, we consider three possibilities for $k_p$ relative to $k_{p-1}$:
    \begin{enumerate}
        \item $k_p>k_{p-1}+2$,
        \item $k_p=k_{p-1}+1$,
        \item $k_{p}=k_{p-1}+2$.
    \end{enumerate}
    Let $A^{(1)}_p$, $A^{(2)}_p$, $A^{(3)}_p$ be defined by the right-hand side of \eqref{Apterm}, but with these respective restrictions. Using the bound $D_k(x)\leq \exp(3||R||_{\infty}\hat{\beta}/\log\frac{1}{\epsilon})-1$, we see that $A^{(2)}_p+A^{(3)}_p\leq2(\exp(3||R||_{\infty}\hat{\beta}/\log\frac{1}{\epsilon})-1)A_{p-1}$. For the $A^{(1)}_p$ term we argue as follows:
    $$A_p^{(1)}=\sum_{1\leq k_1<...<k_{p-1}\leq n_\epsilon-4}\E\biggl[\prod_{i=1}^{p-1}D_{k_i}(x)\sum_{k_p=k_{p-1}+3}^{n_\epsilon-1}D_{k_p}(x)\biggr].$$
    Conditioning on $(X_{k_{p-1}},Y_{k_{p-1}})$ yields
    $$A_p^{(1)}=\sum_{1\leq k_1<...<k_{p-1}\leq n_\epsilon-4}\E\biggl[\prod_{i=1}^{p-1}D_{k_i}(x)\sum_{k_p=k_{p-1}+3}^{n_\epsilon-1}\E\biggl[D_{k_p}(x)\biggl|(X_{k_{p-1}},Y_{k_{p-1}})\biggr]\biggr].$$
    The Markov property shows that
    \[\sum_{k_p=k_{p-1}+3}^{n_\epsilon-1}\E\biggl[D_{k_p}(x)\biggl|(X_{k_{p-1}},Y_{k_{p-1}})\biggr]=\sum_{k_p=3}^{n_\epsilon-k_{p-1}-1}\E\biggl[D_{k_p}(x+\omega_{Y_0}(k_{p-1})-\omega_{Y_0}(k_{p-1}))\biggl|(X_{k_{p-1}},Y_{k_{p-1}})\biggr].\]
    From \eqref{assumptioninprop} this term is $\leq\eta$. Therefore $A_p^{(1)}\leq\eta A_{p-1}$. These observations show that
    \[A_{p}\leq(\eta+2(\exp(3||R||_{\infty}\hat{\beta}/\log\frac{1}{\epsilon})-1))A_{p-1}.\]
    Since $A_1\leq\eta\leq\eta+2(\exp(3||R||_{\infty}\hat{\beta}/\log\frac{1}{\epsilon})-1)$, by assumption, this yields the claimed bound $A_p\leq(\eta+2(\exp(3||R||_{\infty}\hat{\beta}/\log\frac{1}{\epsilon})-1))^p$. Plugging this into \eqref{productexp} concludes the proof.
\end{proof}

\section{The normalization constant $\zeta^{(\epsilon)}_{t/\epsilon^2}$}\label{sec:app2}
In this appendix, we study the asymptotic behavior of $\zeta^{(\epsilon)}_{t/\epsilon^2}$.  This is given by the following lemma.
\begin{mylem}\label{lemmB.1}
    There is a constant $C_1>0$ such that 
    \begin{equation}\label{B.1}
        \zeta^{(\epsilon)}_{t/\epsilon^2}=(C_1+o(1))\frac{t}{\epsilon^2\log \frac{1}{\epsilon}}+o(1),
    \end{equation}
    as $\epsilon\rightarrow0$. The error terms depend only on $R$.
\end{mylem}

\begin{proof}
    First, we assume that $t/\epsilon^2=N_{\epsilon}\in\mathbb{N}$ and we set $\tau=1$. Repeating the construction of the Markov chain, as in \textbf{Section \ref{sec:3}}, while also keeping track of the normalization constants, yields
    $$\hat{\P}^{(\epsilon)}_{N_{\epsilon}}(d\omega)=\Psi^{(\epsilon)}_{}(x_0)\hat{\P}^{(\epsilon)}_{1}(dx_0)\prod_{k=0}^{N_{\epsilon}-2}\hat{\pi}^{(\epsilon)}(x_k,dx_{k+1})\Psi^{(\epsilon)}_{}(x_{N_{\epsilon}-1})^{-1}(\rho^{(\epsilon)})^{N_{\epsilon}-1}e^{N_{\epsilon}\zeta^{(\epsilon)}_1-\zeta^{(\epsilon)}_{N_{\epsilon}}}$$
    with $\Psi^{(\epsilon)}$ and $\rho^{(\epsilon)}$ defined by the eigenvalue problem \eqref{eigenv}. Recall the normalization $\int_{\Omega_1}\Psi^{(\epsilon)}(x)\hat{\P}^{(\epsilon)}(dx)=1$. Since $\hat{\P}^{(\epsilon)}_{N_{\epsilon}}$ is a probability measure
        $$\E[\Psi^{(\epsilon)}_{}(X_{N_{\epsilon}-1})^{-1}]=(\rho^{(\epsilon)})^{1-N_{\epsilon}}e^{-N_{\epsilon}\zeta^{(\epsilon)}_1+\zeta_{N_{\epsilon}}}.$$
    From the observations at the start of \textbf{Section \ref{sec:4}}, $\Psi^{(\epsilon)}_{}(x)\rightarrow1$ uniformly in $x$, as $\epsilon\rightarrow0$. Therefore,
    \begin{equation}\label{B.2}
       \zeta_{N_{\epsilon}}=N_{\epsilon}\zeta^{(\epsilon)}_1-(N_{\epsilon}-1)\log\rho^{(\epsilon)}+o(1).
    \end{equation}
    We estimate the terms $\zeta^{(\epsilon)}_1$ and $\log\rho^{(\epsilon)}$.
       For $\zeta^{(\epsilon)}_1$, recall its definition in \eqref{zeta}: 
       $$\zeta^{(\epsilon)}_1=\log\int_{\Omega_1}\exp\biggl(\frac{\hat{\beta}^2}{2\log\frac{1}{\epsilon}}\int_{[0,1]^2}R(s-u, B(s)-B(u))dsdu\biggr)W_1(dB).$$
       By a Taylor expansion, we see that as $\epsilon\rightarrow0$
    \begin{equation}\label{B.3}
        \zeta^{(\epsilon)}_1=\frac{C}{\log\frac{1}{\epsilon}}+o(\frac{1}{\log\frac{1}{\epsilon}}),
    \end{equation}
    where
    $$C=\frac{\hat{\beta
    }^2}{2}\E_{B}\biggr[\int_{[0,1]^2}R(s-u,B(s)-B(u))dsdu\biggl].$$
    Now for $\log\rho^{(\epsilon)}$, we look at  the eigenvalue equation \eqref{eigenv}
    $$\int_{\Omega_1}e^{I^{(\epsilon)}(x,y)}\Psi^{(\epsilon)}(y)\hat{\P}^{(\epsilon)}_1(dy)=\rho^{(\epsilon)}\Psi^{(\epsilon)}(x),$$
     with $I^{(\epsilon)}$ defined in \eqref{iotepdef}. From the observations at the start of \textbf{Section \ref{sec:4}}, $\hat{\P}^{(\epsilon)}_1$ converges to the Wiener measure in total variation, as $\epsilon\rightarrow0$. Moreover, $I^{(\epsilon)}(x,y)\rightarrow0$ as $\epsilon\rightarrow0$, uniformly in $x,y\in\Omega_1$. These two observations, combined with a Taylor expansion of the above exponential, show that
         $$\rho^{(\epsilon)}-1=\frac{\tilde{C}}{\log\frac{1}{\epsilon}}+o(\frac{1}{\log\frac{1}{\epsilon}}),$$
    as $\epsilon\rightarrow0$, where
    \begin{equation}
    \tilde{C}=\int_{\Omega_1}\int_{\Omega_1}I(x,y)W_1(dx)W_1(dy),
    \end{equation}
    with
    $$I(x,y):=\hat{\beta}^2\int_{[0,1]^2}R(s-u,y(s)+x(1)-x(u))dsdu, \textbf{ }x,y\in\Omega_{1}.$$
    This proves that
    \begin{equation}\label{B.4}
        \log\rho^{(\epsilon)}=\frac{\tilde{C}}{\log\frac{1}{\epsilon}}+o(\frac{1}{\log\frac{1}{\epsilon}}),
    \end{equation}
    Now, plugging $(\ref{B.4})$ and $(\ref{B.3})$ to $(\ref{B.2})$ yields $(\ref{B.1})$ with:
    $$C_1=\hat{\beta^2}\E_{B}\biggl[\int_{[0,1]^2}R(s-u,B(s)-B(u))dsdu\biggr]+\hat{\beta^2}\E_{B^1}\E_{B^2}\biggl[\int_{[0,1]^2}R(s-u,B^2(s)+B^1(1)-B^1(u))dsdu\biggr].$$
    When $t/\epsilon^2\in\mathbb{R}_{>0}$ we have the same asymptotic, since
    $$\zeta^{(\epsilon)}_{t/\epsilon^2}=\zeta^{(\epsilon)}_{[t/\epsilon^2]}+o(1)$$
    and $[t/\epsilon^2]/(t/\epsilon^2)\rightarrow1$ as $\epsilon\rightarrow0$.
\end{proof}

\begin{myackn}
The author thanks their supervisor, Nikos Zygouras, for suggesting the problem and for many helpful discussions and suggestions. The author was supported by the Warwick Mathematics Institute Centre for Doctoral Training and gratefully acknowledges funding from the University of Warwick.  
\end{myackn}

\printbibliography

@misc{duch2022flow,
      title={Flow equation approach to singular stochastic PDEs}, 
      author={Paweł Duch},
      year={2022},
      eprint={2109.11380},
      archivePrefix={arXiv},
      
}

@article{Gu_2018,
	doi = {10.1007/s00220-018-3202-0},
  

  
	year = 2018,
	month = {jul},
  
	publisher = {Springer Science and Business Media {LLC}
},
  
	volume = {363},
  
	number = {2},
  
	pages = {351--388},
  
	author = {Yu Gu and Lenya Ryzhik and Ofer Zeitouni},
  
	title = {The Edwards{\textendash}Wilkinson Limit of the Random Heat Equation in Dimensions Three and Higher},
  
	journal = {Communications in Mathematical Physics}
}

@article{dunlap2021random,
  title={The random heat equation in dimensions three and higher: the homogenization viewpoint},
  author={Dunlap, Alexander and Gu, Yu and Ryzhik, Lenya and Zeitouni, Ofer},
  journal={Archive for Rational Mechanics and Analysis},
  volume={242},
  number={2},
  pages={827--873},
  year={2021},
  publisher={Springer}
}

@article{10.1214/19-AOP1383,
author = {Francesco Caravenna and Rongfeng Sun and Nikos Zygouras},
title = {{The two-dimensional KPZ equation in the entire subcritical regime}},
volume = {48},
journal = {The Annals of Probability},
number = {3},
publisher = {Institute of Mathematical Statistics},
pages = {1086 -- 1127},
year = {2020},
doi = {10.1214/19-AOP1383}
}

@article{a2c66a1e-c941-34a5-824d-2b7be2b94dd1,
 ISSN = {10505164},
 author = {Francesco Caravenna and Rongfeng Sun and Nikos Zygouras},
 journal = {The Annals of Applied Probability},
 number = {5},
 pages = {3050--3112},
 publisher = {Institute of Mathematical Statistics},
 title = {Universality in marginally relevant disordered systems},
 urldate = {},
 volume = {27},
 year = {2017}
}

@article{c28fc3fc-a79f-3fe7-abb0-cce0ee5be155,
 ISSN = {0003486X},
 author = {Martin Hairer},
 journal = {Annals of Mathematics},
 number = {2},
 pages = {559--664},
 publisher = {Annals of Mathematics},
 title = {Solving the KPZ equation},
 urldate = {},
 volume = {178},
 year = {2013}
}

@article{Hairer_2014,
	doi = {10.1007/s00222-014-0505-4},
  
	
  
	year = 2014,
	month = {mar},
  
	publisher = {Springer Science and Business Media {LLC}
},
  
	volume = {198},
  
	number = {2},
  
	pages = {269--504},
  
	author = {M. Hairer},
  
	title = {A theory of regularity structures},
  
	journal = {Inventiones mathematicae}
}

@article{Gubinelli_2016,
	doi = {10.1007/s00220-016-2788-3},
  
	url = {https://doi.org/10.1007%2Fs00220-016-2788-3},
  
	year = 2016,
	month = {nov},
  
	publisher = {Springer Science and Business Media {LLC}
},
  
	volume = {349},
  
	number = {1},
  
	pages = {165--269},
  
	author = {Massimiliano Gubinelli and Nicolas Perkowski},
  
	title = {{KPZ} Reloaded},
  
	journal = {Communications in Mathematical Physics}
}

@article{article,
author = {Mukherjee, Chiranjib and Shamov, Alexander and Zeitouni, Ofer},
year = {2016},
month = {01},
pages = {},
title = {Weak and Strong disorder for the stochastic heat equation and the continuous directed polymer in $d\geq 3$},
volume = {21},
journal = {Electronic Communications in Probability},
doi = {10.1214/16-ECP18}
}

@book{Dalang2008AMO,
  title={A minicourse on stochastic partial differential equations},
  author={Dalang, Robert},
  year={2009},
  publisher={Springer}
}

@article{Kallianpur1954TheSO,
  title={The sequence of sums of independent random variables},
  author={Gopinath Kallianpur and Herbert E. Robbins},
  journal={Duke Mathematical Journal},
  year={1954},
  volume={21},
  pages={285-307}
}

@book{janson_1997, place={Cambridge}, series={Cambridge Tracts in Mathematics}, title={Gaussian Hilbert Spaces}, DOI={10.1017/CBO9780511526169}, publisher={Cambridge University Press}, author={Janson, Svante}, year={1997}, collection={Cambridge Tracts in Mathematics}}

@book{sneddon2012mathematical,
  title={Mathematical Analysis and Numerical Methods for Science and Technology: Volume 1 Physical Origins and Classical Methods},
  author={Sneddon, I.N. and Dautray, R. and Benilan, P. and Lions, J.L. and Cessenat, M. and Gervat, A. and Kavenoky, A. and Lanchon, H.},
  isbn={9783642615276},
  url={https://books.google.gr/books?id=mUzuCAAAQBAJ},
  year={2012},
  publisher={Springer Berlin Heidelberg}
}

@article{AFST_2017_6_26_4_847_0,
     author = {Chandra, Ajay and Weber, Hendrik},
     title = {Stochastic {PDEs,} {Regularity} structures, and interacting particle systems},
     journal = {Annales de la Facult\'e des sciences de Toulouse : Math\'ematiques},
     pages = {847--909},
     publisher = {Universit\'e Paul Sabatier, Toulouse},
     volume = {Ser. 6, 26},
     number = {4},
     year = {2017},
     doi = {10.5802/afst.1555}
}

@article{10.1214/21-AIHP1173,
author = {Dimitris Lygkonis and Nikos Zygouras},
title = {{Edwards-Wilkinson  fluctuations for the directed polymer in the full ${L^{2}}$-regime for dimensions $d\geq 3$}},
volume = {58},
journal = {Annales de l'Institut Henri Poincaré , Probabilités et statistiques},
number = {1},
publisher = {Institut Henri PoincarÃ©},
pages = {65 -- 104},
year = {2022},
doi = {10.1214/21-AIHP1173}
}

@article{NakajimaShuta2022Loln,
title = {Law of large numbers and fluctuations in the sub-critical and L2 regions for SHE and KPZ equation in dimension $d\geq 3$},
address = {},
author = {Cosco, Clément and Nakajima, Shuta and Nakashima, Makoto},
issn = {0304-4149},
journal = {Stochastic processes and their applications},
language = {eng},
pages = {127 - 173},
publisher = {Elsevier B.V},
volume = {151},
year = {2022},
}

@article{c0a84c4c-9371-3683-aa99-89bc719e4a36,
 ISSN = {10505164, 21688737},
 URL = {https://www.jstor.org/stable/26846324},
 author = {Yu Gu and Li-Cheng Tsai},
 journal = {The Annals of Applied Probability},
 number = {5},
 pages = {pp. 3037--3061},
 publisher = {Institute of Mathematical Statistics},
 title = {Another look into the Wong-Zakai Theorem for Stochastic Heat Equation},
 urldate = {},
 volume = {29},
 year = {2019}
}

@article{Wonkzakai,
author = {Hairer, Martin and Pardoux, Etienne},
year = {2014},
month = {09},
pages = {},
title = {A Wong-Zakai theorem for stochastic PDEs},
volume = {67},
journal = {Journal of the Mathematical Society of Japan},
doi = {10.2969/jmsj/06741551}
}

@article{Gu2019FluctuationsOA,
  title={Fluctuations of a nonlinear stochastic heat equation in dimensions three and higher},
  author={Gu, Yu and Li, Jiawei},
  journal={SIAM Journal on Mathematical Analysis},
  volume={52},
  number={6},
  pages={5422--5440},
  year={2020},
  publisher={SIAM}
}

@book{nualart2009malliavin,
  title={Malliavin calculus and its applications},
  author={Nualart, David},
  number={110},
  year={2009},
  publisher={American Mathematical Soc.}
}

@book{lawler2010random,
  title={Random Walk: A Modern Introduction},
  author={Lawler, G.F. and Limic, V.},
  isbn={9781139488761},
  series={Cambridge Studies in Advanced Mathematics},
  year={2010},
  publisher={Cambridge University Press}
}

@article{renormalization,
author = {Kupiainen, Antti},
year = {2014},
month = {10},
pages = {},
title = {Renormalization Group and Stochastic PDEs},
volume = {17},
journal = {Annales Henri Poincaré},
doi = {10.1007/s00023-015-0408-y}
}

@article{Kal-Rob,
 ISSN = {00278424},
 URL = {http://www.jstor.org/stable/88704},
 author = {G. Kallianpur and H. Robbins},
 journal = {Proceedings of the National Academy of Sciences of the United States of America},
 number = {6},
 pages = {525--533},
 publisher = {National Academy of Sciences},
 title = {Ergodic Property of the Brownian Motion Process},
 urldate = {},
 volume = {39},
 year = {1953}
}

@article{Port,
author = {Portenko, N. I.},
title = {Diffusion Processes with Unbounded Drift Coefficient},
journal = {Theory of Probability \& Its Applications},
volume = {20},
number = {1},
pages = {27-37},
year = {1975},
doi = {10.1137/1120003}
}

@article{SHF,
  title={The critical 2d Stochastic Heat Flow},
  author={Francesco Caravenna and Rongfeng Sun and Nikos Zygouras},
  journal={Inventiones mathematicae},
  year={2021},
  volume={233},
  pages={325-460},
  url={},
  doi={}
}

@article{Gu2018GaussianFF,
  title={Gaussian fluctuations from the 2D KPZ equation},
  author={Gu, Yu},
  journal={Stochastics and Partial Differential Equations: Analysis and Computations},
  volume={8},
  pages={150--185},
  year={2020},
  publisher={Springer}
}

@article{SHFmoments,
author = {Gu, Yu and Quastel, Jeremy and Tsai, Li-Cheng},
year = {2021},
month = {03},
pages = {179-219},
title = {Moments of the 2D SHE at criticality},
volume = {2},
journal = {Probability and Mathematical Physics},
doi = {10.2140/pmp.2021.2.179}
}

@article{Caravenna_2023,
   title={The critical 2d stochastic heat flow is not a Gaussian multiplicative chaos},
   volume={51},
   ISSN={0091-1798},
   url={},
   DOI={10.1214/23-aop1648},
   number={6},
   journal={The Annals of Probability},
   publisher={Institute of Mathematical Statistics},
   author={Caravenna, Francesco and Sun, Rongfeng and Zygouras, Nikos},
   year={2023},
   }

@article{Mukherjee2017CentralLT,
  title={Central limit theorem for Gibbs measures on path spaces including long range and singular interactions and homogenization of the stochastic heat equation},
  author={Chiranjib Mukherjee},
  journal={The Annals of Applied Probability},
  year={2017}
}

@article{Weak_Strong_Mukh_Sha_Zei,
author = {Chiranjib Mukherjee and Alexander Shamov and Ofer Zeitouni},
title = {{Weak and strong disorder for the stochastic heat equation and continuous directed polymers in $d\geq 3$}},
volume = {21},
journal = {Electronic Communications in Probability},
publisher = {Institute of Mathematical Statistics and Bernoulli Society},
pages = {1 -- 12},
year = {2016},
doi = {10.1214/16-ECP18}
}

@article{Rege_pro_Tal,
author = {Tal Orenshtein},
title = {{Rough invariance principle for delayed regenerative processes}},
volume = {26},
journal = {Electronic Communications in Probability},
publisher = {Institute of Mathematical Statistics and Bernoulli Society},
pages = {1 -- 13},
year = {2021},
doi = {10.1214/21-ECP406}
}

@article{cosco2023moments,
  title={Moments of partition functions of 2d gaussian polymers in the weak disorder regime-I},
  author={Cosco, Cl{\'e}ment and Zeitouni, Ofer},
  journal={Communications in Mathematical Physics},
  volume={403},
  number={1},
  pages={417--450},
  year={2023},
  publisher={Springer}
}

@book{sznitman2013brownian,
  title={Brownian motion, obstacles and random media},
  author={Sznitman, Alain-Sol},
  year={2013},
  publisher={Springer Science \& Business Media}
}
\end{document}